\documentclass[11pt]{amsart}
\usepackage{amsmath, amsfonts, amsthm, amssymb}
\usepackage[all,cmtip,line]{xy}
\usepackage{array}
\usepackage{enumerate}
\usepackage{srcltx} 

\newtheorem{thm}{Theorem}[section]
\newtheorem{lem}{Lemma}[section]
\newtheorem{prop}[lem]{Proposition}
\newtheorem{cor}[lem]{Corollary}
\newtheorem{q}[lem]{Question}
\newtheorem{defn}[lem]{Definition}
\newtheorem{rem}[lem]{Remark}

\newtheorem{claim}[lem]{Claim}
\numberwithin{equation}{section}

\newcommand \eps{\varepsilon}

\newlength{\originalbase}
\newcommand{\spacing}[1]{\setlength{\baselineskip}{#1\originalbase}}

\begin{document}               

\newcommand{\avint}{{- \hspace{-3.5mm} \int}}

\spacing{1}

\title{On the constant scalar curvature K\"ahler metrics(III)\\---general automorphism group}
\author{Xiuxiong Chen, Jingrui Cheng}
\maketitle
\begin{abstract}
In this paper, we derive estimates for scalar curvature type equations with more singular right hand side.
As an application, we prove Donaldson's conjecture on the equivalence between geodesic stability and existence of cscK when $Aut_0(M,J)\neq0$. Moreover, we also show that when $Aut_0(M,J)\neq0$, the properness of $K$-energy with respect to a suitably defined distance implies the existence of cscK.
\end{abstract}

\tableofcontents
\date{\today}
\section{Introduction}
This is the third and final paper in our series of papers studying extremal K\"ahler metrics in K\"ahler manifolds without boundary \cite{cc1}, \cite{cc2}. In this paper, we continue our study of the twisted cscK equation
\[
t(R_\varphi -\underline{R}) = (1-t) (tr_\varphi \chi- \underline{\chi}), \qquad {\rm where}\;\; t\in [0,1].
\]
  We studied a priori estimates for the preceding equation with $t=1$ in \cite{cc1}
and studied this one-parameter family of equations with $\chi$ being some fixed smooth real $(1,1)$ form in \cite{cc2}.  In this paper, $\chi $ is allowed to vary in a fixed K\"ahler class
with some constraints.
More specifically, we consider
\begin{equation}\label{rhs}
\chi = \chi_0 +\sqrt{-1}\partial \bar \partial f\geq0,\qquad  \sup_M f =0,\qquad \int_Me^{-pf}<\infty\textrm{ for some $p>1$.}
\end{equation}
We are able to extend many of our previous estimates in \cite{cc1}, \cite{cc2} to these  more general right hand side as (\ref{rhs}). These new apriori estimates are crucial for us to extend our proof
of  Donaldson's conjecture on geodesic stability and the Properness conjecture for $K$-energy to the setting with general automorphism group.  The results in this paper generalize our previous results in \cite{cc2} , where we assume the automorphism group is discrete.  For simplicity, we only state and prove the results on constant scalar curvature K\"ahler metrics in this paper. Analogous results for extremal K\"ahler metrics can be proved in a similar way without additional estimates. \\

Throughout this paper, we denote
$ G=Aut_0(M,J) $ as the identity component of automorphism group.
One of the main goals of this paper is to extend the Donaldson's conjecture on geodesic stability ( i.e. geodesic stability is equivalent to the existence of cscK metrics) to the general case. 
One difficulty is to characterize the {\it borderline case} where the geodesic ray is not strictly stable.  This is analogous to the case
of test configurations, for which we want to define $K$-stability, and how to characterize the test configuration whose Donaldson-Futaki invariant vanishes is a key issue.
The precise version of Donaldson's conjecture we prove is the following:
\begin{thm}\label{t1.4new}The following statements are equivalent.
\begin{enumerate}
\item There exists no constant scalar curvature K\"ahler metrics in $(M,[\omega_0])$;
\item Either the Calabi-Futaki invariant of $(M,[\omega_0])$ is nonzero, or there exists a K\"ahler potential $\varphi\in\mathcal{E}_0^1$ with $K(\varphi)<\infty$, and a locally finite energy geodesic ray in $\mathcal{E}_0^1$ initiating from $\varphi$ where the $K$-energy is non-increasing but it is not parallel to a holomorphic line;
\item Either the Calabi-Futaki invariant of $(M,[\omega_0])$ is nonzero, or for any K\"ahler potential $\varphi_0\in \mathcal{E}^1_0$, there exists a locally finite energy geodesic ray initiated from $\varphi_0$ where $K$-energy is non-increasing but it is not parallel to a holomorphic line.
\end{enumerate}
\end{thm}
In the above, {\it holomorphic line} means a continuous curve $h:[0,\infty)\rightarrow\mathcal{E}_0^1$, such that for any $t>0$, the $(1,1)$ current $\omega_{h(t)}:=\omega_0+\sqrt{-1}\partial\bar{\partial}h(t)=\sigma_t^*\omega_{h(0)}$ for a one-parameter family $\sigma_t\in G$ and ``parallelism" is defined as in Definition \ref{d1.4new}. 
$\mathcal{E}^1$ is the metric completion of $\mathcal{H}$ under $L^1$ geodesic distance, and $\mathcal{E}_0^1=\mathcal{E}^1\cap\{\phi:I(\phi)=0\}$.
We refer to section 2 of our second paper \cite{cc2} and references therein for more details.

We will prove this as a result of equivalence between existence of cscK and geodesic stability, a notion to be defined in definition 1.4  below.
Following \cite{cc2}, we introduce $\yen$ invariant associated with geodesic ray and the notion of ``parallelism" between two  locally finite energy geodesic rays. This invariant characterizes the growth of $K$-energy along a geodesic ray.
\begin{defn}
Let $\phi\in\mathcal{E}_0^1$ with $K(\phi)<\infty$. Let $\rho:[0,\infty)\rightarrow\mathcal{E}_0^1$ be a locally finite energy geodesic ray with unit speed, we define:
$$
\yen[\rho]=\lim\inf_{k\rightarrow\infty}\frac{K(\rho(k))}{k}.
$$
\end{defn}
\begin{rem}
From the convexity of $K$-energy along locally finite energy geodesic ray (c.f. \cite{Darvas1602}, Theorem 4.7), we see that actually the above limit exists, namely
$$\yen[\rho]=\lim_{k\rightarrow\infty}\frac{K(\rho(k))}{k}.
$$
Moreover if $K(\rho(t))<\infty$ for any $t\geq0$, 
$$
\yen[\rho]=\lim_{k\rightarrow\infty}\big(K(\rho(k+1))-K(\rho(k))\big).
$$
\end{rem}
\begin{defn}\label{d1.4new}
Let $\rho_i:[0,\infty)\rightarrow\mathcal{E}_0^1$ be two \sloppy continuous curves, $i=1,2$.
We say that $\rho_1$ and $\rho_2$ are parallel, if   
 \sloppy $\sup_{t>0}d_1(\rho_1(t),\rho_2(t))<\infty$.
\end{defn}
Obviously, one can modify this according to $d_p$ topology for any $p \geq 1.\;$ Indeed, the notion of parallelism was first introduced in \cite{cc2}.
We can define a notion of geodesic stability/semistability in terms of $\yen$ invariant as follows:
\begin{defn}\label{d5.3}
Let $\phi_0\in\mathcal{E}_0^1$ be such that $K(\phi_0)<\infty$.
We say $(M,[\omega_0])$ is geodesic stable at $\phi_0$ if for any locally finite energy geodesic ray $\rho:[0,\infty)\rightarrow\mathcal{E}_0^1$ with unit speed, exactly one of the following alternative holds:
\begin{enumerate}
\item $\yen[\rho] > 0$, 
\item $\yen[\rho]=0$, and $\rho$ is parallel to another geodesic ray $\rho':[0,\infty)\rightarrow \mathcal{E}_0^1$, generated from a holomorphic vector field $X \in aut(M, J)$.
\end{enumerate}
We say $(M,[\omega_0])$ is geodesic semistable at $\phi_0$ as long as $\yen[\rho]\geq0$ for all geodesic ray $\rho$ described above.

We say $(M,[\omega_0])$ is geodesic stable(resp. semistable) if it is geodesic stable(resp. semistable) at every $\phi\in\mathcal{E}_0^1$.
\end{defn}
This notion of geodesic stability is equivalent to existence of cscK:
\begin{thm} There exists a cscK metric if and only if $(M,[\omega_0])$ is geodesic stable.
\end{thm}
We remark that the notion of geodesic stability/semistability  is independent of the choice of base potential $\phi_0$, in virtue of  Theorem \ref{t1.2}. (see below).

After we prove this theorem, we obtain the following characterization of geodesic semistability.
\begin{thm}\label{t1.6}$(M,[\omega_0])$ is geodesic semistable if and only if the continuity path $t(R_{\varphi}-\underline{R})=(1-t)(tr_{\varphi}\omega_0-n)$ has a solution for any $t<1$.
\end{thm}
As a consequence of this theorem and Theorem 1.6 of our second paper \cite{cc2}, we deduce
\begin{cor}
If the $K$-energy is bounded from below in $(M,[\omega_0])$, then $(M,[\omega_0])$ is geodesic semistable.
\end{cor}
It is an interesting question to ask if the converse is also true. Namely if $(M,[\omega_0])$ is geodesic semistable, does it follow that $K$-energy is bounded from below?

Note that for the corresponding statement in the algebraic case,  we don't know how to conclude the existence of a lower bound of $K$-energy from $K$-stability  or uniform stability except in
the Fano manifold where the authors proved it indirectly in route of CDS's theorem.\\

We have the following theorem which is useful
to our characterization of borderline case.

\begin{thm}\label{t1.2}  Let $\rho_1(t):[0,\infty)\rightarrow\mathcal{E}_0^1$ be a locally finite energy geodesic ray with unit speed. Then for any $\varphi\in\mathcal{E}_0^1$, there exists at most one unit speed locally finite energy geodesic ray $\rho_2(t):[0,\infty)\rightarrow\mathcal{E}_0^1$ initiating from $\varphi$ which is parallel to $\rho_1$. Moreover, $\yen[\rho_1]=\yen[\rho_2]$ for any such geodesic ray $\rho_2$.

 If $\yen[\rho_1]<\infty$ and $K(\varphi)<\infty$, then there exists such a geodesic ray $\rho_2$ initiating from $\varphi$ and is parallel to $\rho_1$.
\end{thm}
\begin{rem}
It is an interesting question whether such a parallel geodesic ray exists in general, i.e. with no assumption on $\yen$ invariant.
\end{rem}
The uniqueness part and that $\yen$ invariant for the two rays is equal will be proved in the Appendix.
For existence part, we will first give a proof in the special case that $\rho_1(0),\varphi\in \mathcal{E}^2$ which allows us to use the Calabi-Chen theorem (c.f. \cite{calabi-chen}) that $(\mathcal{E}^2, d_2)$ 
is non-positively curved space. Note that when $p\neq 2, $ the infinite dimensional space $(\mathcal{E}^p, d_p)$ is no longer
Riemannian formally. Nonetheless, we prove the following theorem, which follows from the NPC (non-positively curved) property when $p=2$.
\begin{thm}\label{t5.1new}
Let $1\leq p<\infty$. Let $\phi_0$, $\phi_0'$, $\phi_1$, $\phi_1'\in\mathcal{E}^p$.
Denote $\{\phi_{0,t}\}_{t\in[0,1]}$, $\{\phi_{1,t}\}_{t\in[0,1]}$ be the finite energy geodesics connecting $\phi_0$ with $\phi_0'$ and $\phi_1$ with $\phi_1'$ respectively. Then we have
$$
d_p(\phi_{0,t},\phi_{1,t})\leq (1-t)d_p(\phi_0,\phi_1)+td_p(\phi_0',\phi_1').
$$
\end{thm}
\begin{rem}
In \cite{Darvas1602}, Proposition 5.1, the authors obtained Theorem \ref{t5.1new} for the case   $p=1$, using a representation formula of $d_1$.
\end{rem}
The fact that any two parallel geodesic rays have the same $\yen$ invariant motivates us to propose the following

\begin{defn}  We define the sphere at infinity $\mathcal{S}_\infty$ for $\mathcal H$ as the collection of all locally finite energy geodesic rays with unit speed in $\mathcal H$ modulo equivalent classes defined by parallelism.
\end{defn}

Given any base potential $\varphi_0$,  the sphere at infinity $\mathcal{S}_\infty$  is precisely the collection of all locally finite energy geodesic rays with unit speed in $\mathcal H$  (initiating from $\varphi_0$). It follows that $\mathcal{S}_\infty$ can be embedded as a subset of $T_{\varphi_0} \mathcal H. \;$ Given $p\geq1$, following T. Darvas \cite{Darvas1402}, one can complete $\mathcal H$ into a complete metric space $(\mathcal{E}^p, d_p)$ and to this we can adjoin the sphere of infinity $\tilde{\mathcal{S}}_\infty $ as in finite dimensional case.  Then, there is a natural map from $\tilde{\mathcal{S}}_\infty$ to $L^p(M, \omega)$ by sending every locally finite energy geodesic ray to its tangent vector at time 0 (i.e. initial velocity).  The following question
is interesting:

\begin{q} How do we characterize the set of $\psi\in L^p(M,\omega)$ ($p\geq1$) which is the initial velocity of a locally finite energy geodesic ray with unit speed? Or more importantly, is the set of velocities (which give rise to a locally finite geodesic
ray) a convex subset of $L^p(M, \omega)? $ 
\end{q}

In order to answer the above question, one first should be able to define rigorously the notion of ``initial velocity of a locally finite energy geodesic ray",
 which is an interesting question in its own. We remark that the set of initial velocities of a locally finite energy geodesic ray at time  $t=0$ is much more restrictive than the set of tangent vectors for a geodesic segment at initial point (regardless of the length). 
A typical example is on toric varieties where the first set consists of all convex functions
in a polytope and the second set consists of all functions in a polytope. \\

Recall that in the algebraic setting, $(L, h)$ is a line bundle over $M$ such that its curvature form gives rise to the background K\"ahler form $\omega.\;$ For $k$ large enough, $(M, L^k)$  admits ample
holomorphic sections which we can use to construct holomorphic embeddings from $(M, [\omega])$ to $\mathbb{C}\mathbb{P}^N$ for some $N = N(k).\;$ Let $B_N$ be the collection of all induced metrics
 from  holomorphic embeddings of $(M, [\omega])$ into $\mathbb{C}\mathbb{P}^N.\;$  According to Tian's thesis, the space of K\"ahler potentials
 can be approximated by a sequence of  finite dimensional Bergman spaces $B_N; $  Note that these Bergman spaces are
precisely symmetric spaces $SL(N+1, \mathbb{C})/ U(N+1, \mathbb{C})\;$ and the test configuration defines a ``point" at infinity of this finite dimensional Bergman space $B_N$. From complex geometric point of view, the sphere
of infinity of $B_N$ is precisely the closure or ``completion" of the space of test configurations associated with line bundle $(M, L^k)\;$ under appropriate topology. 
 It is natural to conjecture that the sphere at infinity  $\mathcal{S}_\infty$ is precisely the limit of these
 ``spheres at infinity" of these finite dimensional Bergman spaces.  In some sense this is not surprising if one notes that in an earlier paper by Chen-Sun \cite{CS2012}, the authors proved
that the $L^2$ distance in $\mathcal H$ can be approximated by the $L^2$ distance in the finite dimensional Bergman spaces (up to scaling).  Following
this trend of thoughts, one can naturally ask the following question:

\begin{q}  Is the set of points defined by test configurations dense in the sphere of infinity ${\mathcal S}_\infty.\;$  More importantly, does the uniform stability or filtrated stability in the sense of G. Sz$\acute{\text{e}}$kelyhidi\cite{Sz06} imply geodesic stability?
\end{q}

In the intriguing paper \cite{Ross11}, the authors defined a notion of analytic test configurations which is an one-parameter family of K\"ahler potentials, concave with respect to this parameter. They proved that each analytic test configuration defines a unique geodesic ray (up to parallelism), thus defines a point at ${\mathcal S}_\infty.$  It is important to
understand the converse: when does a locally finite energy geodesic ray define a test configuration?
 More importantly, can one extend the definition of $\yen$ invariant for analytic test configuration? Note that by Theorem 1.4, we can view $\yen$ invariant as a function defined on the sphere of infinity ${\mathcal S}_\infty.$

One important technical ingredient to prove existence of cscK from geodesic stability is the following properness theorem, which says that if $K$-energy is ``proper" with respect to a suitably defined distance (defined precisely in Definition \ref{d1.1} below), then cscK exists.
 To extend the properness theorem to the case of general automorphism group, we need
to extend the notion of properness (c.f. \cite{Darvas1605} \cite{Dar-Rub-17}).

\begin{defn}\label{d1.1}
We say $K$-energy is proper with respect to  $L^1$ geodesic distance modulo $G$, if 
\begin{enumerate}
\item For any sequence $\{\varphi_i\}\subset\mathcal{H}_0$, $ \displaystyle \inf_{\sigma\in G}d_1(\omega_0,\sigma^*\omega_{\varphi_i})\rightarrow\infty$ implies $K(\varphi_i)\rightarrow+\infty$,
\item $K$-energy is bounded from below.
\end{enumerate}
\end{defn}
 With this in mind, we will prove that

\begin{thm}\label{t1.1}(Theorem \ref{t3.1})  There exists a constant scalar curvature K\"ahler metric  if and only if  the K energy functional is proper with respect to the $L^1$ distance modulo $G$.
\end{thm}

The direction that existence of cscK implies properness has been established by  Berman, Darvas and Lu in \cite{Darvas1605}.  For the converse direction, Darvas and Rubinstein in \cite{Dar-Rub-17} have reduced this problem to a problem of regularity of weak minimizers of $K$-energy over the space $\mathcal{E}^1$, which we resolved in the second paper \cite{cc2}. (In the special case of toric varieties, Zhou-Zhu\cite{zz08} proved the existence
of toric invariant weak minimizers of the modified $K$-energy  under properness assumption.)
Hence Theorem \ref{t1.1} has been established by combining the known results.
Nonetheless, in this paper we will show how to obtain Theorem \ref{t1.1} by solving along the continuity path already considered in the second paper \cite{cc2}. 
For this purpose, we  develop new estimates for scalar curvature type equations which may be of independent interest. The following is the main result from section 2.
\begin{thm}(Theorem \ref{t2.3})
Let $\varphi$ be a smooth solution to (\ref{eq-?1}), (\ref{eq-?2}), with assumptions in (\ref{rhs}) hold.  Suppose additionally that $p\geq\kappa_n$ for some constant $\kappa_n$ depending only on $n$. Then for any $p'<p$,
$$
||F+f||_{W^{1,2p'}}\leq C_{25.1},\,\,\,||n+\Delta\varphi||_{L^{p'}(\omega_0^n)}\leq C_{25.1}.
$$
Here $C_{25.1}$ depends only on an upper bound of entropy $\int_M\log\big(\frac{\omega_{\varphi}^n}{\omega_0^n}\big)\omega_{\varphi}^n$, $p$, $p'$, the bound for $\int_Me^{-pf}dvol_g$, $||R||_0$, $\max_M|\beta_0|_g$ and background metric $\omega_0$.
\end{thm}

In a series of three fundamental papers \cite{Dona02new}, \cite{Dona08}, \cite{Dona09}, S. Donaldson proved that in toric K\"ahler surfaces, the existence of cscK metric is indeed equivalent to the $K$-stability. This is partially generalized in \cite{CLS10} to extremal K\"ahler metrics (c.f. \cite{CHLS14} and references therein for interior regularity estimates on K\"ahler toric varieties). However, in general algebraic  K\"ahler manifolds, one expects that the $K$-stability might fall short of
the existence of cscK metrics. There are a lot of works in literatures trying to formulate the right algebraic conditions and one particular important notion is due to G. Sz$\acute{\text{e}}$kelyhidi\cite{Sz06}
where he defined some notion of uniform stability through filtrations of test configurations. 
In a surprising paper \cite{BBJ15}, the authors proved that the uniform stability implies properness
of Ding functional in terms  of Aubin functional on Fano K\"ahler manifolds. Moreover, using the properness of Ding functional, they are able to prove subsequently that
uniform stability  implies the existence of KE metrics. To understand these different notions of stability, perhaps it will be instructive for us to take a brief detour into toric K\"ahler manifolds. Following Donaldson \cite{Dona02new}, \cite{Dona08}, \cite{Dona09}, 
a polarized K\"ahler manifold $(M,[\omega], L) $ corresponds to a certain closed polytope $P \subset \mathbb{R}^n$. In the following, $d\,\sigma$ denotes the standard surface measure on the boundary and $d\mu$ is the $n$-dimensional Lebesgue measure. Then, we can define a linear functional on the space of functions on $P.\;$

\[
{\mathcal L}_P(f) = \displaystyle \int_{\partial P} \; f \; d\, \sigma - A  \int_P \; f\;d\,\mu, \qquad {\rm where}\; A = {{\displaystyle \int_{\partial P}  \; d\, \sigma}\over {\int_P \;d\,\mu}}.
\]
A function is normalized if it is perpendicular to constant and linear functions (under the $L^2$ inner product defined by $d\mu$) in the polytope $P.\;$ For the convenience of readers, we list various notions of
stability as follows.
\begin{enumerate}
\item {\bf K stable}:  For all non zero piecewise linear convex function $f$ in $P$, we have ${\mathcal L}_P(f) > 0.\;$
\item {\bf Filtrated stable} (in the sense of G. Sz$\acute{\text{e}}$kelyhidi\cite{Sz06}): For any convex, continuous  function $f$, we have ${\mathcal L}_P(f) > 0.\;$
\item {\bf Uniform stability}: There is an $\epsilon > 0$ such that for all normalized, non zero piecewise linear convex function $f$ in $P$, we have ${\mathcal L}_P(f) > \epsilon \displaystyle \int_P\; |f| d\,\mu.\;$
\item {\bf $L^1$ stability}: For all convex function $f$ whose boundary value lie in $L^1(\partial P, d\,\sigma), \;$ we have ${\mathcal L}_P(f)  > 0.\;$
\end{enumerate}

It appears that the notion of uniform stability is strongest among the four. However, if one examines carefully the original papers of Donaldson \cite{Dona02new}, \cite{Dona08}, \cite{Dona09} on this subject, it is already known through work there that $L^1$ stability implies both filtrated stability and uniform stability, from which properness of $K$-energy follows.  In other words, $L^1$ stability implies that $K$-energy is proper in terms of $L^1$ distance among all toric invariant potentials in the setting of toric varieties (c.f. Section 5 of Donaldson \cite{Dona02new}).\footnote{ We notice a recent work of T. Hisamoto \cite{HIS16} who proved that the uniform stability implies the properness of $K$-energy in toric K\"ahler manifolds which leads to a different path from
uniform stability to properness in the toric setting. } Hence one can use  Theorem \ref{t1.1} to deduce the existence of cscK.  Alternatively, Donaldson \cite{Dona02new} already proved the existence of weak minimizers of $K$-energy among toric invariant metrics (c.f. Theorem 4.7 of CLS \cite{CLS11}) in symplectic coordinates. Consequently, one should be able to conclude that there exists a smooth toric invariant cscK metric, as long as one can show the regularity of mimimizers in this setting, similar to what we obtained in Theorem 1.5 of \cite{cc2}. Therefore
we have
 \begin{thm}\label{t1.8new}  On toric K\"ahler manifold, the existence of cscK metric is equivalent to the $L^1$ stability. In other words, there exists a constant scalar curvature K\"ahler metric
if and only if ${\mathcal L}_P(f)  > 0\;$ holds for all convex function $f$
 such that $f\mid_{\partial P} \subset L^1(\partial P, d\,\sigma).\;$ \end{thm}

It is not immediately clear if filtrated stability alone will be sufficient to imply the existence of cscK metrics. More broadly, there are some intriguing questions
which might link a stability notion in the  algebraic sense to some analytical or geometric stability.  With Theorem \ref{t1.8new} in mind, one wonders if we can replace the $L^1$ stability condition
 by some algebraic conditions which can be checked relatively easily.\\

Continuing from Question 1.10 above and discussions in our second paper \cite{cc2}  in this series, the first named author believes the following question is very interesting.
\begin{q} In a polarized algebraic manifold $(M, L, [\omega])$,  the  following notions of stability are listed in the seemingly increasing order of strength:
\begin{enumerate}
\item  Filtrated Stable;
\item Uniform Stable;
\item Geodesic stable in $\mathcal{E}_0^\infty$; 
\item Geodesic stable in $\mathcal{E}_0^p (1<p < \infty);$ 
\item  Geodesic stable in $\mathcal{E}_0^1. $ 
\end{enumerate}
Are these notions of stability actually equivalent?
\end{q}

Here uniform stability and filtrated stability are algebraic notions (G. Sz$\acute{\text{e}}$kelyhidi \cite{Sz06}). Among all the space $\mathcal{E}_0^p (1\leq p \leq \infty) $ we studied
in this and preceding paper \cite{cc2},  perhaps both the two extremal cases $\mathcal{E}_0^1$ and  $\mathcal{E}_0^ \infty$ are special and important.
The notion of stability at $\mathcal{E}_0^ \infty$ might be the one which bridge over from algebraic notions of stability to more
analytic notions of stability. There is
a large and rapidly increasing volume of research centered around the concept of filtrated stability and uniform stability where we refer to Hisamoto  \cite{HIS16}  J. Ross
\cite{Ross05}, G. Sz$\acute{\text{e}}$kelyhidi \cite{Sz06},  Berman-Boucksom-Jonsson \cite{BBJ15}, Boucksom-Hisamoto-Jonsson \cite{BHJ16},  R. Dervan \cite{Dervan171}and references therein for more details.\\

Finally we explain the organization of the paper:

In section 2, we derive estimates for scalar curvature type equations with more general right hand side.

In section 3, we apply the estimates obtained in section 2 to prove the properness conjecture when the automorphism group is non-discrete.

In section 4, we prove the equivalence between geodesic stability and existence of cscK metrics.

In the appendix, we prove some results about the non-positively curved properties of the metric space $(\mathcal{E}^p,d_p)$ which will be useful to us. Such results may be of independent interest.\\

\noindent {\bf Acknowledgement} Both authors wish to thank Chen Gao for his meticulously careful readings of the earlier draft of this set of papers
and numerous valuable comments, corrections.  The second named author wishes to thank his advisor Mikhail Feldman for his interest in this work and warm encouragement. Both authors are also grateful for Sir Simon Donaldson, Weiyong He, Sun Song and Chengjian Yao for their interest in this work
and  their insightful comments and suggestions. \\

In this set of three papers, the first named author has been partially supported by NSF grant DMS-1515795.

\section{Scalar curvature type equations with singular right hand side}
Let $(M,J,\omega_0)$ be a compact K\"ahler manifold.
We consider the following scalar curvature type equations:
\begin{align}
\label{eq-?1}
&\det(g_{i\bar{j}}+\varphi_{i\bar{j}})=e^F\det g_{i\bar{j}},\\
\label{eq-?2}
&\Delta_{\varphi}F=tr_{\varphi}(Ric-\beta)-R.
\end{align}
In the above, $\beta=\beta_0+\sqrt{-1}\partial\bar{\partial}f\geq0$. Also we assume that $\beta_0$ is a bounded $(1,1)$ form and $f$ is normalized to be $\sup_Mf=0$, $e^{-f}\in L^{p_0}(M)$ for some $p_0>1$.
$R$ is a bounded function.

As before, $\varphi$ should be such that $g_{i\bar{j}}+\varphi_{i\bar{j}}>0$ on $M$, so that $\omega_{\varphi}:=\omega_0+\sqrt{-1}\partial\bar{\partial}\varphi$ defines a new K\"ahler metric in the same class as $\omega_0$.
We note that (\ref{eq-1}), (\ref{eq-2}) can be combined to give the following scalar curvature type equations:
\begin{equation}
R_{\varphi}=tr_{\varphi}\beta+R.
\end{equation}
Here $R_{\varphi}$ denotes the scalar curvature of the metric $\omega_{\varphi}$.
In the following, we will always assume that the solution $\varphi$ is smooth and our goal is to derive apriori estimates.
\subsection{Boundedness of $F+f$}
The estimate in this subsection only requires a bound for $\int_Me^{-p_0f}dvol_g$ for some $p_0>1$. In particular, we don't need any positivity assumption on the form $\beta$.
\begin{lem}\label{l2.1}
Let $\psi$ be the solution to the following equation:
\begin{align}
\label{MA-1}
&\det(g_{i\bar{j}}+\psi_{i\bar{j}})=\frac{e^F\sqrt{F^2+1}}{\int_Me^F\sqrt{F^2+1}dvol_g}\det g_{i\bar{j}},\\
\label{MA-2}
&\sup_M\psi=0.
\end{align}
Suppose also $\sup_M\varphi=0$. Then for any $0<\eps_0<1$, there exists a constant $C_0$, such that
$$
F+f+\eps_0\psi-2(\max_M|Ric-\beta_0|_g+1)\varphi\leq C_0.
$$
Here $C_0$ depends only on $\eps_0$, the upper bound of the entropy $\int_MFe^Fdvol_g$, the bound for $\max_M|\beta_0|_g$, $||R||_0$ and the background metric $(M,\omega_0)$.
\end{lem}

\begin{proof}
Similar to the cscK case, the proof is by Alexandrov maximum principle. Observe from (\ref{eq-?2}) that
\begin{equation}
\Delta_{\varphi}(F+f)=tr_{\varphi}(Ric-\beta_0)-R.
\end{equation}
Denote $C=2(\max_M|Ric-\beta_0|_g+1)$ and we compute
\begin{equation}\label{1.7}
\Delta_{\varphi}(F+f+\eps_0\psi-C\varphi)=tr_{\varphi}(Ric-\beta_0)-R+\eps_0\Delta_{\varphi}\psi-Cn+Ctr_{\varphi}g.
\end{equation}
Using arithmetic-geometric mean inequality, we have
$$
\Delta_{\varphi}\psi=g_{\varphi}^{i\bar{j}}(g_{i\bar{j}}+\psi_{i\bar{j}})-tr_{\varphi}g\geq A^{-\frac{1}{n}}(F^2+1)^{\frac{1}{2n}}-tr_{\varphi}g.
$$
Here $A=\int_Me^F\sqrt{F^2+1}dvol_g$. 
Also due to our choice of the constant $C$, we obtain from (\ref{1.7}) that
\begin{equation}
\Delta_{\varphi}(F+f+\eps_0\psi-C\varphi)\geq \frac{C}{2}tr_{\varphi}g+\eps_0A^{-\frac{1}{n}}(F^2+1)^{\frac{1}{2n}}-C_1.
\end{equation}
Here $C_1$ has the said dependence as stated in the lemma.
By Proposition 2.1 in \cite{tian87}, there exists $\alpha>0$, and a constant $C_2$, such that for any $\omega_0$-psh function $\phi$, we have 
$$
\int_Me^{-\alpha(\phi-\sup_M\phi)}dvol_g\leq C_2.
$$
Now denote $u=F+f+\eps_0\psi-C\varphi$, $\delta=\frac{\alpha}{2nC}$, and let $0<\theta<1$ to be determined. First for any $p\in M$, we can construct a cut-off function $\eta_p$, 
so that $\eta_p(p)=1$, $\eta_{p}\equiv 1-\theta$ outside the ball $B_{d_0}(p)$, 
and $|\nabla\eta_p|\leq\frac{2\theta}{d_0}$, $|\nabla^2\eta|\leq \frac{2\theta}{d_0^2}$.
Here $d_0$ is a sufficiently small constant depending only on the background metric $(M,\omega_0)$.
Assume that the function $u$ achieves maximum at the point $p_0$, then we compute
\begin{equation}\label{1.9}
\begin{split}
&\Delta_{\varphi}(e^{\delta u}\eta_{p_0})=e^{\delta u}\delta^2|\nabla_{\varphi}u|_{\varphi}^2\eta_{p_0}+e^{\delta u}\eta_{p_0}\delta\Delta_{\varphi}u+e^{\delta u}\Delta_{\varphi}\eta_{p_0}+2e^{\delta u}\delta\nabla_{\varphi}u\cdot_{\varphi}\nabla_{\varphi}\eta_{p_0}\\
&\geq e^{\delta u}\delta^2|\nabla_{\varphi}u|_{\varphi}^2\eta_{p_0}+e^{\delta u}\delta\eta_{p_0}(\frac{C}{2}tr_{\varphi}g+\eps_0A^{-\frac{1}{n}}(F^2+1)^{\frac{1}{2n}}-C_1)\\
&-e^{\delta u}|\nabla^2\eta_{p_0}|tr_{\varphi}g-e^{\delta u}\delta^2|\nabla_{\varphi}u|_{\varphi}^2\eta_{p_0}-e^{\delta u}\frac{|\nabla_{\varphi}\eta_{p_0}|_{\varphi}^2}{\eta_{p_0}}\\
&\geq e^{\delta u}\eta_{p_0}\big(\frac{\delta C}{2}-\frac{2\theta}{d_0^2(1-\theta)}-\frac{4\theta^2}{d_0^2(1-\theta)^2}\big)tr_{\varphi}g+e^{\delta u}\delta \eta_{p_0}\big(\eps_0A^{-\frac{1}{n}}(F^2+1)^{\frac{1}{2n}}-C_1\big)
\end{split}
\end{equation}
Choose $\theta$ small enough so that (note that $\delta C=\frac{\alpha}{2n}$)
$$
\frac{\delta C}{2}-\frac{2\theta}{d_0^2(1-\theta)}-\frac{4\theta^2}{d_0^2(1-\theta)}>0.
$$
With this choice of $\theta$, (\ref{1.9}) gives
\begin{equation}\label{2.10new}
\Delta_{\varphi}(e^{\delta u}\eta_{p_0})\geq e^{\delta u}\delta \eta_{p_0}(\eps_0 A^{-\frac{1}{n}}(F^2+1)^{\frac{1}{2n}}-C_1)\geq -e^{\delta u}\delta \eta_{p_0}C_1\chi_{\{F\leq C_3\}}.
\end{equation}
Here $\chi_{\{F\leq C_3\}}$ is the indicator function of the set $\{F\leq C_3\}$, and $C_3$ is a constant determined by the inequality
$$
\eps_0 A^{-\frac{1}{n}}\big(F^2+1\big)^{\frac{1}{2n}}-C_1\leq 0\textrm{ implies $F\leq C_3$.}
$$ 
Hence $C_3$ depends only on $\eps_0$, $C_1$ and $A$. 
We wish to apply Alexandrov maximum principle to (\ref{2.10new}) inside $B_{d_0}(p_0)$, and with a similar derivation as (5.16) in the first paper \cite{cc1}, we obtain:
\begin{equation}\label{2.11}
\begin{split}
e^{\delta u}&\eta_{p_0}(p_0)\leq \sup_{\partial B_{d_0}(p_0)}e^{\delta u}\eta_{p_0}\\
&+C_nd_0\bigg(\int_{B_{d_0}(p_0)}\delta^{2n}e^{2F}e^{2n\delta u}\big((\eps_0A^{-\frac{1}{n}}(F^2+1)^{\frac{1}{2n}}-C_1)^-\big)^{2n}dvol_g\bigg)^{\frac{1}{2n}}.
\end{split}
\end{equation}
To estimate the integral appearing above, observe that $f\leq 0$, $\psi\leq0$, then we have
\begin{equation}
\begin{split}
\int_{B_{d_0}(p_0)}&e^{2F}e^{2n\delta u}\big((\eps_0A^{-\frac{1}{n}}(F^2+1)^{\frac{1}{2n}}-C_1)^-\big)^{2n}dvol_g\\
&\leq \int_Me^{2F+2n\delta F}e^{-2nC\delta\varphi}\chi_{\{F\leq C_3\}}C_1^{2n}dvol_g\leq e^{(2+2n\delta)C_3}C_2(C_1)^{2n}.
\end{split}
\end{equation}
Since $\eta_{p_0}\leq 1-\theta$ on $\partial B_{d_0}(p_0)$, the result follows from (\ref{2.11}).
Indeed, since $e^{\delta u}$ achieves maximum at $p_0$, we have
\begin{equation*}
e^{\delta u}(p_0)\leq (1-\theta)e^{\delta u}(p_0)+C_nd_0\delta e^{(2+2n\delta)C_3}C_2(C_1)^{2n}.
\end{equation*}
The desired estimate then follows.
\end{proof}
\begin{cor}\label{c2.2}
There exists a constant $C_4$, such that
$$
F+f\leq C_4.
$$
In particular, if $\varphi$ is normalized so that $\sup_M\varphi=0$, then 
$$
||\varphi||_0\leq C_{4.5}.
$$
Here $C_4$ and $C_{4.5}$ depends only on the upper bound for the entropy $\int_Me^FFdvol_g$, the bound for $\max_M|\beta_0|_g$, $||R||_0$, $p_0$(uniform for $p_0>1$ as long as $p_0-1$ bounded away from 0), the bound $\int_Me^{-p_0f}dvol_g$ and the background metric $(M,\omega_0)$.
\end{cor}
\begin{proof}
First we obtain from Lemma \ref{l2.1} that
\begin{equation}
\frac{\alpha}{\eps_0}\big(F+f-2(\max_M|Ric-\beta_0|_g+1)\varphi\big)\leq-\alpha\psi+\frac{\alpha C_0}{\eps_0}.
\end{equation}
Hence for any $p>1$, if we choose $\eps_0$ so that $p=\frac{\alpha}{\eps_0}$, then we obtain
\begin{equation}
\int_Me^{p(F+f)}dvol_g\leq C_5.
\end{equation}
The constant $C_5$ has the dependence as described in Lemma \ref{l2.1} with additional dependence on $p$, but will be uniform in $p$ as long as $p$ remains bounded.
Choose $\eps_1=\frac{p_0-1}{2}$, then we can estimate
\begin{equation}\label{2.15}
\begin{split}
\int_M&e^{(1+\eps_1)F}dvol_g=\int_Me^{(1+\eps_1)(F+f)}\cdot e^{-(1+\eps_1)f}dvol_g\\
&\leq \bigg(\int_Me^{-p_0f}dvol_g\bigg)^{\frac{1+\eps_1}{p_0}}\cdot\bigg(\int_Me^{\frac{p_0}{p_0-(1+\eps_1)}(F+f)}dvol_g\bigg)^{1-\frac{1+\eps_1}{p_0}}\leq C_{5.5}.
\end{split}
\end{equation}
Here $C_{5.5}$ is uniform in $p_0$ as long as $p_0-1$ is bounded away from 0. 
Then we can conclude from (\ref{2.15}) and Kolodziej's main result (c.f. \cite{Kolo98}) that 
\begin{equation}
||\varphi||_0,\,\,||\psi||_0\leq C_6.
\end{equation}
The result now follows from Lemma \ref{l2.1}, with choice of $\eps_0$ so that $\frac{\alpha}{\eps_0}=\frac{p_0}{p_0-(1+\eps_1)}=\frac{2p_0}{p_0-1}$.
\end{proof}
Next we would like to estimate the lower bound for $F+f$.
\begin{lem}\label{l2.3}
There exists a constant $C_7$ such that
$$
F+f\geq -C_7.
$$
Here $C_7$ depends only on $||\varphi||_0$, $\max_M|\beta_0|_g$, $||R||_0$, the background metric $g$, the bound for $\int_Me^{-p_0f}dvol_g$, and $p_0$(uniform in $p_0$ as long as $p_0-1$ bounded away from 0).
In particular
$$
F\geq -C_7.
$$
\end{lem}
\begin{proof}
We choose $C=2(\max_M|Ric-\beta_0|_g+1)$.
Then we have
\begin{equation}
\Delta_{\varphi}(F+f+C\varphi)=tr_{\varphi}(Ric-\beta_0)-R+Cn-Ctr_{\varphi}g\leq -tr_{\varphi}g+||R||_0+Cn.
\end{equation}
Choose $\eps_2=\frac{p_0}{2n(p_0-1)}$, and the cut-off function $\eta_p$ as in the proof of Lemma \ref{l2.1}(with a parameter $\theta$ to be chosen later), and denote $u_1=F+f+C\varphi$. Assume the function $u_1$ achieves minimum at $p_1\in M$. We may compute
\begin{equation}
\begin{split}
&\Delta_{\varphi}(e^{-\eps_2 u_1}\eta_{p_1})=e^{-\eps_2u}(\eps_2^2|\nabla_{\varphi}u_1|_{\varphi}^2\eta_{p_1}\\
&\quad\quad\quad\quad\quad\quad\quad-\eps_2\Delta_{\varphi}u_1\eta_{p_1}+\Delta_{\varphi}\eta_{p_1}-2\eps_2\nabla_{\varphi}u_1\cdot_{\varphi}\nabla_{\varphi}\eta_{p_1})\\
&\geq e^{-\eps_2u}\big(\eps_2^2|\nabla_{\varphi}u_1|_{\varphi}^2\eta_{p_1}+\eps_2tr_{\varphi}g\eta_{p_1}-\eps_2(||R||_0+Cn)-|\nabla^2\eta_{p_1}|_gtr_{\varphi}g\\
&-\eps_2^2|\nabla_{\varphi}u_1|_{\varphi}^2\eta_{p_1}-\frac{|\nabla_{\varphi}\eta_{p_1}|_{\varphi}^2}{\eta_{p_1}}\big)\geq e^{-\eps_2u}\big(tr_{\varphi}g\eta_{p_1}(\eps_2-\frac{2\theta}{d_0^2(1-\theta)}-\frac{4\theta^2}{d_0^2(1-\theta)})\\
&-\eps_2(||R||_0+Cn)\big).
\end{split}
\end{equation}
Since $\eta_{p_1}\geq 1-\theta$, we may choose $\theta$ sufficiently small so that
$$
(1-\theta)\eps_2-\frac{2\theta}{d_0^2(1-\theta)}-\frac{4\theta^2}{d_0^2(1-\theta)}>0.
$$
With this choice, we then have 
\begin{equation}
\Delta_{\varphi}\big(e^{-\eps_2u_1}\eta_{p_1}\big)\geq -\eps_2e^{-\eps_2u}(||R||_0+Cn).
\end{equation}
Hence if we apply the Alexandrov maximum principle in $B_{d_0}(p_0)$, we have
\begin{equation}
\begin{split}
e^{-\eps_2u_1}\eta_{p_1}(p_1)&\leq\sup_{\partial B_{d_0}(p_1)}e^{-\eps_2u_1}\eta_{p_1}\\
&+C_nd_0\bigg(\int_Me^{2F}e^{-2n\eps_2u_1}\eps_2^{2n}(||R||_0+Cn)^{2n}dvol_g\bigg)^{\frac{1}{2n}}.
\end{split}
\end{equation}
To estimate the integral appearing above, we may calculate:
\begin{equation}
\begin{split}
\int_M&e^{2F}e^{-2n\eps_2u_1}\eps_2^{2n}(||R||_0+Cn)^{2n}dvol_g\leq C_8\int_Me^{(2-2n\eps_2)F-2n\eps_2f}dvol_g\\
&=C_8\int_Me^{\frac{p_0-2}{p_0-1}F}\cdot e^{-\frac{p_0}{p_0-1}f}dvol_g\leq C_8\bigg(\int_Me^Fdvol_g\bigg)^{\frac{p_0-2}{p_0-1}}\cdot\bigg(\int_Me^{-p_0f}dvol_g\bigg)^{\frac{1}{p_0-1}}.
\end{split}
\end{equation}
Since we have $\eta_{p_1}=1-\theta$ on $\partial B_{d_0}(p_1)$, the desired estimate then follows in the same way as in the last part of the proof for Lemma \ref{l2.1}.
\end{proof}
\subsection{$W^{2,p}$ estimate}
In this subsection, we will need to assume $\beta\geq0$ (or more generally a lower bound for $\beta$), besides assuming a bound for $\int_Me^{-p_0f}dvol_g$ for some $p_0>1$. 
\begin{thm}\label{t2.1}
Assume $\beta\geq0$ in (\ref{eq-?1}), (\ref{eq-?2}). For any $p\geq1$, there exists a constant $C_p$, depending only on $||F+f||_0$, $||R||_0$, $\max_M|\beta_0|_g$, the background metric $(M,\omega_0)$, a bound for $\int_Me^{-p_0f}dvol_g$, $||\varphi||_0$ and $p$, such that
$$
\int_Me^{(p-1)f}(n+\Delta\varphi)^pdvol_g\leq C_p.
$$
\end{thm}
\begin{proof}
Let $\kappa>0$, $C>0$, $0<\delta<1$ be constants to be chosen later, we will compute:
\begin{equation}\label{2.22}
\begin{split}
\Delta_{\varphi}&\big(e^{-\kappa(F+\delta f+C\varphi)}(n+\Delta\varphi)\big)=\Delta_{\varphi}\big(e^{-\kappa(F+\delta f+C\varphi)}\big)(n+\Delta\varphi)+e^{-\kappa(F+\delta f+C\varphi)}\Delta_{\varphi}(\Delta\varphi)\\
&-2\kappa e^{-\kappa(F+\delta f+C\varphi)}\nabla_{\varphi}(F+\delta f+C\varphi)\cdot_{\varphi}\nabla_{\varphi}(\Delta\varphi).
\end{split}
\end{equation}
We can compute
\begin{equation}
\begin{split}
&\Delta_{\varphi}\big(e^{-\kappa(F+\delta f+C\varphi)}\big)=e^{-\kappa(F+\delta f+C\varphi)}\big(\kappa^2|\nabla_{\varphi}(F+\delta f+C\varphi)|_{\varphi}^2\\
&-\kappa\Delta_{\varphi}(F+f+C\varphi)+\kappa(1-\delta)\Delta_{\varphi}f\big)\\
&=e^{-\kappa(F +\delta f+C\varphi)}\kappa^2|\nabla_{\varphi}(F+\delta f+C\varphi)|_{\varphi}^2\\
&\quad\quad\quad\quad\quad+e^{-\kappa(F +\delta f+C\varphi)}\kappa\big(Ctr_{\varphi}g-tr_{\varphi}(Ric-\beta_0)\big)\\
&\quad\quad\quad\quad\quad +e^{-\kappa(F +\delta f+C\varphi)}(\kappa R-\kappa Cn)+\kappa(1-\delta)e^{-\kappa(F+\delta f+C\varphi)}\Delta_{\varphi}f.
\end{split}
\end{equation}
We choose $C\geq 2(\max_M|Ric-\beta_0|_g+1)$, then we obtain from above:
\begin{equation}\label{2.24}
\begin{split}
\Delta_{\varphi}\big(&e^{-\kappa(F +\delta f+C\varphi)}\big)\geq e^{-\kappa(F +\delta f+C\varphi)}\kappa^2|\nabla_{\varphi}(F+\delta f+C\varphi)|_{\varphi}^2\\
&+e^{-\kappa(F +\delta f+C\varphi)}\frac{\kappa C}{2} tr_{\varphi}g
+\kappa e^{-\kappa(F +\delta f+C\varphi)}(1-\delta)\Delta_{\varphi}f-\kappa e^{-\kappa(F +\delta f+C\varphi)}C_9.
\end{split}
\end{equation}
The constant $C_9$ appearing above will depend on our choice of $C$. 
On the other hand, let $p\in M$, we choose normal coordinate in a neighborhood of $p$ so that
$$
g_{i\bar{j}}(p)=\delta_{ij},\,\,\nabla g_{i\bar{j}}(p)=0,\,\,\varphi_{i\bar{j}}=\varphi_{i\bar{i}}\delta_{ij}.
$$
we have computed in our first paper \cite{cc1} (following Yau \cite{Yau78}) that
\begin{equation}\label{2.25}
\begin{split}
\Delta_{\varphi}&(\Delta\varphi)=\frac{R_{i\bar{i}\alpha\bar{\alpha}}(1+\varphi_{i\bar{i}})}{1+\varphi_{\alpha\bar{\alpha}}}+\frac{|\varphi_{\alpha\bar{\beta}i}|^2}{(1+\varphi_{\alpha\bar{\alpha}})(1+\varphi_{\beta\bar{\beta}})}+\Delta F-R\\
&\geq -C_{10}tr_{\varphi}g(n+\Delta\varphi)+\frac{|\varphi_{\alpha\bar{\beta}i}|^2}{(1+\varphi_{\alpha\bar{\alpha}})(1+\varphi_{\beta\bar{\beta}})}+\Delta F-R.
\end{split}
\end{equation}
Here $C_{10}$ depends only on the curvature bound of $g$. 
Also we notice the complete square similar to our calculation in cscK case:
\begin{equation*}
\begin{split}
&\kappa^2|\nabla_{\varphi}(F+\delta f+C\varphi)|_{\varphi}^2(n+\Delta\varphi)+\frac{|\varphi_{\alpha\bar{\beta}i}|^2}{(1+\varphi_{\alpha\bar{\alpha}})(1+\varphi_{\beta\bar{\beta}})}\\
&\quad\quad\quad\quad-2\kappa\nabla_{\varphi}(F+\delta F+C\varphi)\cdot_{\varphi}\nabla_{\varphi}(\Delta\varphi)\\
&\geq \kappa^2|\nabla_{\varphi}(F+\delta f+C\varphi)|_{\varphi}^2(n+\Delta\varphi)+\frac{|\nabla_{\varphi}(\Delta\varphi)|^2_{\varphi}}{n+\Delta\varphi}\\
&\quad\quad\quad\quad-2\kappa\nabla_{\varphi}(F+\delta f+C\varphi)\cdot_{\varphi}\nabla_{\varphi}\Delta\varphi\geq0.
\end{split}
\end{equation*}
Combining (\ref{2.22}), (\ref{2.24}) and (\ref{2.25}), we conclude
\begin{equation}
\begin{split}
&\Delta_{\varphi}\big(e^{-\kappa(F+\delta f+C\varphi)}(n+\Delta\varphi)\big)\geq e^{-\kappa(F+\delta f+C\varphi)}\big(\frac{\kappa C}{2}-C_{10}\big)tr_{\varphi}g(n+\Delta\varphi)\\
&+\kappa e^{-\kappa(F+\delta f+C\varphi)}(1-\delta)\Delta_{\varphi}f(n+\Delta\varphi)+e^{-\kappa(F+\delta f+C\varphi)}\Delta F-e^{-\kappa(F+\delta f+C\varphi)}(C_9\kappa+R).
\end{split}
\end{equation}
In the following, we will always choose $\kappa\geq1$, hence if we choose $C\geq 4C_{10}$, we obtain for some constant $C_{11}$, it holds:
\begin{equation}
\begin{split}
\Delta_{\varphi}\big(e^{-\kappa(F+\delta f+C\varphi)}(n+\Delta\varphi)\big)&e^{\kappa(F+\delta f+C\varphi)}\geq \frac{\kappa C}{4}tr_{\varphi}g(n+\Delta\varphi)\\
&+\kappa (1-\delta )\Delta_{\varphi}f(n+\Delta\varphi)+\Delta F-\kappa C_{11}.
\end{split}
\end{equation}
The constant $C_{11}$ above will depend on our choice of $C$.
Let $p\geq 1$, denote $v=e^{-\kappa(F+\delta f+C\varphi)}(n+\Delta\varphi)$, we have
\begin{equation}\label{2.28}
\begin{split}
&\int_M(p-1)v^{p-2}|\nabla_{\varphi}v|_{\varphi}^2dvol_{\varphi}=\int_Mv^{p-1}(-\Delta_{\varphi}v)dvol_{\varphi}\\
&\leq -\int_Mv^{p-1}\big(\frac{\kappa C}{4}vtr_{\varphi}g+e^{-\kappa(F+\delta f+C\varphi)}\kappa(1-\delta)\Delta_{\varphi}f(n+\Delta\varphi)\\
&+e^{-\kappa(F+\delta f+C\varphi)}\Delta F-\kappa C_{11}e^{-\kappa(F+\delta f+C\varphi)}\big)dvol_{\varphi}.
\end{split}
\end{equation}
We will handle the term involving $\Delta F$ via integrating by parts, but somewhat differently from the calculation for cscK(here we assume $\kappa>1$):
\begin{equation}\label{2.29}
\begin{split}
-&\int_Mv^{p-1}e^{-\kappa(F+\delta f+C\varphi)}\Delta Fdvol_{\varphi}=-\int_Mv^{p-1}e^{(1-\kappa)F-\kappa\delta f-\kappa C\varphi}\Delta Fdvol_g\\
&=-\int_Mv^{p-1}e^{(1-\kappa)F-\kappa\delta f-\kappa C\varphi}\frac{1}{1-\kappa}\Delta\big((1-\kappa)F-\kappa\delta f-\kappa C\varphi\big)dvol_g\\
&-\int_Mv^{p-1}e^{(1-\kappa)F-\kappa\delta f-\kappa C\varphi}\frac{\kappa\delta \Delta f+\kappa C\Delta\varphi}{1-\kappa}dvol_g.
\end{split}
\end{equation}
For the first term in (\ref{2.29}), we have
\begin{equation}\label{2.30}
\begin{split}
&-\int_Mv^{p-1}e^{(1-\kappa)F-\kappa\delta f-\kappa C\varphi}\frac{1}{1-\kappa}\Delta\big((1-\kappa)F-\kappa\delta f-\kappa C\varphi\big)dvol_g\\
&=-\int_M\frac{v^{p-1}e^{(1-\kappa)F-\kappa\delta f-\kappa C\varphi}}{\kappa-1}|\nabla\big((1-\kappa)F-\kappa\delta f-\kappa C\varphi\big)|^2dvol_g\\
&-\int_M\frac{p-1}{\kappa-1}v^{p-2}e^{(1-\kappa)F-\kappa\delta f-\kappa C\varphi}\nabla v\cdot\nabla\big((1-\kappa)F-\kappa\delta f-\kappa C\varphi\big)dvol_g\\
&\leq \int_M\frac{(p-1)^2}{2(\kappa-1)}v^{p-3}e^{(1-\kappa)F-\kappa\delta f-\kappa C\varphi}|\nabla v|^2dvol_g\\
&\leq\int_M\frac{(p-1)^2}{2(\kappa-1)}v^{p-3}e^{-\kappa(F+\delta f+C\varphi)}|\nabla_{\varphi}v|_{\varphi}^2(n+\Delta\varphi)dvol_{\varphi}\\
&=\int_M\frac{(p-1)^2}{2(\kappa-1)}v^{p-2}|\nabla_{\varphi}v|_{\varphi}^2dvol_{\varphi}.
\end{split}
\end{equation}
From 3rd line to 4th line above, we observed that
\begin{equation*}
\begin{split}
-&\frac{p-1}{\kappa-1}v^{p-2}e^{(1-\kappa)F-\kappa\delta f-\kappa C\varphi}\nabla v\cdot\nabla\big((1-\kappa)F-\kappa\delta f-\kappa C\varphi\big)\\
&\leq\frac{v^{p-1}e^{(1-\kappa)F-\kappa\delta f-\kappa C\varphi}}{2(\kappa-1)}|\nabla\big((1-\kappa)F-\kappa\delta f-\kappa C\varphi\big)|^2\\
&+\frac{(p-1)^2}{2(\kappa-1)}v^{p-3}e^{(1-\kappa)F-\kappa\delta f-\kappa C\varphi}|\nabla v|^2.
\end{split}
\end{equation*}
Combining (\ref{2.29}), (\ref{2.30}), we see
\begin{equation}\label{2.31n}
\begin{split}
-\int_Mv^{p-1}e^{-\kappa(F+\delta f+C\varphi)}&\Delta Fdvol_{\varphi}\leq \int_M\frac{(p-1)^2}{2(\kappa-1)}v^{p-2}|\nabla_{\varphi}v|_{\varphi}^2dvol_{\varphi}\\
&-\int_Mv^{p-1}e^{-\kappa(F+\delta f+C\varphi)}\frac{\kappa\delta \Delta f+\kappa C\Delta\varphi}{1-\kappa}dvol_{\varphi}.
\end{split}
\end{equation}

Plug (\ref{2.31n}) back to (\ref{2.28}), we obtain
\begin{equation}\label{2.31}
\begin{split}
&\int_M\big(p-1-\frac{(p-1)^2}{2(\kappa-1)}\big)v^{p-2}|\nabla_{\varphi}v|_{\varphi}^2dvol_{\varphi}\leq-\int_M\frac{\kappa C}{4}tr_{\varphi}gv^pdvol_{\varphi}\\
&+\int_Mv^{p-1}e^{-\kappa(F+\delta f+C\varphi)}\big(-\kappa(1-\delta)\Delta_{\varphi}f(n+\Delta\varphi)-\frac{\kappa\delta\Delta f}{1-\kappa}\big)dvol_{\varphi}\\
&+\int_Mv^{p-1}e^{-\kappa(F+\delta f+C\varphi)}\big(\kappa C_{11}-\frac{\kappa C\Delta\varphi}{1-\kappa}\big)dvol_{\varphi}.
\end{split}
\end{equation}
Now we choose $\delta=\frac{\kappa-1}{\kappa}$, then we have
\begin{equation}\label{2.32}
\begin{split}
-\kappa&(1-\delta)\Delta_{\varphi}f(n+\Delta\varphi)-\frac{\kappa\delta\Delta f}{1-\kappa}=-\Delta_{\varphi}f(n+\Delta\varphi)+\Delta f\\
&=-\sum_{i\neq j}\frac{f_{i\bar{i}}(1+\varphi_{j\bar{j}})}{1+\varphi_{i\bar{i}}}\leq\sum_{i\neq j}\frac{(\beta_0)_{i\bar{i}}(1+\varphi_{j\bar{j}})}{1+\varphi_{i\bar{i}}}\leq\max_{M}|\beta_0|_gtr_{\varphi}g(n+\Delta\varphi).
\end{split}
\end{equation}
We also have for $\kappa\geq 2$, 
\begin{equation}\label{2.33}
\kappa C_{11}-\frac{\kappa C\Delta\varphi}{1-\kappa}\leq \kappa (C_{12}+C)(n+\Delta\varphi).
\end{equation}
Here we used the fact that $n+\Delta\varphi\geq e^{\frac{F}{n}}$, which is bounded from below in terms of $||f+F||_0$. Indeed, $F\geq -f-||f+F||_0\geq -||f+F||_0$.
Hence if we plug (\ref{2.32}), (\ref{2.33}) back to (\ref{2.31}), we conclude that for $p\geq1$, $\kappa\geq2$, $C$ chosen sufficiently large depending only on the curvature bound of the background metric
and $\max_M|\beta_0|_g$, we have
\begin{equation}\label{2.34}
\begin{split}
\int_M&\big(p-1-\frac{(p-1)^2}{\kappa-1}\big)v^{p-2}|\nabla_{\varphi}v|_{\varphi}^2dvol_{\varphi}+\int_M\big(\frac{\kappa C}{4}-\max_M|\beta_0|_g\big)tr_{\varphi}gv^pdvol_{\varphi}\\
&\leq\int_M\kappa(C_{12}+C)v^pdvol_{\varphi}.
\end{split}
\end{equation}
Next we choose $\kappa$ so that $\kappa\geq 2$ and $\kappa\geq p$, with this choice, we have
$$
p-1-\frac{(p-1)^2}{\kappa-1}\geq0.
$$
Choose $C$ sufficiently so as to satisfy $C\geq 8(\max_M|\beta_0|_g+1)$, with this choice, we can guarantee
$$
\frac{\kappa C}{4}-\max_M|\beta_0|_g\geq\frac{\kappa C}{8}\geq\kappa.
$$
Hence we obtain from (\ref{2.34}) that for some constant $C_{13}$
\begin{equation}
\int_Me^{-\frac{F}{n-1}}(n+\Delta\varphi)^{\frac{1}{n-1}}v^pdvol_{\varphi}\leq\int_Mtr_{\varphi}gv^pdvol_{\varphi}\leq\int_MC_{13}v^pdvol_{\varphi}
\end{equation}
Recall our definition for $v$, this means:
\begin{equation}
\begin{split}
\int_Me^{(\frac{n-2}{n-1}-p\kappa)F-p(\kappa-1)f-p\kappa C\varphi}&(n+\Delta\varphi)^{p+\frac{1}{n-1}}dvol_g\\
&\leq C_{13}\int_Me^{(1-p\kappa)F-p(\kappa-1)f-p\kappa C\varphi}(n+\Delta\varphi)^pdvol_g.
\end{split}
\end{equation}
From the boundedness of $F+f$ and $\varphi$ proved in Corollary \ref{c2.2} and Lemma \ref{l2.3}, we obtain for $p\geq 1$:
\begin{equation}\label{2.37}
\int_Me^{(p-\frac{n-2}{n-1})f}(n+\Delta\varphi)^{p+\frac{1}{n-1}}dvol_g\leq C_{14}\int_Me^{(p-1)f}(n+\Delta\varphi)^pdvol_g.
\end{equation}

Take $p=1+k\frac{1}{n-1}$ in (\ref{2.37}) with $k\geq0$, the result follows from induction on $k$.
\end{proof}

As a consequence of above calculation, we obtain:
\begin{cor}\label{c2.4}
For any $1<q<p_0$, there exists a constant $\tilde{C}_q$, depending only on the bound $\int_Me^{-p_0f}dvol_g$, $||R||_0$, $\max_M|\beta|_g$, the bound $||F+f||_0$, $||\varphi||_0$, the background metric $(M,\omega_0)$, and $q$, such that
$$
\int_M(n+\Delta\varphi)^qdvol_g\leq\tilde{C}_q.
$$
Besides, $\tilde{C}_q$ is uniform in $q$ as long as $q$ is bounded away from $p_0$ and remains bounded.
\end{cor}
\begin{proof}
Choose $s=\frac{(q-1)p_0}{p_0-1}$, then we can calculate
\begin{equation*}
\begin{split}
\int_M(n+\Delta\varphi)^qdvol_g&=\int_Me^{-sf}\cdot e^{sf}(n+\Delta\varphi)^qdvol_g\\
&\leq\bigg(\int_Me^{-p_0f}dvol_g\bigg)^{\frac{s}{p_0}}\cdot\bigg(\int_Me^{\frac{sp_0}{p_0-s}f}(n+\Delta\varphi)^{\frac{p_0q}{p_0-s}}dvol_g\bigg)^{1-\frac{s}{p_0}}.
\end{split}
\end{equation*}
Notice our choice of $s$ makes $\frac{sp_0}{p_0-s}=\frac{p_0q}{p_0-s}-1$, so the result follows from Theorem \ref{t2.1}.
\end{proof}

\subsection{Estimate on $\nabla(F+f)$}
In this section, we continues to assume $\beta\geq0$. Moreover, we also need $p_0$ to be sufficiently large depending only on $n$. Our goal is to obtain the following estimate.
\begin{thm}\label{t2.2}
There exists $\kappa_n$, depending only on $n$, such that as long as $p_0> \kappa_n$, we have
$$
|\nabla_{\varphi}(F+f)|_{\varphi}\leq C_{14}.
$$
Here $C_{14}$ is a constant with the same dependence as in Theorem \ref{t2.1}.
\end{thm}
\begin{proof}
Denote $w=F+f$, we need to calculate:
\begin{equation}\label{8.58}
\begin{split}
&\Delta_{\varphi}\big(e^{\frac{1}{2}w}|\nabla_{\varphi}w|_{\varphi}^2\big)=\Delta_{\varphi}(e^{\frac{1}{2}w})|\nabla_{\varphi}w|_{\varphi}^2\\
&\quad\quad\quad\quad+e^{\frac{1}{2}w}\Delta_{\varphi}(|\nabla_{\varphi}w|_{\varphi}^2)+e^{\frac{1}{2}w}\nabla_{\varphi}w\cdot_{\varphi}\nabla_{\varphi}(|\nabla_{\varphi}w|_{\varphi}^2)\\
&=\frac{1}{4}e^{\frac{1}{2}w}|\nabla_{\varphi}w|_{\varphi}^4+\frac{1}{2}e^{\frac{1}{2}w}\Delta_{\varphi}w|\nabla_{\varphi}w|_{\varphi}^2\\
&\quad\quad\quad\quad+e^{\frac{1}{2}w}\Delta_{\varphi}(|\nabla_{\varphi}w|_{\varphi}^2)+e^{\frac{1}{2}w}\nabla_{\varphi}w\cdot_{\varphi}\nabla_{\varphi}(|\nabla_{\varphi}w|_{\varphi}^2).
\end{split}
\end{equation}
Now we have 
\begin{equation}
\Delta_{\varphi}w=tr_{\varphi}(Ric-\beta_0)-R.
\end{equation}
Also
\begin{equation}
\begin{split}
\Delta_{\varphi}(|\nabla_{\varphi}w|_{\varphi}^2)=g_{\varphi}^{i\bar{j}}&g_{\varphi}^{\kappa\bar{\beta}}w_{,\kappa i}w_{,\bar{\beta}\bar{j}}+g_{\varphi}^{i\bar{j}}g_{\varphi}^{\kappa\bar{\beta}}w_{,\alpha\bar{j}}w_{,\bar{\beta}i}+2\nabla_{\varphi}w\cdot_{\varphi}\nabla_{\varphi}\Delta_{\varphi}w\\
&+g_{\varphi}^{i\bar{j}}g_{\varphi}^{\alpha\bar{\beta}}Ric_{\varphi,i\bar{\beta}}w_{\alpha}w_{\bar{j}}.
\end{split}
\end{equation}
Besides,
\begin{equation}
\nabla_{\varphi}w\cdot_{\varphi}\nabla_{\varphi}(|\nabla_{\varphi}w|^2)=Re\big(g_{\varphi}^{i\bar{j}}g_{\varphi}^{\alpha\bar{\beta}}w_i(w_{,\alpha\bar{j}}w_{\bar{\beta}}+w_{\alpha}w_{,\bar{\beta}\bar{j}})\big).
\end{equation}
In the above, $w_{,i\alpha}$ denotes the covariant derivative under the metric $g_{\varphi}$.
Again observe the complete square:
\begin{equation*}
\begin{split}
\frac{1}{4}&|\nabla_{\varphi}w|_{\varphi}^4+g_{\varphi}^{i\bar{j}}g_{\varphi}^{\alpha\bar{\beta}}w_{,\alpha i}w_{,\bar{\beta}\bar{j}}+Re\big(g_{\varphi}^{i\bar{j}}g_{\varphi}^{\alpha\bar{\beta}}w_iw_{\alpha}w_{,\bar{\beta}\bar{j}}\big)\\
&=g_{\varphi}^{i\bar{j}}g_{\varphi}^{\alpha\bar{\beta}}(w_{,i\alpha}+\frac{1}{2}w_iw_{\alpha})(w_{,\bar{j}\bar{\beta}}+\frac{1}{2}w_{\bar{j}}w_{\bar{\beta}}).
\end{split}
\end{equation*}
Hence we obtain from (\ref{8.58}):
\begin{equation}\label{8.62}
\begin{split}
&\Delta_{\varphi}(e^{\frac{1}{2}w}|\nabla_{\varphi}w|_{\varphi}^2)\geq\frac{1}{2}e^{\frac{1}{2}w}|\nabla_{\varphi}w|_{\varphi}^2\big(tr_{\varphi}(Ric-\beta_0)-R\big)+e^{\frac{1}{2}w}g_{\varphi}^{i\bar{j}}g_{\varphi}^{\alpha\bar{\beta}}w_{,\alpha\bar{j}}w_{,\bar{\beta}i}\\
&+e^{\frac{1}{2}w}2\nabla_{\varphi}w\cdot_{\varphi}\nabla_{\varphi}\Delta_{\varphi}w+g_{\varphi}^{i\bar{j}}g_{\varphi}^{\alpha\bar{\beta}}Ric_{\varphi,i\bar{\beta}}w_{\alpha}w_{\bar{j}}+Re\big(g_{\varphi}^{i\bar{j}}g_{\varphi}^{\alpha\bar{\beta}}w_iw_{\bar{\beta}}w_{,\alpha\bar{j}}\big).
\end{split}
\end{equation}
Note that 
$$
Ric_{\varphi,i\bar{\beta}}=Ric_{i\bar{\beta}}-F_{i\bar{\beta}}.
$$
Hence
\begin{equation}
\begin{split}
&g_{\varphi}^{i\bar{j}}g_{\varphi}^{\alpha\bar{\beta}}Ric_{\varphi,i\bar{\beta}}w_{\alpha}w_{\bar{j}}+Re\big(g_{\varphi}^{i\bar{j}}g_{\varphi}^{\alpha\bar{\beta}}w_iw_{\bar{\beta}}w_{,\alpha\bar{j}}\big)\\
&=g_{\varphi}^{i\bar{j}}g_{\varphi}^{\alpha\bar{\beta}}Ric_{i\bar{\beta}}w_{\alpha}w_{\bar{j}}+Re\big(g_{\varphi}^{i\bar{j}}g_{\varphi}^{\alpha\bar{\beta}}w_iw_{\bar{\beta}}(w_{,\alpha\bar{j}}-F_{\alpha\bar{j}})\big)\\
&=g_{\varphi}^{i\bar{j}}g_{\varphi}^{\alpha\bar{\beta}}Ric_{i\bar{\beta}}w_{\alpha}w_{\bar{j}}+Re\big(g_{\varphi}^{i\bar{j}}g_{\varphi}^{\alpha\bar{\beta}}w_iw_{\bar{\beta}}f_{\alpha\bar{j}}\big)\\
&\geq g_{\varphi}^{i\bar{j}}g_{\varphi}^{\alpha\bar{\beta}}Ric_{i\bar{\beta}}w_{\alpha}w_{\bar{j}}-Re\big(g_{\varphi}^{i\bar{j}}g_{\varphi}^{\alpha\bar{\beta}}w_iw_{\bar{\beta}}(\beta_0)_{\alpha\bar{j}}\big).
\end{split}
\end{equation}
In the last line, we use the fact that $\sqrt{-1}\partial\bar{\partial}f=\beta-\beta_0\geq-\beta_0$, hence $f_{i\bar{j}}\geq-(\beta_0)_{i\bar{j}}$.
Hence we obtain from (\ref{8.62}):
\begin{equation}\label{8.64}
\begin{split}
\Delta_{\varphi}&(e^{\frac{1}{2}w}|\nabla_{\varphi}w|_{\varphi}^2)\geq \frac{1}{2}e^{\frac{1}{2}w}|\nabla_{\varphi}w|_{\varphi}^2\big(tr_{\varphi}(Ric-\beta_0)-R\big)+e^{\frac{1}{2}w}2\nabla_{\varphi}w\cdot_{\varphi}\nabla_{\varphi}\Delta_{\varphi}w\\
&+e^{\frac{1}{2}w}g_{\varphi}^{i\bar{j}}g_{\varphi}^{\alpha\bar{\beta}}Ric_{i\bar{\beta}}w_{\alpha}w_{\bar{j}}-e^{\frac{1}{2}w}Re\big(g_{\varphi}^{i\bar{j}}g_{\varphi}^{\alpha\bar{\beta}}w_iw_{\bar{\beta}}(\beta_0)_{\alpha\bar{j}}\big).
\end{split}
\end{equation}
Next we estimate:
\begin{equation}
tr_{\varphi}\big(Ric-\beta_0\big)-R\geq -C_{15}(tr_{\varphi}\omega_0+1)\geq -C_{15}(e^{-F}(n+\Delta\varphi)^{n-1}+1).
\end{equation}
Also
\begin{equation}
\begin{split}
g_{\varphi}^{i\bar{j}}&g_{\varphi}^{\alpha\bar{\beta}}Ric_{i\bar{\beta}}w_{\alpha}w_{\bar{j}}\geq -C_{14.5}(tr_{\varphi}\omega_0)^2|\nabla w|^2\geq -C_{15}(tr_{\varphi}\omega_0)^2(n+\Delta\varphi)|\nabla_{\varphi}w|_{\varphi}^2\\
&\geq -C_{15}e^{-2F}(n+\Delta\varphi)^{2n-1}|\nabla_{\varphi}w|_{\varphi}^2.
\end{split}
\end{equation}
We can also estimate
\begin{equation}
\begin{split}
-Re\big(g_{\varphi}^{i\bar{j}}g_{\varphi}^{\alpha\bar{\beta}}w_iw_{\bar{\beta}}(\beta_0)_{\alpha\bar{j}}\big)&\geq -C_{14.5}(tr_{\varphi}\omega_0)^2|\nabla w|^2\\
&\geq -C_{15}e^{-2F}(n+\Delta\varphi)^{2n-1}|\nabla_{\varphi}w|_{\varphi}^2.
\end{split}
\end{equation}
Hence we may conclude from (\ref{8.64}) that
\begin{equation}
\begin{split}
\Delta_{\varphi}&(e^{\frac{1}{2}w}|\nabla_{\varphi}w|_{\varphi}^2)\geq 2e^{\frac{1}{2}w}\nabla_{\varphi}w\cdot_{\varphi}\nabla_{\varphi}\Delta_{\varphi}w-e^{\frac{1}{2}w}|\nabla_{\varphi}w|_{\varphi}^2\\
&\times C_{15}\big(2e^{-2F}(n+\Delta\varphi)^{2n-1}+e^{-F}(n+\Delta\varphi)^{n-1}+1\big).
\end{split}
\end{equation}
Denote $u=e^{\frac{1}{2}w}|\nabla_{\varphi}w|_{\varphi}^2+1$, $\tilde{G}=C_{15}\big(2e^{-2F}(n+\Delta\varphi)^{2n-1}+e^{-F}(n+\Delta\varphi)^{n-1}+1\big)$.
Then we have
\begin{equation}
\Delta_{\varphi}u\geq 2e^{\frac{1}{2}w}\nabla_{\varphi}w\cdot_{\varphi}\nabla_{\varphi}\Delta_{\varphi}w-u\tilde{G}.
\end{equation}
Now let $p\geq 1$, then we have 
\begin{equation}\label{8.70}
\begin{split}
\int_M&(p-1)u^{p-2}|\nabla_{\varphi}u|_{\varphi}^2dvol_{\varphi}=\int_Mu^{p-1}(-\Delta_{\varphi}u)dvol_{\varphi}\\
&\leq \int_Mu^p\tilde{G}dvol_{\varphi}-\int_M2u^{p-1}e^{\frac{1}{2}w}\nabla_{\varphi}w\cdot_{\varphi}\nabla_{\varphi}\Delta_{\varphi}wdvol_{\varphi}.
\end{split}
\end{equation}
We need to integrate by parts to handle the last term above. We have
\begin{equation}
\begin{split}
&-\int_M2u^{p-1}e^{\frac{1}{2}w}\nabla_{\varphi}w\cdot_{\varphi}\nabla_{\varphi}\Delta_{\varphi}wdvol_{\varphi}=\int_M2u^{p-1}e^{\frac{1}{2}w}(\Delta_{\varphi}w)^2dvol_{\varphi}\\
&+\int_Mu^{p-1}e^{\frac{1}{2}w}|\nabla_{\varphi}w|_{\varphi}^2\Delta_{\varphi}wdvol_{\varphi}+\int_M2(p-1)u^{p-2}e^{\frac{1}{2}w}\nabla_{\varphi}u\cdot_{\varphi}\nabla_{\varphi}w\Delta_{\varphi}wdvol_{\varphi}\\
&\leq \int_M2u^{p-1}e^{\frac{1}{2}w}(\Delta_{\varphi}w)^2dvol_{\varphi}+\int_Mu^p\Delta_{\varphi}wdvol_{\varphi}-\int_Mu^{p-1}\Delta_{\varphi}wdvol_{\varphi}\\
&+\int_M\frac{p-1}{2}u^{p-2}|\nabla_{\varphi}u|_{\varphi}^2dvol_{\varphi}+\int_M2(p-1)u^{p-2}e^w|\nabla_{\varphi}w|_{\varphi}^2(\Delta_{\varphi}w)^2dvol_{\varphi}\\
&\leq \int_M2p u^{p-1}e^{\frac{1}{2}w}(\Delta_{\varphi}w)^2dvol_{\varphi}+\int_Mu^p\big((\Delta_{\varphi}w)^2+1\big)dvol_{\varphi}\\
&+\int_M\frac{p-1}{2}u^{p-2}|\nabla_{\varphi}u|_{\varphi}^2dvol_{\varphi}.
\end{split}
\end{equation}
Some explanations of above calculations are in order.

In the first inequality, we observed that $u^{p-1}e^{\frac{1}{2}w}|\nabla_{\varphi}w|^2_{\varphi}\Delta_{\varphi}w=u^p\Delta_{\varphi}w-u^{p-1}\Delta_{\varphi}w$, from our definition of $u$.
Also we observed that
$$
2(p-1)u^{p-2}e^{\frac{1}{2}w}\nabla_{\varphi}u\cdot_{\varphi}\nabla_{\varphi}w\Delta_{\varphi}w\leq\frac{p-1}{2}u^{p-2}|\nabla_{\varphi}u|^2_{\varphi}+2(p-1)u^{p-2}e^w|\nabla_{\varphi}w|^2_{\varphi}(\Delta_{\varphi}w)^2.
$$

In the second inequality, we noticed that 
$$u^p\Delta_{\varphi}w-u^{p-1}\Delta_{\varphi}w\leq\frac{1}{2} (u^p+u^{p-1})\big(1+(\Delta_{\varphi}w)^2\big)\leq u^p\big(1+(\Delta_{\varphi}w)^2\big).
$$
Hence we conclude from (\ref{8.70}):
\begin{equation}
\begin{split}
&\int_M\frac{p-1}{2}u^{p-2}|\nabla_{\varphi}u|_{\varphi}^2dvol_{\varphi}\leq\int_Mu^p\big(\tilde{G}+(\Delta_{\varphi}w)^2+1\big)dvol_{\varphi}\\
&+\int_M2pu^{p-1}e^{\frac{1}{2}w}(\Delta_{\varphi}w)^2dvol_{\varphi}\\
&\leq \int_Mu^p\big(\tilde{G}+(\Delta_{\varphi}w)^2+1\big)dvol_{\varphi}+\int_M2pu^pe^{\frac{1}{2}w}(\Delta_{\varphi}w)^2dvol_{\varphi}.
\end{split}
\end{equation}
From 1st line to 2nd line above, we noticed $u\geq 1$. 
Now denote\\ $G=\tilde{G}+(\Delta_{\varphi}w)^2+1+2e^{\frac{1}{2}w}(\Delta_{\varphi}w)^2$, we have
\begin{equation}\label{2.54new}
\begin{split}
\int_M&\frac{p-1}{2}u^{p-2}|\nabla_{\varphi}u|_{\varphi}^2dvol_{\varphi}\leq \int_Mp u^pGe^Fdvol_g.
\end{split}
\end{equation}
For the left hand side, we have
\begin{equation}\label{2.55new}
\int_M\frac{p-1}{2}u^{p-2}|\nabla_{\varphi}u|_{\varphi}^2dvol_{\varphi}\geq \frac{1}{C_{16}}\int_M\frac{2(p-1)}{p^2}|\nabla_{\varphi}(u^{\frac{p}{2}})|_{\varphi}^2dvol_g.
\end{equation}
Let $\eps>0$ to be determined, we can also estimate
\begin{equation}\label{8.75}
\begin{split}
\int_M&|\nabla(u^{\frac{p}{2}})|^{2-\eps}dvol_g\leq \int_M|\nabla_{\varphi}(u^{\frac{p}{2}})|_{\varphi}^{2-\eps}(n+\Delta\varphi)^{1-\eps/2}dvol_g\\
&\leq \bigg(\int_M|\nabla_{\varphi}(u^{\frac{p}{2}})|_{\varphi}^2dvol_g\bigg)^{\frac{2-\eps}{2}}\times\bigg(\int_M(n+\Delta\varphi)^{\frac{2}{\eps}-1}dvol_g\bigg)^{\frac{\eps}{2}}.
\end{split}
\end{equation}
Denote
\begin{equation}\label{2.57new}
K_{\eps}=\bigg(\int_M(n+\Delta\varphi)^{\frac{2}{\eps}-1}dvol_g\bigg)^{\frac{\eps}{2-\eps}}.
\end{equation}
Then we have
\begin{equation}
\begin{split}
||\nabla(u^{\frac{p}{2}})||_{L^{2-\eps}(\omega_0^n)}^2&\leq K_{\eps}\Vert|\nabla_{\varphi}(u^{\frac{p}{2}})|_{\varphi}\Vert_{L^2(\omega_0^n)}^2\leq\frac{K_{\eps}C_{16}p^2}{4}\int_Mu^{p-2}|\nabla_{\varphi}u|^2_{\varphi}dvol_{\varphi}\\
&\leq\frac{K_{\eps}C_{16}p^3}{2(p-1)}\int_Mu^p Ge^Fdvol_g.
\end{split}
\end{equation}
In the above, the first inequality follows from (\ref{8.75}).
The second inequality follows from (\ref{2.55new}), and the last inequality uses (\ref{2.54new}).

Apply the Sobolev inequality with exponent $2-\eps$ to conclude
\begin{equation}\label{8.77}
\begin{split}
||&u^{\frac{p}{2}}||_{L^{\frac{2n(2-\eps)}{2n-2+\eps}}}^2\leq C_{\eps}\big(||\nabla(u^{\frac{p}{2}})||_{L^{2-\eps}}^2+||u^{\frac{p}{2}}||_{L^{2-\eps}}^2\big)\\
&\leq C_{\eps}\big(\frac{K_{\eps}C_{16}p^3}{2(p-1)}\int_Mu^p Ge^Fdvol_g+||u^{\frac{p}{2}}||_{L^{2-\eps}}^2\big)\\
&\leq D_{\eps}\big(\frac{K_{\eps}C_{16}p^3}{2(p-1)}\bigg(\int_Mu^{\frac{2p}{2-\eps}}dvol_g\bigg)^{\frac{2-\eps}{2}}\times\bigg(\int_MG^{\frac{2}{\eps}}e^{\frac{2F}{\eps}}dvol_g\bigg)^{\frac{\eps}{2}}+||u^{\frac{p}{2}}||_{L^{\frac{4}{2-\eps}}}^2\big).
\end{split}
\end{equation}
In the last line above, we use H\"older's inequality to estimate $||u^{\frac{p}{2}}||_{L^{2-\eps}}$, and $D_{\eps}$ depends on $C_{\eps}$ and $vol(M)$.
Denote 
\begin{equation}\label{2.60new}
L_{\eps}=\bigg(\int_MG^{\frac{2}{\eps}}e^{\frac{2F}{\eps}}dvol_g\bigg)^{\frac{\eps}{2}}.
\end{equation}
Also we choose $\eps$ to be sufficiently small so that the following holds:
\begin{equation}\label{2.58}
\frac{2n(2-\eps)}{2n-2+\eps}>\frac{4}{2-\eps}.
\end{equation}
Hence we may conclude from (\ref{8.77}) that 
\begin{equation}\label{8.78}
||u^{\frac{p}{2}}||^2_{L^{\frac{2n(2-\eps)}{2n-2+\eps}}}\leq C_{18}\frac{p^3}{p-1}(K_{\eps}L_{\eps}+1)||u^{\frac{p}{2}}||^2_{L^{\frac{4}{2-\eps}}}.
\end{equation}
We need to have a bound for $K_{\eps}$, $L_{\eps}$. 
Choose $\eps=\frac{1}{2n}$, it is clear that this $\eps$ verifies (\ref{2.58}) since $n\geq2$.
With this choice, from the expressions of $K_{\eps}$ and $L_{\eps}$ in (\ref{2.57new}) and (\ref{2.60new}), we need $\int_M(n+\Delta\varphi)^{4n-1}dvol_g$, $\int_MG^{4n}e^{4nF}dvol_g$ is bounded.
First from Corollary \ref{c2.4}, if $p_0\geq 4n$, then $\int_M(n+\Delta\varphi)^{4n-1}dvol_g$ is bounded.

While for $\int_MG^{4n}e^{4n F}dvol_g$, first we have
\begin{equation}
\begin{split}
&G=C_{15}\big(2e^{-2F}(n+\Delta\varphi)^{2n-1}+e^{-F}(n+\Delta\varphi)^{n-1}+1\big)+\big(\Delta_{\varphi}(F+f)\big)^2+1\\
&+2e^{\frac{1}{2}(F+f)}(\Delta_{\varphi}(F+f))^2\leq C_{21}(n+\Delta\varphi)^{2n-1}+C_{21}(tr_{\varphi}g)+C_{21}(tr_{\varphi}g)^2\\
&\leq C_{21}(n+\Delta\varphi)^{2n-1}+C_{21}e^{-F}(n+\Delta\varphi)^{n-1}+C_{21}e^{-2F}(n+\Delta\varphi)^{2n-2}\\
&\leq C_{22}(n+\Delta\varphi)^{2n-1}.
\end{split}
\end{equation}
Hence
\begin{equation*}
\begin{split}
\int_M&G^{\frac{2}{\eps}}e^{\frac{2F}{\eps}}dvol_g\leq\big(\int_MG^{8n}dvol_g\big)^{\frac{1}{2}}\times\big(\int_Me^{8nF}dvol_g\big)^{\frac{1}{2}}\\
&\leq\big(\int_MC_{22}^{8n}(n+\Delta\varphi)^{8n(2n-1)}dvol_g\big)^{\frac{1}{2}}\times\big(\int_Me^{8nF}dvol_g\big)^{\frac{1}{2}}
\end{split}
\end{equation*}
By Corollary \ref{c2.2} and \ref{c2.4}, it's enough to assume that $p_0\geq 8n(2n-1)+1$.
With this choice, we know that $K_{\eps}$ and $L_{\eps}$ given by (\ref{2.57new}), (\ref{2.60new}) are bounded with the said dependence in the theorem. Then we can iterate (\ref{8.78}) as in cscK case to deduce $||u||_{L^{\infty}}$ is bounded in terms of $||u||_{L^1(\omega_0^n)}$.

To see that we have an estimate for $||u||_{L^1}$, we can compute
\begin{equation}
\Delta_{\varphi}(e^{\frac{1}{2}w})=\frac{1}{4}e^{\frac{1}{2}w}|\nabla_{\varphi}w|_{\varphi}^2+\frac{1}{2}e^{\frac{1}{2}w}\Delta_{\varphi}w.
\end{equation}
Hence
\begin{equation}
\begin{split}
\int_M&e^{\frac{1}{2}w}|\nabla_{\varphi}w|_{\varphi}^2dvol_g\leq C_{23}\int_Me^{\frac{1}{2}w}|\nabla_{\varphi}w|_{\varphi}^2dvol_{\varphi}\leq C_{23}\int_M2e^{\frac{1}{2}w}(-\Delta_{\varphi}w)dvol_{\varphi}\\
&\leq \int_MC_{24}(tr_{\varphi}\omega_0+1)dvol_{\varphi}=(n+1)C_{24}vol(M).
\end{split}
\end{equation}
\end{proof}
As an immediate consequence, we observe
\begin{cor}\label{c8.14}
Assume $\beta\geq0$ in (\ref{eq-?1}), (\ref{eq-?2}). 
Suppose $p_0\geq\kappa_n$, where $\kappa_n$ is as in Theorem \ref{t2.2}, then for any $p<p_0$, we have
$$
||\nabla(F+f)||_{L^{2p}(\omega_0^n)}\leq C_{25}.
$$
Here $C_{25}$ has the same dependence as in Theorem \ref{t2.1}, but additionally on $p$.
Besides, the bound is uniform in $p$ as long as $p$ is bounded away from $p_0$.
\end{cor}
\begin{proof}
We know from Theorem \ref{t2.2} that 
$|\nabla_{\varphi}(F+f)|_{\varphi}\leq C_{14}$.
On the other hand, we have
\begin{equation}
|\nabla(F+f)|^2\leq|\nabla_{\varphi}(F+f)|_{\varphi}^2(n+\Delta\varphi)\leq C_{14}^2(n+\Delta\varphi).
\end{equation}
Hence the result follows from Corollary \ref{c2.4}.
\end{proof}

Combining the estimates in this section, we can formulate the following theorem.
\begin{thm}\label{t2.3}
Assume $\beta\geq0$ in (\ref{eq-?1}), (\ref{eq-?2}). 
Let $\varphi$ be a smooth solution to (\ref{eq-?1}), (\ref{eq-?2}). Suppose $p_0\geq\kappa_n$ for some constant $\kappa_n$ depending only on $n$. Then for any $p<p_0$,
$$
||F+f||_{W^{1,2p}}\leq C_{25.1},\,\,\,||n+\Delta\varphi||_{L^p(\omega_0^n)}\leq C_{25.1}.
$$
Here $C_{25.1}$ depends only on an upper bound of entropy $\int_M\log\big(\frac{\omega_{\varphi}^n}{\omega_0^n}\big)\omega_{\varphi}^n$, $p_0>1$, $p<p_0$, the bound for $\int_Me^{-p_0f}dvol_g$, $||R||_0$, $\max_M|\beta_0|_g$ and background metric $\omega_0$. Besides, the bound is uniform in $p_0$ as long as $p_0$ is bounded away from 1 and $p$ bounded away from $p_0$.
\end{thm}
\section{Properness conjecture when $Aut_0(M,J)\neq0$}

 Define
$$
\mathcal{H}_{0}=\{\varphi\in C^{\infty}(M):\omega_{\varphi}:=\omega_0+\sqrt{-1}\partial\bar{\partial}\varphi\geq0,\,\,\,I(\varphi)=0\}.
$$
Here the functional $I$ is defined as
$$
I(\varphi)=\frac{1}{(n+1)!}\int_M\varphi\sum_{k=0}^n\omega_0^k\wedge\omega_{\varphi}^{n-k}.
$$
The set $\mathcal{H}_0$ can be identified as the set of K\"ahler metrics cohomologous to $\omega_0$.
We also know that for any $\varphi\in \mathcal{H}_{0}$, any $\sigma\in G$, one has $\sigma^*\omega_{\varphi}$ is still in the K\"ahler class $[\omega_0]$. 
Hence there exists a unique element $\psi\in\mathcal{H}_{0}$, such that $\sigma^*\omega_{\varphi}=\omega_{\psi}$.
We will write in short as $\sigma.\varphi=\psi$.
It is clear that this defines an action of $G$ on $\mathcal{H}_{0}$.

Let $d_1$ be the $L^1$ geodesic distance considered in the second paper, \cite{cc2}. Now we try to explain how to extend the notion of properness
to the general case.   For any given metric  $\omega_0$, we may consider its $G$ orbit
\[
{\mathcal O}_{\omega_0} = \{ \varphi \in {\mathcal H} \mid \sigma^* \omega_0 = \omega_\varphi, \textrm{ for some} \;\sigma \in G\}.
\]
Note that if $\omega_0$ is a cscK metric, then it is symmetric with respect to a maximal compact subgroup (\cite{calabi82}, \cite{calabi85}). Moreover,  one can check directly
that ${\mathcal O}_{\omega_0} \subset {\mathcal H}$  is a totally  geodesic submanifold (c.f. Proposition 2.1, \cite{CPZ}). Therefore, it is natural  to define a notion  of distance  to this submanifold ${\mathcal O}_{\omega_0} $ from  any  K\"ahler potential $\varphi$
 by 
\[
\begin{array}{lcl} d_p (\varphi, {\mathcal O}_{\omega_0} ) & =& \displaystyle \inf_{\psi \in {\mathcal O}_{\omega_0} }\; d_p (\varphi, \psi)\\
 &=&  \displaystyle \inf_{ \sigma \in G, \omega_\psi= \sigma^*\omega_0 }\; d_p (\varphi, \psi)\\
& =&  \displaystyle \inf_{ \sigma \in G, \omega_\psi =\sigma^*\omega_\varphi }\; d_p (0, \psi)
\end{array}
\]
More importantly, this infimum can be realized (Lemma 2.2, \cite{CPZ}), i.e., there exists a $\sigma_0 \in G$ such that
\[
d_p(\omega_\varphi, \sigma_0^* \omega_0) = d_p (\varphi, {\mathcal O}_{\omega_0} ).
\]
It means that this distance is positive unless $\varphi$ lies in this orbit.
 Motivated by this observation, we extend the properness definition to the general case, following \cite{Dar-Rub-17}.  First, as in \cite{Dar-Rub-17}, one can define

\begin{equation}
d_{1,G}(\varphi,\psi)=\displaystyle \inf_{\sigma_1,\sigma_2\in G}d_1(\sigma_1.\varphi,\sigma_2.\psi),\textrm{ for any $\varphi$, $\psi\in\mathcal{H}_{0}$.}
\end{equation}
The group $G$ acts on $\mathcal{H}_{0}$ by isometry, in the sense that
$$
d_1(\sigma.\varphi,\sigma.\psi)=d_1(\varphi,\psi),\textrm{ for any $\sigma\in G$, any $\varphi$, $\psi\in\mathcal{H}_{0}$.}
$$
As a result of this, we see that
\begin{equation}\label{3.2}
d_{1,G}(\varphi,\psi)=\inf_{\sigma\in G}d_1(\varphi,\sigma.\psi)=\inf_{\sigma\in G}d_1(\sigma.\varphi,\psi).
\end{equation}
Also it is immediate to check that $d_{1,G}$ satisfies triangle inequality: for any $\varphi_i\in\mathcal{H}_0$, $i=1,2,3$, we have
\begin{equation}
d_{1,G}(\varphi_1,\varphi_3)\leq d_{1,G}(\varphi_1,\varphi_2)+d_{1,G}(\varphi_2,\varphi_3).
\end{equation}
The cscK metrics in the class $[\omega_0]$ are critical points of the $K$-energy, which is implicitly defined by
\begin{equation}
\frac{d K(\varphi)}{dt}=\int_M\frac{\partial \varphi}{\partial t}(\underline{R}-R_{\varphi})\frac{\omega_{\varphi}^n}{n!}.
\end{equation}
In the above, $\underline{ R}$ is the average scalar curvature, $R_{\varphi}$ is the scalar curvature of the metric $\omega_{\varphi}$.
The $K$-energy has the following explicit formula:
\begin{equation}\label{3.4}
K(\varphi)=\int_M\log\bigg(\frac{\omega_{\varphi}^n}{\omega_0^n}\bigg)\frac{\omega_{\varphi}^n}{n!}+J_{-Ric}(\varphi),
\end{equation}
where for any $(1,1)$ form $\chi$, we define $J_{\chi}$ as
\begin{equation}
\begin{split}
J_{\chi}(\varphi)=\int_0^1&\int_M\varphi\bigg(\chi\wedge\frac{\omega_{\lambda\varphi}^{n-1}}{(n-1)!}-\underline{\chi}\frac{\omega_{\lambda\varphi}^n}{n!}\bigg)d\lambda\\
&=\frac{1}{n!}\int_M\varphi\sum_{k=0}^{n-1}\chi\wedge\omega_0^k\wedge\omega_{\varphi}^{n-1-k}-\frac{1}{(n+1)!}\int_M\underline{\chi}\varphi\sum_{k=0}^n\omega_0^k\wedge\omega_{\varphi}^{n-k}.
\end{split}
\end{equation}

\begin{equation}\label{3.5n}
\frac{d J_{\chi}(\varphi)}{dt}=\int_M\frac{\partial\varphi}{\partial t}(tr_{\varphi}\chi-\underline{\chi})\frac{\omega_{\varphi}^n}{n!}.
\end{equation}
The readers may look up section 2 of \cite{cc2} for more details.
First we make precise the notion of properness of $K$-energy with respect to $d_{1,G}$, in a similar vein as properness with respect to $d_1$ introduced in the second paper.
\begin{defn}\label{d3.1}
We say $K$-energy is proper with respect to $d_{1,G}$, if 
\begin{enumerate}
\item for any sequence $\{\varphi_i\}\subset\mathcal{H}_0$, $d_{1,G}(0,\varphi_i)\rightarrow\infty$ implies $K(\varphi_i)\rightarrow+\infty$.
\item $K$-energy is bounded from below on $\mathcal{H}$.
\end{enumerate}
\end{defn}
In this section, we will prove the following result:
\begin{thm}\label{t3.1}
Suppose that $K$-energy functional is proper with respect to  $d_{1,G}$ as defined in (\ref{d3.1}), then the class $[\omega_0]$ admits a cscK metric.
\end{thm}

\begin{rem}
The converse direction has been established by  \cite{Darvas1605} and \cite{Dar-Rub-17}.
\end{rem}
As a preliminary step, we observe that the assumption $K$-energy being bounded from below implies it is invariant under the action of $G$.
\begin{lem}\label{l3.3new}
Suppose that the $K$-energy is bounded from below, then the $K$-energy is invariant under the action of $G$, i.e. $K(\sigma.\varphi)=K(\varphi)$ for any $\varphi\in\mathcal{H}$ and $\sigma\in G$.
\end{lem}
\begin{proof}
We will prove this by showing the Calabi-Futaki invariant vanishes. Let $\sigma\in G$, then there exists a holomorphic vector field $X$ which generates a one-parameter path $\{\sigma(t)\}_{t\in\mathbb{R}}$, with $\sigma(0)=id$ and $\sigma(1)=\sigma$.

From the definition of $K$-energy and Calabi-Futaki invariant, we know that
$$
\frac{d}{dt}\big(K(\sigma(t)^*\omega_{\varphi})\big)=Re\big(\mathcal{F}(X,[\omega_0])\big)=a.
$$
Here $a$ is a constant depending only on the holomorphic vector field $X$ and cohomology class of $[\omega_0]$.
Since $K$-energy is bounded from below on the holomorphic line $\{\sigma(t)^*\omega_{\varphi}\}_{t\in\mathbb{R}}$, we must have $a=0$. This implies that $K(\sigma.\varphi)=K(\varphi)$.
\end{proof}


Theorem \ref{t3.1} will be proved by solving the following path of continuity:
\begin{equation}\label{3.5}
t(R_{\varphi}-\underline{R})=(1-t)(tr_{\varphi}\omega_0-n),\textrm{ $t\in[0,1]$.}
\end{equation}
Let $\varphi$ solves (\ref{3.5}), then we call $\omega_{\varphi}$ to be twisted cscK metric.
 For $t>0$, equation (\ref{3.5}) can be equivalently put as:
\begin{align}
\label{eq-1}
&\det(g_{i\bar{j}}+\varphi_{i\bar{j}})=e^F\det g_{i\bar{j}},\\
\label{eq-2}
&\Delta_{\varphi}F=-\big(\underline{R}-\frac{1-t}{t}n\big)+tr_{\varphi}\big(Ric(\omega_0)-\frac{1-t}{t}\omega_0\big).
\end{align}
One important fact about this continuity path is that the set of solvable $t$ is open, more precisely, 
\begin{lem} (\cite{chen15}, \cite{zeng}, \cite{Hashi})
Suppose for some $0\leq t_0<1$, (\ref{3.5}) has a solution $\varphi\in C^{4,\alpha}(M)$ with $t=t_0$, then for some $\delta>0$, (\ref{3.5}) has a solution in $C^{4,\alpha}(M)$ for any $t\in(t_0-\delta,t_0+\delta)\cap[0,1]$.
\end{lem}
\begin{rem}
One can see by bootstrap that the solution $\varphi$ of (\ref{3.5})(or equivalently of (\ref{eq-1}), (\ref{eq-2}) for $t>0$) is smooth if we know it's in $C^{4,\alpha}$.
\end{rem}
Another important fact about twisted path is that solutions to (\ref{3.5}) are minimizers of the twisted $K$-energy, defined as \begin{equation}\label{3.9}
K_{\omega_0,t}=tK+(1-t)J_{\omega_0}=t\int_M\log\bigg(\frac{\omega_{\varphi}^n}{\omega_0^n}\bigg)\frac{\omega_{\varphi}^n}{n!}+J_{-tRic+(1-t)\omega_0},\,\,\,t\in[0,1].
\end{equation}
First we observe that if the $K$-energy satisfies the assumptions of Definition \ref{d3.1}, the twisted path (\ref{3.5}) is solvable for any $0\leq t<1$. Indeed, we have
\begin{lem}\label{l3.5}
Suppose the $K$-energy is bounded from below, then (\ref{3.5}) is solvable for $0\leq t<1$.
\end{lem}
This will be proved as a result of the following theorem, proved in the second paper:
\begin{thm}\label{t3.2}
(\cite{cc2}, Theorem 4.1)Fix $0<t_0\leq 1$, suppose the twisted $K$-energy defined in (\ref{3.9})(with $t=t_0$) is proper with respect to $d_1$, then there exists a smooth solution to (\ref{3.5}).
\end{thm}
Now we are ready to prove Lemma \ref{l3.5}.
\begin{proof}
(of Lemma \ref{l3.5})
In view of Theorem \ref{t3.2}, we just need to verify for $0<t_0<1$, $K_{\omega_0,t_0}$ is proper with respect to $d_1$.
More specifically, since we know $K$-energy is bounded from below, we just need to observe $J_{\omega_0}$ is proper with respect to $d_1$.

To see that $J_{\omega_0}$ is proper, this follows from Proposition 22 in \cite{CoSz}, which says  that for some $\delta>0$ and some $C>0$, one has
$$
J_{\omega_0}(\varphi)\geq\delta J(\varphi)-C,\textrm{ for any $\varphi\in\mathcal{H}_0$.}
$$
Here $J$ is Aubin's $J$-functional, defined as
$$
J(\varphi)=\int_M\varphi(\omega_0^n-\omega_{\varphi}^n).
$$
It is elementary to show that $J(\varphi)\geq \frac{1}{C'}d_1(0,\varphi)-C'$ for $\varphi\in\mathcal{H}_0$(c.f. \cite{Dar-Rub-17}, Proposition 5.5). Hence we see that $J_{\omega_0}$ is proper with respect to $d_1$.
\end{proof}
Hence to get existence of cscK, the only remaining issue is to understand what happens  as $t\rightarrow1$.
We will handle this difficulty now. Throughout the rest of this section, we assume the $K$-energy is proper with respect to $d_{1,G}$, in the sense defined by Definition \ref{d3.1}.

Let $t_i<1$, and $t_i$ monotonically increase to 1.
Denote $\tilde{\varphi}_i\in\mathcal{H}_0$ to be solutions to (\ref{3.5}) with $t=t_i$. They exist due to Lemma \ref{l3.5}.
First we show that for the sequence $\tilde{\varphi}_i$, the $K$-energy is uniformly bounded from above.
\begin{lem}\label{l3.6}
Let $\tilde{\varphi}_i$ be as in previous paragraph, then we have
\begin{equation}\label{3.10n}
K_{\omega_0,t_i}(\tilde{\varphi}_i)=\inf_{\mathcal{H}}K_{\omega_0,t_i}(\varphi)\rightarrow\inf_{\mathcal{H}}K(\varphi),\textrm{ as $t_i\rightarrow1$.}
\end{equation}
Also
\begin{equation}\label{3.11n}
K(\tilde{\varphi}_i)\rightarrow\inf_{\mathcal{H}}K(\varphi), \textrm{ as $t_i\rightarrow1$.}
\end{equation}
\end{lem}
\begin{proof}
That $K_{\omega_0,t_i}(\tilde{\varphi}_i)=\inf_{\mathcal{H}}K_{\omega_0,t_i}(\varphi)$ follows from the convexity of the twisted $K$-energy and has been proved in Corollary 4.5 of our second paper, \cite{cc2}.
By the second part of Definition \ref{d3.1}, we know that $\inf_{\mathcal{H}}K(\varphi)>-\infty$.
On the other hand, let $\varphi^{\eps}\in\mathcal{H}$ be such that $K(\varphi^{\eps})\leq \inf_{\mathcal{H}}K(\varphi)+\eps$, and we know that 
\begin{equation}\label{8.6}
\lim\sup_{i\rightarrow\infty}K_{\omega_0,t_i}(\tilde{\varphi}_i)\leq\lim\sup_{i\rightarrow\infty}K_{\omega_0,t_i}(\varphi^{\eps})= K(\varphi^{\eps})\leq\inf_{\mathcal{H}}K(\varphi)+\eps.
\end{equation}
On the other hand, we also know that
\begin{equation}
K_{\omega_0,t_i}(\tilde{\varphi}_i)=t_iK(\tilde{\varphi}_i)+(1-t_i)J_{\omega_0}(\tilde{\varphi}_i)\geq t_i\inf_{\mathcal{H}}K(\varphi)+(1-t_i)J_{\omega_0}(0).
\end{equation}
In the last inequality above, we used the fact that $0$ is the solution to $tr_{\varphi}\omega_0=n$, therefore a minimizer of $J_{\omega_0}$. 
Hence we have 
\begin{equation}\label{8.8}
\lim\inf_{t_i\rightarrow1}K_{\chi,t_i}(\tilde{\varphi}_i)\geq\inf_{\mathcal{H}}K(\varphi).
\end{equation}
From (\ref{8.6}) and (\ref{8.8}), (\ref{3.10n}) follows.
To see (\ref{3.11n}), we observe for $t_i$ sufficiently close to 1, we have
$$
\inf_{\mathcal{H}}K(\varphi)+\eps\geq t_iK(\tilde{\varphi}_i)+(1-t_i)J_{\omega_0}(\tilde{\varphi}_i)\geq t_iK(\tilde{\varphi}_i)+(1-t_i)J_{\omega_0}(0).
$$
The first inequality follows from (\ref{3.10n}). Hence we have
\begin{equation}
\lim\sup_{t_i\rightarrow1}K(\tilde{\varphi}_i)\leq\lim_{t_i\rightarrow1}\big(\frac{1}{t_i}(\inf_{\mathcal{H}}K(\varphi)+\eps)-\frac{1-t_i}{t_i}J_{\omega_0}(0)\big)\leq \inf_{\mathcal{H}}K(\varphi)+\eps.
\end{equation}
From this (\ref{3.11n}) follows.
\end{proof}
As an immediate consequence of Lemma \ref{l3.6} and the properness assumption of $K$-energy, we deduce
\begin{cor}\label{c3.8new}
Let $\tilde{\varphi}_i$ be as in previous lemma, we have
$$
\sup_id_{1,G}(0,\tilde{\varphi}_i)<\infty.
$$
\end{cor}
The following proposition is the key technical result from which Theorem \ref{t3.1} immediately follows.
\begin{prop}\label{l3.17}
Consider the continuity path (\ref{3.5}).
Suppose for some sequence $t_i\nearrow 1$, there exists a solution $\tilde{\varphi}_i$ to (\ref{3.5}) with $t=t_i$ with $\tilde{\varphi}_i\in\mathcal{H}_0^1$ and $\sup_id_{1,G}(0,\tilde{\varphi}_i)<\infty$.
Let $\varphi_i\in\mathcal{H}_0^1$ be in the same $G$-orbit as $\tilde{\varphi}_i$ such that $\sup_id_1(0,\varphi_i)<\infty$. 
Suppose also that $K$-energy is $G$-invariant, then $\{\varphi_i\}_i$ contains a subsequence which converges in $C^{1,\alpha}$ (for any $0<\alpha<1$) to a smooth cscK potential.
\end{prop}
Let $\sigma_i\in G$ be such that 
\begin{equation}
\sup_id_1(0,\sigma_i.\tilde{\varphi}_i)<\infty.
\end{equation}
The existence of such a sequence $\sigma_i$ follows from Corollary \ref{c3.8new}. 
Denote $\varphi_i=\sigma_i.\tilde{\varphi}_i$. Next we briefly explain how to obtain above proposition.

First we write down the equation satisfied by the sequence $\varphi_i$, and they turn out to satisfy an equation in the form studied in section 2, as shown by Lemma \ref{l3.10nn} below.
Moreover, the integrability exponent $p_0$ improves to infinity as $t_i$ approaches $1$. Hence the estimates in section 2 allow us to get uniform bounds of $\varphi_i$ in $W^{2,p}$ for any $p<\infty$.
Hence we can use compactness to take limit and we show the limit solves a weak form of cscK equation, as shown in Proposition \ref{p3.12} below.
Finally one argues that this weak solution of cscK equation is actually smooth.

As a preliminary step, we show the sequence $\{\varphi_i\}$ has uniformly bounded entropy.
\begin{lem}
Denote $\varphi_i=\sigma_i.\tilde{\varphi}_i$, then we have
$$
\sup_i\int_M\log\bigg(\frac{\omega_{\varphi_i}^n}{\omega_0^n}\bigg)\omega_{\varphi_i}^n<\infty.
$$
\end{lem}
\begin{proof}
First due to the $G$-invariance of $K$-energy observed in Lemma \ref{l3.3new}, we have
\begin{equation}\label{3.17}
\sup_iK(\varphi_i)=\sup_iK(\tilde{\varphi}_i)<\infty.
\end{equation}
On the other hand, we know from Lemma 4,4 of paper 2 that 
\begin{equation}\label{3.18}
\sup_i|J_{-Ric}(\varphi_i)|\leq\sup_iC_n|Ric|_{\omega_0}d_1(0,\varphi_i)<\infty.
\end{equation}
From (\ref{3.17}), (\ref{3.18}), and recall the formula for $K$-energy in (\ref{3.4}), the desired conclusion follows.
\end{proof}
Next we derive the equation satisfied by the sequence $\varphi_i$.
We have the following result:
\begin{lem}\label{l3.10nn}
Let $\theta_i$ be such that $\sigma_i^*\omega_0=\omega_{\theta_i}$, with $\sup_M\theta_i=0$.
Then $\varphi_i$ satisfies the following equations:
\begin{align}
\label{eq-n-1}
&\det(g_{\alpha\bar{\beta}}+(\varphi_i)_{\alpha\bar{\beta}})=e^{F_i}\det g_{\alpha\bar{\beta}},\\
\label{eq-n-2}
&\Delta_{\varphi_i}F_i=-\big(\underline{R}-\frac{1-t_i}{t_i}n\big)+tr_{\varphi_i}\big(Ric(\omega_0)-\frac{1-t_i}{t_i}\omega_{\theta_i}\big).
\end{align}
\end{lem}
\begin{proof}
Define $e^{\tilde{F}_i}=\frac{\omega_{\tilde{\varphi}_i}^n}{\omega_0^n}$. We have the following calculations:
\begin{equation}
\sigma_i^*(\omega_{\tilde{\varphi}_i}^n)=(\sigma_i^*\omega_{\tilde{\varphi}_i})^n=\omega_{\varphi_i}^n.
\end{equation}
On the other hand,
\begin{equation}
\sigma_i^*(e^{\tilde{F}_i}\omega_0^n)=e^{\tilde{F}_i\circ\sigma_i}(\sigma_i^*\omega_0)^n.
\end{equation}
So
\begin{equation}
\frac{\omega_{\varphi_i}^n}{(\sigma_i^*\omega_0)^n}=e^{\tilde{F}_i\circ\sigma_i}.
\end{equation}
Hence if we define $F_i$ to be $e^{F_i}=\frac{\omega_{\varphi_i}^n}{\omega_0^n}$, so as to make sure (\ref{eq-n-1}) always holds, we have
\begin{equation}\label{3.24}
F_i=\tilde{F}_i\circ\sigma_i+\log\bigg(\frac{(\sigma_i^*\omega_0)^n}{\omega_0^n}\bigg).
\end{equation}
To see (\ref{eq-n-2}), we go back to (\ref{eq-2}), and note that (\ref{eq-2}) is equivalent to:
\begin{equation}
\sqrt{-1}\partial\bar{\partial}\tilde{F}_i\wedge\frac{\omega_{\tilde{\varphi}_i}^{n-1}}{(n-1)!}=-\big(\underline{R}-\frac{1-t_i}{t_i}n\big)\frac{\omega_{\tilde{\varphi}_i}^n}{n!}+\big(Ric(\omega_0)-\frac{1-t_i}{t_i}\omega_0\big)\wedge\frac{\omega_{\tilde{\varphi}_i}^{n-1}}{(n-1)!}.
\end{equation}
Pulling back using $\sigma_i$, we obain
\begin{equation}
\begin{split}
\sqrt{-1}\partial\bar{\partial}(\tilde{F}_i\circ\sigma_i)\wedge\frac{\omega_{\varphi_i}^{n-1}}{(n-1)!}&=-\big(\underline{R}-\frac{1-t_i}{t_i}n\big)\frac{\omega_{\varphi_i}^n}{n!}\\
&+\big(Ric(\sigma_i^*\omega_0)-\frac{1-t_i}{t_i}\sigma_i^*\omega_0\big)\wedge\frac{\omega_{\varphi_i}^{n-1}}{(n-1)!}.
\end{split}
\end{equation}
Using (\ref{3.24}) and recall that 
$$
\sqrt{-1}\partial\bar{\partial}\log\bigg(\frac{(\sigma_i^*\omega_0)^n}{\omega_0^n}\bigg)=Ric(\omega_0)-Ric(\sigma_i^*\omega_0), 
$$
we conclude
\begin{equation}
\begin{split}
\big(\sqrt{-1}\partial\bar{\partial}F_i+Ric&(\sigma_i^*\omega_0)-Ric(\omega_0)\big)\wedge\frac{\omega_{\varphi_i}^{n-1}}{(n-1)!}=-\big(\underline{R}-\frac{1-t_i}{t_i}n\big)\frac{\omega_{\varphi_i}^n}{n!}\\
&+\big(Ric(\sigma_i^*\omega_0)-\frac{1-t_i}{t_i}\sigma_i^*\omega_0\big)\wedge\frac{\omega_{\varphi_i}^{n-1}}{(n-1)!}.
\end{split}
\end{equation}
This is equivalent to (\ref{eq-n-2}).
\end{proof}
Next we would like to use the result obtained in the last section to study the regularity of $\varphi_i$.
Denote $R_i=\underline{R}-\frac{1-t_i}{t_i}n$, $\beta_i=\frac{1-t_i}{t_i}\omega_{\theta_i}$, $(\beta_0)_i=\frac{1-t_i}{t_i}\omega_0$, and $f_i=\frac{1-t_i}{t_i}\theta_i$.
Then we have $\beta_i\geq0$, and $\beta_i=(\beta_0)_i+\sqrt{-1}\partial\bar{\partial}f_i$.
Here we prove a property about the $f_i$ which will be crucial for our proof.
\begin{lem}\label{l3.10}
There exists a constant $C_{26}$, which depends only on the background metric $\omega_0$, such that for any $p>1$, there exists $\eps_p>0$, depending only on $p$ and the background metric $\omega_0$, such that for any $t_i\in(1-\eps_p,1)$, one has $e^{-f_i}\in L^p(\omega_0^n)$ with $||e^{-f_i}||_{L^p(\omega_0^n)}\leq C_{26}$.
\end{lem}
\begin{proof}
Since we know that $\omega_{\theta_i}=\omega_0+\sqrt{-1}\partial\bar{\partial}\theta_i\geq0$, with $\sup_M\theta_i=0$, hence by a result of Tian(c.f. \cite{tian87}, Proposition 2.1), we know that there exists $\alpha>0$, $C_{25.5}>0$, depending only on the background metric $\omega_0$, such that for any $u\in C^2(M)$, $\omega_0+\sqrt{-1}\partial\bar{\partial}u\geq0$, one has $\int_Me^{-\alpha(u-\sup_Mu)}dvol_g\leq C_{25.5}$.

Given $p>1$, suppose $t_i$ is sufficiently close to 1 such that $p\frac{1-t_i}{t_i}<\alpha$, then we have
\begin{equation}
\begin{split}
\int_Me^{-pf_i}dvol_g=\int_M&e^{-p\frac{1-t_i}{t_i}\theta_i}dvol_g\leq\big(\int_Me^{-\alpha \theta_i}dvol_g\big)^{p\frac{1-t_i}{\alpha t_i}}vol(M)^{1-\frac{p(1-t_i)}{\alpha t_i}}\\
&\leq C_{25.5}^{p\frac{1-t_i}{\alpha t_i}}vol(M)^{1-\frac{p(1-t_i)}{\alpha t_i}}\leq \max\big(C_{25.5},vol(M)\big):=C_{26}.
\end{split}
\end{equation}
\end{proof}
As an application of the estimate in Theorem \ref{t2.3}, we conclude the following uniform estimate for the sequence $\varphi_i$.
\begin{prop}
For any $p>1$, there exists a constant $C_{27}$, and $\eps'_p>0$, such that for any $t_i\in(1-\eps'_p,1)$, 
$$
||F_i+f_i||_{W^{1,2p}}\leq C_{27},\,\,\,||n+\Delta\varphi_i||_{L^p(\omega_0^n)}\leq C_{27}.
$$
In the above, $\eps_p'$ depends only on $p$ and background metric $\omega_0$, and $C_{27}$ depends on $p$, background metric $\omega_0$ and the uniform entropy bound $\sup_i\int_M\log\big(\frac{\omega_{\varphi_i}^n}{\omega_0^n}\big)\omega_{\varphi_i}^n$.
\end{prop}
\begin{proof}
We may assume that $\eps_p'$ is chosen so small such that for any $t_i\in(1-\eps_p',1)$, $e^{-f_i}\in L^q$ for some $q\geq\kappa_n$.
Such smallness depends only on $n$ and the $\alpha$-invariant of the background metric. The result then follows from Lemma \ref{l3.10} and Theorem \ref{t2.3}.
\end{proof}
With this preparation, we can pass to the limit.
Hence we may take a subsequence of $\varphi_i$(without relabeled), and a function $\varphi_*\in W^{2,p}$ for any $p<\infty$, and another function $F_*\in W^{1,p}$ for any $p<\infty$, such that 
\begin{align}
\label{conv-1}
&\varphi_i\rightarrow\varphi_*\textrm{ in $C^{1,\alpha}$ for any $0<\alpha<1$ and $\sqrt{-1}\partial\bar{\partial}\varphi_i\rightarrow\sqrt{-1}\partial\bar{\partial}\varphi_*$ weakly in $L^p$.}\\
\label{conv-2}
&F_i+f_i\rightarrow F_*\textrm{ in $C^{\alpha}$ for any $0<\alpha<1$ and $\nabla(F_i+f_i)\rightarrow\nabla F_*$ weakly in $L^p$.}
\end{align}
As a result of (\ref{conv-1}), we have
\begin{equation}\label{3.31}
\omega_{\varphi_i}^k\rightarrow\omega_{\varphi_*}^k,\textrm{ weakly in $L^p$ for any $1\leq k\leq n$ and $p<\infty$.}
\end{equation}
Here we provide an argument(more or less standard) for this weak convergence.
\begin{lem}
Suppose the convergence in (\ref{conv-1}) holds. Then for any $p<\infty$ and any $1\leq k\leq n$,
$$
\omega_{\varphi_i}^k\rightarrow\omega_{\varphi_*}^k\textrm{ weakly in $L^p$.}
$$
\end{lem}
\begin{proof}
We need to show that, for any $\zeta$, a smooth $(n-k,n-k)$ form, the  following convergence holds:
\begin{equation}\label{3.33n}
\int_M\omega_{\varphi_i}^k\wedge\zeta\rightarrow\int_M\omega_{\varphi_*}^k\wedge\zeta,\textrm{ as $i\rightarrow\infty$.}
\end{equation}
Since $\omega_{\varphi_i}^k$ is uniformly bounded in $L^p$ for any $p<\infty$, (\ref{3.33n}) will imply the same convergence holds for any $\zeta\in L^q$ with $q>1$.
Now we prove (\ref{3.33n}) by induction in $k$.

First observe that when $k=1$, (\ref{3.33n}) follows from the weak convergence of $\sqrt{-1}\partial\bar{\partial}\varphi_i$.

Now assume (\ref{3.33n}) holds for $k=l-1$, we need to show (\ref{3.33n}) holds for $k=l$. Indeed, let $\zeta$ be a smooth $(n-l,n-l)$ form, we have
\begin{equation}\label{3.34n}
\begin{split}
\int_M\omega_{\varphi_i}^l\wedge\zeta&=\int_M\omega_{\varphi_i}^{l-1}\wedge\omega_0\wedge\zeta+\int_M\omega_{\varphi_i}^{l-1}\wedge\sqrt{-1}\partial\bar{\partial}\varphi_i\wedge\zeta\\
&=\int_M\omega_{\varphi_i}^{l-1}\wedge\omega_0\wedge\zeta-\int_M\omega_{\varphi_i}^{l-1}\wedge d^c\varphi_i\wedge d\zeta.
\end{split}
\end{equation}
Here $d^c=\frac{\sqrt{-1}}{2}(\partial-\bar{\partial})$.
From the induction hypothesis, we know that 
\begin{equation}\label{3.35n}
\int_M\omega_{\varphi_i}^{l-1}\wedge\omega_0\wedge\zeta\rightarrow\int_M\omega_{\varphi_*}^{l-1}\wedge\omega_0\wedge\zeta,\textrm{ as $i\rightarrow\infty$.}
\end{equation}
On the other hand, we know from (\ref{conv-1}) that $d^c\varphi_i\rightarrow d^c\varphi_*$ uniformly, hence $d^c\varphi_i\wedge d\zeta\rightarrow d^c\varphi_*\wedge d\zeta$ strongly in $L^q$ for any $q>1$.
This combined with the weak convergence of $\omega_{\varphi_i}^{l-1}$ is sufficient to imply
\begin{equation}\label{3.36n}
\int_M\omega_{\varphi_i}^{l-1}\wedge d^c\varphi_i\wedge d\zeta\rightarrow\int_M\omega_{\varphi_*}^{l-1}\wedge d^c\varphi_*\wedge d\zeta,\textrm{ as $i\rightarrow\infty$.}
\end{equation}
Combining (\ref{3.34n}), (\ref{3.35n}) and (\ref{3.36n}), we conclude as $i\rightarrow\infty$,
\begin{equation}
\int_M\omega_{\varphi_i}^l\wedge\zeta\rightarrow\int_M\omega_{\varphi_*}^{l-1}\wedge\omega_0\wedge\zeta-\int_M\omega_{\varphi_*}^{l-1}\wedge d^c\varphi_*\wedge d\zeta=\int_M\omega_{\varphi_*}^l\wedge\zeta.
\end{equation}
This proves (\ref{3.33n}) for $k=l$ and finishes the induction.
\end{proof}

It is crucial matter to identify the limit. Actually we will show the solution $\varphi_*$ is a weak solution to cscK in the following sense:
\begin{prop}\label{p3.12}
Let $\varphi_*$, $F_*$ be the limit obtained in (\ref{conv-1}), (\ref{conv-2}).
Then $\varphi_*$ is a weak solution to cscK in the following sense:
\begin{enumerate}
\item $\omega_{\varphi_*}^n=e^{F_*}\omega_0^n$,

\item For any $\eta\in C^{\infty}(M)$, we have
\begin{equation}\label{3.32}
-\int_Md^cF_*\wedge d\eta\wedge\frac{\omega_{\varphi_*}^{n-1}}{(n-1)!}=\int_M-\eta\underline{R}\frac{\omega_{\varphi_*}^n}{n!}+\eta Ric\wedge\frac{\omega_{\varphi_*}^{n-1}}{(n-1)!}.
\end{equation}
\end{enumerate}
In the above, $d^c=\frac{\sqrt{-1}}{2}(\partial-\bar{\partial})$.
\end{prop}
Before we prove this proposition, we need the following lemma, which shows $f_i\rightarrow0$ in $L^1$. This is needed to justify (1) in the above proposition.
\begin{lem}\label{l3.13}
Recall $\theta_i$ is defined as $\sigma_i^*\omega_0=\omega_{\theta_i}$ with $\sup_M\theta_i=0$.
$f_i=\frac{1-t_i}{t_i}\theta_i$. Then we have
$$
e^{-f_i}\rightarrow 1\textrm{ in $L^p(\omega_0^n)$ as $t_i\rightarrow1$ for any $p<\infty$.}
$$
\end{lem}
\begin{proof}
First we know from (\ref{3.10n}) that there exists $\eps_i\rightarrow0$, such that 
\begin{equation}\label{3.34nn}
\begin{split}
\inf_{\mathcal{H}}K(\varphi)+\eps_i\geq K_{\omega_0,t_i}(\tilde{\varphi}_i)&=t_iK(\tilde{\varphi}_i)+(1-t_i)J_{\omega_0}(\tilde{\varphi}_i)\\
&\geq t_i\inf_{\mathcal{H}}K(\varphi)+(1-t_i)\delta d_1(0,\tilde{\varphi}_i)-(1-t_i)C.
\end{split}
\end{equation}
This implies $(1-t_i)d_1(0,\tilde{\varphi}_i)\rightarrow0$ as $t_i\rightarrow1$. On the other hand, denote $\tilde{\theta}_i=\theta_i-\frac{I(\theta_i)}{vol(M)}$, then we have $\tilde{\theta}_i\in\mathcal{H}_0$ and $\sigma_i.0=\tilde{\theta}_i$.
Also we know that $G$ acts on $\mathcal{H}_0$ by isometry, hence
\begin{equation}
d_1(0,\tilde{\theta}_i)-d_1(0,\varphi_i)\leq d_1(\tilde{\theta}_i,\varphi_i)=d_1(\sigma_i.0,\sigma_i.\tilde{\varphi}_i)=d_1(0,\tilde{\varphi}_i).
\end{equation}
Since $\sup_id_1(0,\varphi_i)<\infty$, we know $(1-t_i)d_1(0,\tilde{\theta}_i)\rightarrow0$. 
Therefore from \cite{Darvas1402}, Theorem 5.5, we see that as $t_i\rightarrow1$,
$$
(1-t_i)\int_M|\tilde{\theta}_i|\omega_0^n\leq(1-t_i)d_1(0,\tilde{\theta}_i)\rightarrow0.
$$
Now we claim that 
\begin{equation}\label{3.34}
I(\theta_i)(1-t_i)\rightarrow0,\textrm{ as $t_i\rightarrow1$.}
\end{equation}
If we have shown this claim, then we will have $\int_M|f_i|\omega_0^n\rightarrow0$. Since we already know for any $1<p'<\infty$, we have $\sup_i\int_Me^{-p'f_i}\omega_0^n<\infty$, the claimed result then follows(by taking $p'>p$).

Hence it only remains to show the claim. Since we know that $\sup_M\theta_i=0$, we know that 
$$
0\leq\int_M(-\theta_i)\omega_0^n\leq C_{28},\textrm{ $C_{28}$ depends only on background metric $\omega_0$.}
$$
On the other hand, 
\begin{equation}
\begin{split}
&I(\theta_i)+\int_M(-\theta_i)\frac{\omega_0^n}{n!}=\frac{1}{(n+1)!}\int_M\theta_i\sum_{k=0}^n\big(\omega_0^k\wedge\omega_{\theta_i}^{n-k}-\omega_0^n\big)\\
&=\frac{1}{(n+1)!}\int_M\theta_i\sqrt{-1}\partial\bar{\partial}\theta_i\wedge\sum_{k=0}^{n-1}(n-k)\omega_{\theta_i}^k\wedge\omega_0^{n-k-1}\\
&\geq -\frac{n}{(n+1)!}\int_M\sqrt{-1}\partial \theta_i\wedge\bar{\partial}\theta_i\wedge\sum_{k=0}^{n-1}\omega_0^k\wedge\omega_{\theta_i}^{n-1-k}\\
&=-\frac{n}{(n+1)!}\int_M\tilde{\theta}_i(\omega_0^n-\omega_{\tilde{\theta}_i}^n)\geq-Cd_1(0,\tilde{\theta}_i).
\end{split}
\end{equation}
Hence we have 
$$
0\geq I(\theta_i)\geq-C'(1+d_1(0,\tilde{\theta}_i)).
$$
From here the claim (\ref{3.34}) immediately follows.
\end{proof}
Now we are ready to show Proposition \ref{p3.12}.
We will obtain this as the result of the previous lemma
\begin{proof}
(of Proposition \ref{p3.12}) 

First we show the equation (1) holds. First for each fixed $i$, we have $\omega_{\varphi_i}^n=e^{F_i}\omega_0^n$.
(\ref{3.31}) shows $\omega_{\varphi_i}^n\rightarrow\omega_{\varphi_*}^n$ weakly in $L^p$ for any $p<\infty$. For the convergence of the right hand side, we can write $
e^{F_i}=e^{F_i+f_i}\cdot e^{-f_i}.
$
According to (\ref{conv-2}), we see that $F_i+f_i$ is uniformly bounded, and converges to $F_*$ strongly in $L^p$ for $p<\infty$.
This implies $e^{F_i+f_i}\rightarrow e^{F_*}$ in $L^p$ for any finite $p$.
On the other hand, we have just shown in Lemma \ref{l3.13} that $e^{-f_i}\rightarrow 1$ in $L^p$ for any $p<\infty$.
From here we can conclude $e^{F_i}\rightarrow e^{F_*}$ in $L^p$ for $p<\infty$.
Hence the equation (1) of Proposition follows.

To see the second equation, first we see from (\ref{eq-n-2}) that 
$$
\Delta_{\varphi_i}(F_i+f_i)=-\big(\underline{R}-\frac{1-t_i}{t_i}n\big)+tr_{\varphi_i}\big(Ric-\frac{1-t_i}{t_i}\omega_0\big).
$$
This implies for $\eta\in C^{\infty}(M)$, one has
\begin{equation}\label{3.36}
\begin{split}
\int_M&(F_i+f_i)d^cd\eta\wedge\frac{\omega_{\varphi_i}^{n-1}}{(n-1)!}\\
&=\int_M-\eta\big(\underline{R}-\frac{1-t_i}{t_i}n\big)\frac{\omega_{\varphi_i}^n}{n!}+\eta\big(Ric-\frac{1-t_i}{t_i}\omega_0\big)\wedge\frac{\omega_{\varphi_i}^{n-1}}{(n-1)!}.
\end{split}
\end{equation}
We wish to pass to limit in (\ref{3.36}) as $t_i\rightarrow1$. First because of (\ref{3.31}), we can easily conclude:
\begin{equation}
\textrm{R.H.S of (\ref{3.36})}\rightarrow\int_M\eta\bigg(-\underline{R}\frac{\omega_{\varphi_*}^n}{n!}+Ric\wedge\frac{\omega_{\varphi_*}^{n-1}}{(n-1)!}\bigg).
\end{equation}
For the left hand side, since $F_i+f_i\rightarrow F_*$ strongly in $L^p$, $\omega_{\varphi_i}^{n-1}\rightarrow\omega_{\varphi_*}^{n-1}$ weakly in $L^p$ for any $p<\infty$, we can conclude
$$
\int_M(F_i+f_i)d^cd\eta\wedge\frac{\omega_{\varphi_i}^{n-1}}{(n-1)!}\rightarrow\int_MF_* d^cd\eta\wedge\frac{\omega_{\varphi_*}^{n-1}}{(n-1)!}.
$$
Since $F_*\in W^{1,p}$, we have
\begin{equation}
-\int_Md^cF_*\wedge d\eta\wedge\frac{\omega_{\varphi_*}^{n-1}}{(n-1)!}=\int_MF_*d^cd\eta\wedge\frac{\omega_{\varphi_*}^{n-1}}{(n-1)!}=\int_M\eta\bigg(-\underline{R}\frac{\omega_{\varphi_*}^n}{n!}+Ric\wedge\frac{\omega_{\varphi_*}^{n-1}}{(n-1)!}\bigg).
\end{equation}
\end{proof}

Next we argue that $\omega_{\varphi_*}$ is quasi-isometric to $\omega_0$.
\begin{lem}\label{l3.14}
There exists a constant $C_{29}$, such that $\frac{1}{C_{29}}\omega_0\leq\omega_{\varphi_*}\leq C_{29}\omega_0.$
\end{lem}
\begin{proof}
We know that $F_*\in W^{1,p}$ for any $p<\infty$, hence we may take $G_k\in C^{\infty}(M)$, uniformly bounded, and $G_k\rightarrow F_*$ in $W^{1,p}$.
Let $\psi_k$ be the solution to $\omega_{\psi_k}^n=e^{G_k}\omega_0^n$ with $\sup_M\psi_k=0$.
The result of \cite{chenhe12}, Theorem 1.1 shows that for any $p<\infty$, one has
$$
\sup_k||\psi_k||_{W^{3,p}}<\infty.
$$
Hence up to a subsequence, we can assume  that for some $\psi_*\in W^{3,p}$ for any finite $p$, $\psi_k\rightarrow\psi_*$ in $W^{2,p}$ for any finite $p$.
Therefore $\omega_{\psi_*}^n=e^{F_*}\omega_0^n$.
Because of uniqueness result of Monge-Amp$\grave{e}$re equations(c.f. \cite{blocki}, Theorem 1.1), we can conclude $\varphi_*$ and $\psi_*$ differ by a constant, hence $\omega_{\varphi_*}=\omega_{\psi_*}\leq C_{29}\omega_0$.
That $\omega_{\varphi_*}=\omega_{\psi_*}\geq\frac{1}{C_{29}}\omega_0$ follows from $F_*$ is bounded from below.
\end{proof}
As a result of this, we now show that $\varphi_*$ is actually a smooth cscK.
\begin{cor}
$\varphi_*$ is a smooth solution to cscK.
\end{cor}
\begin{proof}
We know from the proof of Lemma \ref{l3.14} that $\varphi_*\in W^{3,p}$ for any $p<\infty$, hence we know that $\omega_{\varphi_*}\in C^{\alpha}$ for any $0<\alpha<1$. From (\ref{3.32}) and Schauder estimate, we conclude $F_*\in C^{2,\alpha}$ for any $0<\alpha<1$. Then the higher regularity follows from bootstrap.
\end{proof}

\section{Geodesic stability and existence of cscK}
In this section, we will first propose a definition of geodesic stability when there exist nontrivial holomorphic vector fields, then we show that geodesic stability is equivalent to the existence of cscK.
This is a generalization of our previous results in \cite{cc2}. In that work, we restricted to the case when $G$ is trivial and geodesic stability is defined as that $\yen$ invariant(defined below) is positive for any geodesic ray. The main result we will prove in this section is:
\begin{thm}\label{t5.1}
The following statements are equivalent:
\begin{enumerate}
\item The K\"ahler class $[\omega_0]$ admits a cscK metric.
\item There exists $\phi_0\in\mathcal{E}_0^1$ with $K(\phi_0)<\infty$, such that $(M,[\omega_0])$ is geodesic stable at $\phi_0$.
\item $(M,[\omega_0])$ is geodesic stable.
\end{enumerate}
\end{thm}
Here geodesically stability is defined as in Definition \ref{d5.3}. 
Observe that the implication $(3)\Rightarrow(2)$ is trivial.
Therefore we will focus on the implications $(2)\Rightarrow(1)$ and $(1)\Rightarrow(3)$.
First we show the implication $(2)\Rightarrow(1)$.
As a preliminary step, we observe that $(2)$ implies that $K$-energy is invariant under $G$.
\begin{lem}\label{l4.1new}
If $(M,[\omega_0])$ is geodesic semistable at $\phi_0$, in particular, if (2) of Theorem \ref{t5.1} holds, then the  $K$-energy is invariant under $G$.
\end{lem}
\begin{proof}
Let $\sigma\in G$, and let $\varphi\in\mathcal{H}_0$, we need to check $K(\varphi)=K(\sigma.\varphi)$.
Here $\sigma.\varphi$ is defined as in the beginning of section 3.
We will prove the desired result by showing that the Calabi-Futaki invariant must vanish.
To see why this implies our result, let $X$ be a holomorphic vector field and $\{\sigma(t)\}_{t\in\mathbb{R}}$ be the one-parameter family of holomorphic transformation generated by $Re(X)$, such that $\sigma$ lies inside the one-parameter subgroup $\{\sigma(t)\}_{t\in\mathbb{R}}$.
Define $\varphi_t:=\sigma(t).\varphi\in\mathcal{H}_0$. Then for any $t\in\mathbb{R}$ we have
\begin{equation}\label{4.1}
\frac{d K(\varphi_t)}{dt}=\int_M\partial_t\varphi(\underline{R}-R_{\varphi})dvol_{\varphi}=-\int_MRe(X)(\xi)dvol_{\varphi}=-Re\big(\mathcal{F}(X,[\omega_0])\big).
\end{equation}
In the above, $\xi$ is a function chosen so that $\Delta_{\varphi}\xi=R_{\varphi}-\underline{R}$.
$\mathcal{F}(X,[\omega_0])$ is the Calabi-Futaki invariant which depends only on $X$ and K\"ahler class $[\omega_0]$.
So the right hand side of (\ref{4.1}) is a constant. Our result immediately follows as long as we can show the following claim:
\begin{claim}
$$\frac{d}{dt}(K(\varphi_t))=0.$$
\end{claim}
To see the claim, we can assume that $\frac{d}{dt}(K(\varphi_t)):=a<0$, and consider the holomorphic ray $\{\varphi_t\}_{t\in[0,\infty)}$. If instead we have $a>0$, we can consider the holomorphic ray $\{\varphi_t\}_{t\in(-\infty,0]}$, and the same argument below applies.

First we show that $d_1(\varphi,\varphi_t)\rightarrow\infty$ as $t\rightarrow\infty$.
Indeed, we know that $K(\varphi_t)=K(\varphi)+at\leq K(\varphi)$.
 If there exists a sequence of $t_k\rightarrow\infty$, such that $\sup_kd_1(\varphi,\varphi_{t_k})<\infty$, then we may apply \cite{BBEGZ},
 Theorem 2.17, or \cite{Darvas1602}, Corollary 4.8 to conclude that there exists a subsequence $t_{k_l}$, and $\varphi_0\in\mathcal{E}^1$, such that $d_1(\varphi_{t_{k_l}}
,\varphi_0)\rightarrow0$.
But then from the lower semicontinuity of $K$-energy, we know that $K(\varphi_0)\leq\lim\inf_{l\rightarrow\infty}K(\varphi_{t_{k_l}})=-\infty$.
This is a contradiction.

Besides, we also have $d_1(\varphi,\varphi_t)\leq C t$ for some $C>0$. Indeed, if denote $\theta=\partial_t\varphi|_{t=0}$, then $\partial_t\varphi(t)=\theta(\sigma(t))$.
To see this, fix $t_0>0$, we can compute
\begin{equation*}
\begin{split}
\frac{d}{dt}&\big(\sigma(t)^*\omega_{\varphi}\big)|_{t=t_0}=\sqrt{-1}\partial\bar{\partial}\big(\partial_t\varphi|_{t=t_0}\big)=\frac{d}{dt}\sigma(t_0)^*\big(\sigma(t)^*\omega_{\varphi}\big)|_{t=0}=\sigma(t_0)^*(\sqrt{-1}\partial\bar{\partial}\theta)\\
&=\sqrt{-1}\partial\bar{\partial}\big(\theta\circ\sigma(t_0)\big).
\end{split}
\end{equation*}
Hence $\partial_t\varphi|_{t=t_0}=\theta\circ\sigma(t_0)+h(t_0)$, for some function $h$, with $h(0)=0$.
Then from the normalization $I(\varphi_t)=0$, we get
\begin{equation*}
0=\frac{d}{dt}I(\varphi_t)=\int_M\partial_t\varphi\frac{\omega_{\varphi_t}^n}{n!}=\int_M\big(\theta\circ\sigma(t)+h(t)\big)\sigma(t)^*\bigg(\frac{\omega_{\varphi}^n}{n!}\bigg)=\int_M\theta\frac{\omega_{\varphi}^n}{n!}+h(t)vol(M).
\end{equation*}
Since $h(0)=0$, we have $\int_M\theta\frac{\omega_{\varphi}^n}{n!}=0$, which implies $h(t)=0$ for all $t$. But then
$$
d_1(\varphi,\varphi_{\tau})\leq\int_0^{\tau}\int_M|\partial_t\varphi(t)|\frac{\omega_{\varphi_t}^n}{n!}dt=\tau\int_M|\theta|\frac{\omega_{\varphi}^n}{n!}.
$$

Let $t_k\nearrow\infty$ and let $\rho_k(s):[0,d_1(\phi_0,\varphi_{t_k})]\rightarrow \mathcal{E}_0^1$ be the unit speed finite energy geodesic connecting $\phi_0$ and $\varphi_{t_k}$.
Using the convexity of $K$-energy along $\rho_k$ (c.f. \cite{Ber14-01}), we know that for any $s\in[0,d_1(\phi_0,\varphi_{t_k})]$, 
\begin{equation}\label{4.2new}
\begin{split}
K(\rho_k(s))&\leq \big(1-\frac{s}{d_1(\phi_0,\varphi_{t_k})}\big)K(\phi_0)+\frac{s}{d_1(\phi_0,\varphi_{t_k})}K(\varphi_{t_k})\\
&= \big(1-\frac{s}{d_1(\phi_0,\varphi_{t_k})}\big)K(\phi_0)+\frac{s(K(\varphi)+at_k)}{d_1(\phi_0,\varphi_{t_k})}\\
&\leq \max(K(\phi_0),K(\varphi))+\frac{sat_k}{d_1(\phi_0,\varphi_{t_k})}\\
&\leq \max(K(\phi_0),K(\varphi))+\frac{sat_k}{d_1(\varphi,\varphi_{t_k})+d_1(\phi_0,\varphi)}\\
&\leq \max(K(\phi_0),K(\varphi))+\frac{sat_k}{Ct_k+d_1(\phi_0,\varphi)}
\end{split}
\end{equation}
In the first line of (\ref{4.2new}), we used the convexity of $K$-energy along $\rho_k$. 
From the first to the second line, we used that $K(\varphi_{t_k})=K(\varphi)+t_ka$.
From the third to the forth line, we used triangle inequality for $d_1$ and also $a<0$.
From the forth line to the last line, we used $d_1(\varphi,\varphi_{t_k})\leq Ct_k$.

In particular, for each fixed $s$, the $K$-energy is bounded from above, uniform in $k$. Hence we can use the compactness result \cite{Darvas1602}, Corollary 4.8 to conclude there exists a subsequence $\rho_{k_l}(s)$ which converges under $d_1$ distance.
Then we may apply the same argument as in \cite{cc2}, Lemma 6.3 to conclude there exists a subsequence $k_l$, such that for all $s\geq0$, $\rho_{k_l}(s)$ converges under $d_1$ distance. And the limit, denoted as $\rho_{\infty}(s)$, is a unit speed locally finite energy geodesic ray initiating from $\phi_0$.
Using the lower semicontinuity of $K$-energy, we obtain from (\ref{4.2new}):
\begin{equation*}
K(\rho_{\infty}(s))\leq\lim\inf_lK(\rho_{k_l}(s))\leq \max(K(\varphi),K(\phi_0))+\frac{sa}{C}.
\end{equation*}
Hence we get 
\begin{equation*}
\yen[\rho_{\infty}]=\lim_{s\rightarrow\infty}\frac{K(\rho(s))}{s}\leq \frac{a}{C}<0.
\end{equation*}
This contradicts the geodesic semistability.
\end{proof}

As a preliminary step, we show that $(2)$ implies $K$-energy is bounded from below.
\begin{prop}\label{p5.4}
Under the assumption of point (2) of Theorem \ref{t5.1}, we have that $K$-energy is bounded from below.
\end{prop}
\begin{proof}
Suppose otherwise, then there exists a sequence of potentials $\tilde{\varphi}_i\in\mathcal{E}_0^1$, such that $K(\tilde{\varphi}_i)\rightarrow-\infty$. We can choose $\sigma_i\in G$, such that for $\varphi_i:=\sigma_i.\tilde{\varphi}_i\in\mathcal{E}_0^1$, we have $d_{1,G}(\phi_0,\tilde{\varphi}_i)\leq d_1(\phi_0,\varphi_i)\leq d_{1,G}(\phi_0,\tilde{\varphi}_i)+1$.
Because we have shown $K$-energy is invariant under $G$, we know $K(\varphi_i)\rightarrow-\infty$ as well. Next we distinguish two cases and we show there is contradiction in both cases.

(1)$\sup_id_1(\phi_0,\varphi_i)<\infty$. We can invoke
\cite{BBEGZ}, Theorem 2.17, or \cite{Darvas1602}, Corollary 4.8 to conclude that there exists a subsequence $\varphi_{i_k}\stackrel{d_1}{\rightarrow}\psi\in\mathcal{E}^1$. Because of lower semicontinuity of $K$-energy (c.f. \cite{Darvas1602}, Theorem 4.7), we see that $K(\psi)\leq\lim\inf_{i_k}K(\varphi_{i_k})=-\infty$. This is not possible.

(2)$\sup_id_1(\phi_0,\varphi_i)=\infty$. Without loss of generality, we can assume $d_1(\phi_0,\varphi_i)\rightarrow\infty$.
Let $\rho_i:[0,d_1(\phi_0,\varphi_i)]\rightarrow\mathcal{E}_0^1$ be unit speed geodesic segment connecting $\phi_0$ with $\varphi_i$. Since $K$-energy is convex along $\rho_i$ (c.f. \cite{Darvas1602}, Theorem 4.7), we conclude that for any $t\in[0,d_1(\phi_0,\varphi_i)]$,
\begin{equation}\label{5.1}
K(\rho_i(t))\leq(1-\frac{t}{d_1(\phi_0,\varphi_i)})K(\phi_0)+\frac{t}{d_1(\phi_0,\varphi_i)}K(\varphi_i)\leq\max(K(\phi_0),K(\varphi_i)).
\end{equation}
Hence for each fixed $t>0$, we may apply \cite{Darvas1602}, Corollary 4.8 to conclude there exists a subsequence, denoted as $i_k$, such that $\rho_{i_k}(t)$ converges under $d_1$.
Repeat the argument of Lemma 5.3 of \cite{cc2}, one can actually conclude it is possible to take a subsequence $i_k$, such that $\rho_{i_k}(t)$ converges for all $t\in\mathbb{R}$, and the limit $\rho_{\infty}(t)$ is a unit speed locally finite energy geodesic ray(first use Cantor's process to get a subsequence which converges for all $t\in\mathbb{Q}$, then use geodesic property to extend to $t\in\mathbb{R}$).
Also because of lower semicontinuity of $K$-energy and (\ref{5.1}), we actully have $K$-energy is uniformly bounded from above 
on $\rho_{\infty}$. Due to convexity, the alternative (1) in Definition \ref{d5.3} cannot hold for $\rho_{\infty}$.
Hence $\rho_{\infty}$ must be in the second alternative, which means, $\rho_{\infty}$ is parallel to a geodesic ray $\rho'$, which is generated from a holomorphic vector field. This implies $\rho_{\infty}$ is $d_{1,G}$ bounded. 
Indeed, for any $t>0$,
\begin{equation*}
\begin{split}
d_{1,G}(\rho_{\infty}(0),\rho_{\infty}(t))&\leq d_{1,G}(\rho_{\infty}(0),\rho'(0))+d_{1,G}(\rho'(0),\rho'(t))+d_{1,G}(\rho'(t),\rho_{\infty}(t))\\
&\leq d_1(\rho_{\infty}(0),\rho'(0))+\sup_{t>0}d_1(\rho'(t),\rho_{\infty}(t)).
\end{split}
\end{equation*}
In the above, note that $d_{1,G}(\rho'(0),\rho'(t))=0$ since $\rho'$ is generated from a one-parameter family of holomorphic automorphism. 
Also we have $\sup_{t>0}d_1(\rho'(t),\rho_{\infty}(t))<\infty$ since $\rho'$ and $\rho_{\infty}$ are parallel.

 On the other hand, due to the following lemma, we know that $d_{1,G}(\rho_i(t),\phi_0)\geq t-1$, for any $t\in[1,d_1(\phi_0,\varphi_i)]$.
Therefore, 
\begin{equation*}
\begin{split}
d_{1,G}(\rho_{\infty}(t),\phi_0)&\geq d_{1,G}(\rho_i(t),\phi_0)-d_{1,G}(\rho_i(t),\rho_{\infty}(t))\geq t-1-d_1(\rho_i(t),\rho_{\infty}(t))\\
&\rightarrow t-1,\textrm{ as $i\rightarrow\infty$.}
\end{split}
\end{equation*}
This contradicts that $\rho_{\infty}$ is $d_{1,G}$ bounded.
\end{proof}
Above proof involves the use of the following lemma:
\begin{lem}
Let $\varphi$, $\psi\in\mathcal{E}_0^1$.
Suppose that for some $\eps>0$, we have $d_1(\varphi,\psi)\leq d_{1,G}(\varphi,\psi)+\eps$. Let $\rho:[0,K]\rightarrow\mathcal{E}_0^1$ be a finite energy geodesic connecting $\varphi$ and $\psi$, then we have $d_{1,G}(\varphi, \rho(t))\geq d_1(\varphi,\rho(t))-\eps$.
\end{lem}
\begin{proof}
Let $\sigma\in G$ be arbitrary, we need to show 
\begin{equation}\label{5.2}
d_1(\varphi,\sigma.\rho(t))\geq d_1(\varphi,\rho(t))-\eps.
\end{equation}
Indeed,
\begin{equation*}
\begin{split}
d_{1,G}&(\varphi,\psi)\leq d_1(\varphi,\sigma.\psi)\leq d_1(\varphi,\sigma.\rho(t))+d_1(\sigma.\rho(t),\sigma.\psi)\\
&=d_1(\varphi,\sigma.\rho(t))+d_1(\rho(t),\psi)=d_1(\varphi,\sigma.\rho(t))+d_1(\varphi,\psi)-d_1(\varphi,\rho(t))\\
&\leq d_1(\varphi,\sigma.\rho(t))+d_{1,G}(\varphi,\psi)+\eps-d_1(\varphi,\rho(t)).
\end{split}
\end{equation*}
In the first equality of the second line, we use that $G$ is $d_1$-isometry.
In the second equality, we use that $\rho(t)$ is a geodesic.
In the last inequality, we use our assumption.
(\ref{5.2}) immediately follows from this calculation.
\end{proof}
With this preparation, we are ready to prove $(2)\Rightarrow(1)$.
\begin{proof}
Consider the continuity path (\ref{3.5}). Since we have shown $K$-energy is bounded from below, we know from Theorem 1.6 of \cite{cc2} to conclude that (\ref{3.5}) can be solved for any $t<1$(This follows from the properness of twisted $K$-energy $tK+(1-t)J_{\omega_0}$).

Let $t_i\nearrow1$, and let $\tilde{\varphi}_i$ be solution to (\ref{3.5}). We distinguish two cases:

(1)$\sup_id_{1,G}(\phi_0,\tilde{\varphi}_i)<\infty$.
Since we have shown $K$-energy is invariant under the action of $G$ in Lemma \ref{l4.1new}, Proposition \ref{l3.17} applies and we are done.

(2)$\sup_id_{1,G}(\phi_0,\tilde{\varphi}_i)=\infty$. We will show contradiction occurs in this case.
Without loss of generality, we may assume $d_{1,G}(\phi_0,\tilde{\varphi}_i)\rightarrow\infty$.
We may find $\sigma_i\in G$, such that for $\varphi_i=\sigma_i.\tilde{\varphi}_i$, we have $d_{1,G}(\phi_0,\tilde{\varphi}_i)\leq d_1(\phi_0,\varphi_i)\leq d_{1,G}(\phi_0,\tilde{\varphi}_i)+1$.
From Lemma \ref{l3.6}, we know that in particular $\sup_iK(\tilde{\varphi}_i)<\infty$.
From $G$-invariance of $K$-energy, we know that $\sup_iK(\varphi_i)<\infty$.
From now on, the argument is very similar to Proposition \ref{p5.4}. Indeed, let $\rho_i$ be the unit speed finite energy geodesic connecting $\phi_0$, $\varphi_i$.
From the convexity of $K$-energy, we see that $K$-energy is uniformly bounded from above on $\rho_i$(independent of $i$).
Hence we may take limit and get a geodesic ray $\rho_{\infty}$ initiating from $\phi_0$, on which the $K$-energy is decreasing.
Hence the first alternative in Definition \ref{d5.3} fails for $\rho_{\infty}$.
On the other hand, the argument of Proposition \ref{p5.4} shows that $\rho_{\infty}$ is $d_{1,G}$ unbounded. Hence the second alternative in Definition \ref{d5.3} fails as well.
Therefore $\rho_{\infty}$ violates geodesic stability at $\phi_0$.
\end{proof}
\begin{rem}\label{r4.5}
In the proof for existence, we observe that one can weaken the second alternative in Definition \ref{d5.3} to only assume this geodesic ray is $d_{1,G}$ bounded.
\end{rem}

Next we will move on to show the implication $(1)\Rightarrow(3)$.

\begin{proof}
(of $(1)\Rightarrow(3)$)

Without loss generality, we may assume $\omega_0$ itself is cscK.
By the main result of \cite{Darvas1605} and \cite{Dar-Rub-17}, the existence of cscK metric implies that $K$-energy is $G$-invariant and  $K(\varphi)\geq Cd_{1,G}(0,\varphi)-D$, for some constant $C>0$, $D>0$.

Let $\phi\in\mathcal{E}_0^1$ be such that $K(\phi)<\infty$ and let $\rho:[0,\infty)\rightarrow\mathcal{E}_0^1$ be a geodesic ray initiating from $\phi$. There is no loss of generality to assume it is of unit speed. Namely $d_1(\rho(s),\rho(t))=|s-t|$, for any $s,\,t\geq0$. Again we distinguish two cases:

(1)$K$-energy is unbounded from above on $\rho$.
Since $K$-energy is convex on $\rho$, we see that we are in the first alternative of Definition \ref{d5.3}.

(2)$K$-energy is bounded from above on $\rho$. We need to argue that we are in the second alternative of Definition \ref{d5.3}.
Actually we will show that $\rho$ is parallel to a geodesic ray which initiates from $0$ and consists of minimizers of $K$-energy. From the main result of section 5 of our second paper \cite{cc2} (or alternatively from main results of \cite{Darvas1605}), we know the ray consists of cscK potentials.
Then the uniqueness result of \cite{Ber14-01}(Theorem 1.3) applies and shows they differ from each other by a holomorphic transformation.

Let $t_k>0$ be such that $t_k\rightarrow\infty$.
Let $r_k:[0,d_1(0,\rho(t_k))]\rightarrow\mathcal{E}_0^1$ be the unit speed finite energy geodesic segment connecting $0$ and $\rho(t_k)$. Due to the convexity of $K$-energy along $r_k$ and cscK are minimizers of $K$-energy, we know for $t\in[0,d_1(0,\rho(t_k))]$,
\begin{equation}\label{5.3}
\begin{split}
K(r_k(t))&\leq(1-\frac{t}{d_1(0,\rho(t_k))})K(0)+\frac{t}{d_1(0,\rho(t_k))}K(\rho(t_k))\\
& \leq(1-\frac{t}{d_1(0,\rho(t_k))})\inf_{\mathcal{E}^1}K+\frac{t}{d_1(0,\rho(t_k))}\sup_{t\geq0}K(\rho(t)).
\end{split}
\end{equation}
In particular, this shows that $K$-energy is uniformly bounded from above, independent of $k$ and $t$. 
Hence we may repeat the argument of Lemma 5.3 in \cite{cc2}(in particular the compactness result \cite{Darvas1602}, Corollary 4.8) to conclude that one may take a subsequence, denoted as $k_l$, such that $r_{k_l}(t)\rightarrow r_{\infty}(t)$ for any $t\geq0$, and $r_{\infty}(t)$ is a locally finite energy geodesic ray with unit speed.
Now one can replace $k$ by $k_l$ in (\ref{5.3}) and take the limit $k_l\rightarrow\infty$, we see that
\begin{equation}
K(r_{\infty}(t))\leq\lim\inf_{k_l} K(r_{k_l}(t))\leq \inf_{\mathcal{E}^1}K,\textrm{ for any $t\geq0$.}
\end{equation} 
This again uses lower semicontinuity of $K$-energy with respect to $d_1$-convergence (c.f. \cite{Darvas1602}, Theorem 4.7).
So we get $r_{\infty}$ is a unit speed geodesic ray consisting of minimizers of $K$-energy.
The only matter left is to show $r_{\infty}$ and $\rho$ are parallel. We prove this in the following lemma.
\end{proof}
\begin{lem}\label{l4.6}
Let $\rho:[0,\infty)\rightarrow\mathcal{E}_0^1$ be a locally finite energy geodesic ray with unit speed.
Let $t_k\nearrow\infty$, $\phi\in\mathcal{E}_0^1$, and $r_k:[0,d_1(\phi,\rho(t_k))]\rightarrow\mathcal{E}_0^1$ be the finite energy geodesic connecting $\phi$ and $\rho(t_k)$ with unit speed. Suppose $r_k(t)\rightarrow r_{\infty}(t)$ as $k\rightarrow \infty$ in $d_1$, for any $t\geq0$.
Then $r_{\infty}$ is a locally finite energy geodesic with unit speed parallel to $\rho$.
\end{lem}
\begin{proof}
That $r_{\infty}$ is a unit speed locally finite energy geodesic follows the argument in Lemma 5.3 of \cite{cc2}.
It only remains to show that $r_{\infty}$ and $\rho$ are parallel.

Fix $t>0$, we may take $t_k$ sufficiently large so that $t_k\geq t+d_1(\phi,\rho(0))$. Define $s$ so as to satisfy
$$
\frac{t}{t_k}=\frac{s}{d_1(\phi,\rho(t_k))}.
$$
Observe that 
\begin{equation}\label{5.5}
d_1(\rho(t),r_k(t))\leq d_1(\rho(t),r_k(s))+d_1(r_k(s),r_k(t))=d_1(\rho(t),r_k(s))+|s-t|.
\end{equation}
Now
\begin{equation}\label{5.6}
\begin{split}
|s-t|&=t\frac{|t_k-d_1(\phi,\rho(t_k))|}{d_1(\phi,\rho(t_k))}=t\frac{|d_1(\rho(0),\rho(t_k))-d_1(\phi,\rho(t_k))|}{d_1(\phi,\rho(t_k))}\\
&\leq t\frac{d_1(\rho(0),\phi)}{d_1(\phi,\rho(t_k))}\leq t\frac{d_1(\rho(0),\phi)}{t_k-d_1(\rho(0),\phi)}\leq d_1(\rho(0),\phi).
\end{split}
\end{equation}
Hence it only remains to bound $d_1(\rho(t),r_k(s))$.
For this we consider the reparametrization: for $\tau\in[0,1]$, define $\tilde{\rho}(\tau)=\rho\big((1-\tau)t_k\big)$, $\tilde{r}_k(\tau)=r_k\big((1-\tau)d_1(\phi,\rho(t_k))\big)$.

First we consider the case where one has $\phi$, $\rho(0)\in\mathcal{E}^2$. 
The main result of \cite{calabi-chen} and also the extension in \cite{Mabuchi-completion} shows that $(\mathcal{E}^2,d_2)$ is non-positively curved.
 Hence
$$
d_1(\tilde{\rho}(\tau),\tilde{r}_k(\tau))\leq d_2(\tilde{\rho}(\tau),\tilde{r}_k(\tau))\leq\tau d_2(\tilde{\rho}(1),\tilde{r}_k(1))=\tau d_2(\rho(0),\phi),\textrm{ for any $\tau\in[0,1]$.}
$$
Now we take $\tau=1-\frac{t}{t_k}$ to conclude
\begin{equation}\label{5.7}
d_1\big(\rho(t),r_k\big(tt_k^{-1}d_1(\phi,\rho(t_k))\big)=d_1(\rho(t),r_k(s))\leq \big(1-\frac{t}{t_k}\big)d_2(\rho(0),\phi)\leq d_2(\rho(0),\phi).
\end{equation}
Combining (\ref{5.5}), (\ref{5.6}), (\ref{5.7}), we conclude that $d_1(\rho(t),r_k(t))\leq 2d_2(\phi,\rho(0))$ for all $t_k$ sufficiently large.
We can send $k\rightarrow\infty$ and use that $r_k(t)\rightarrow r_{\infty}(t)$ in $d_1$ to conclude that
$$
d_1(\rho(t),r_{\infty}(t))\leq 2d_2(\phi,\rho(0)).
$$
In the general case where we don't assume that $\rho(0)$ or $\phi\in\mathcal{E}^2$, we need to use Theorem \ref{t6.1} to conclude
\begin{equation}
d_1(\tilde{\rho}(\tau),\tilde{r}_k(\tau))\leq \tau d_1(\tilde{\rho}(1),\tilde{r_k}(1))=\tau d_1(\rho(0),\phi).
\end{equation}
Then the rest of the above argument goes through but we no longer need to use $d_2$ distance.
\end{proof}

Next we will prove Theorem \ref{t1.4new}, as an application of equivalence between geodesic stability and existence of cscK metric. Again observe that the implication $(3)\Rightarrow(2)$ is trivial.
It only remains to show the implications $(2)\Rightarrow(1)$ and $(1)\Rightarrow(3)$.
\begin{proof}
(of Theorem \ref{t1.4new})
First we show $(2)\Rightarrow(1)$. 
If Calabi-Futaki invariant is nonzero, then we know cscK metric cannot exist.

In the other case,
let $\rho:[0,\infty)\rightarrow\mathcal{E}_0^1$ be such a geodesic ray as described in $(2)$, initiating from $\varphi$.
We show that this geodesic ray violates the geodesic stability at $\varphi$. Indeed, since $K$-energy is non-increasing on $\rho$, we have $\yen[\rho]\leq0$.

If $\yen[\rho]<0$, then it violates both alternatives in Definition \ref{d5.3}.

If $\yen[\rho]=0$, then Definition \ref{d5.3} requires $\rho$ to be parallel to a geodesic ray generated from a holomorphic vector field, but we assumed this is not the case.

Next we show $(1)\Rightarrow(3)$. 
If Calabi-Futaki invariant is nonzero, then (3) already holds. Now suppose this invariant is zero and there exists $\varphi\in\mathcal{E}_0^1$, such that all geodesic rays either have $K$-energy unbounded from above, or is parallel to a holomorphic ray.
Observe that Calabi-Futaki invariant being zero means $K$-energy is $G$-invariant.
Also for all geodesic rays $\rho$ initiating from $\varphi$, either $\yen[\rho]>0$(when $K$-energy is unbounded), or $\rho$ is bounded under $d_{1,G}$, when $\rho$ is parallel to a holomorphic ray, following the argument of Proposition \ref{p5.4}.
As observed in Remark \ref{r4.5}, this is sufficient to imply cscK metric exists.
\end{proof}

Finally we prove Theorem \ref{t1.6}.
\begin{proof}
(of Theorem \ref{t1.6})
First we assume that $(M,[\omega_0])$ is geodesic semistable.
Fix $0<t_0<1$, if there is no solution to the twisted equation $t_0(R_{\varphi}-\underline{R})=(1-t_0)(tr_{\varphi}\omega_0-n)$, then we can apply Corollary 6.4 of the second paper \cite{cc2} to conclude there exists a locally finite energy geodesic ray with unit speed $\rho(s):[0,\infty)\rightarrow\mathcal{E}_0^1$, such that $K_{\omega_0,t_0}=t_0K+(1-t_0)J_{\omega_0}$ is non-increasing along $\rho$.
On the other hand, from \cite{CoSz}, Proposition 21, we know that $J_{\omega_0}(\varphi)\geq Cd_1(0,\varphi)-D$, for some constant $C,\,D>0$ and any $\varphi\in\mathcal{H}_0^1$.
This implies 
$$
K_{\omega_0,t_0}(\rho(0))\geq K_{\omega_0,t_0}(\rho(s))\geq t_0K(\rho(s))+(1-t_0)Cs-(1-t_0)D.
$$
This means $\yen[\rho]\leq-\frac{C(1-t_0)}{t_0}<0$, contradicting the geodesic semistability.

Then we assume that the twisted equation can be solved for any $0<t<1$. Since we know the solutions are minimizers of the twisted $K$-energy (c.f. \cite{cc2}, Corollary 4.5), we see that $K_{\omega_0,t_0}$ are bounded from below.
From this we can conclude that for any locally finite energy geodesic ray,
\begin{equation*}
\begin{split}
-C_{t_0}&\leq t_0K(\rho(s))+(1-t_0)J_{\omega_0}
(\rho(s))\leq t_0K(\rho(s))+(1-t_0)C'd_1(0,\rho(s))\\
&\leq t_0K(\rho(s))+(1-t_0)C'(d_1(0,\rho(0))+s).
\end{split}
\end{equation*}
In the second inequality above, we used Lemma 4.4 of our second paper \cite{cc2}. Here $C'$ depends only on the background metric $\omega_0$.
In the last inequality, we use that $\rho(s)$ is of unit speed.

Hence 
\begin{equation*}
\yen[\rho]=\lim_{s\rightarrow\infty}\frac{K(\rho(s))}{s}\geq-\frac{(1-t_0)C'}{t_0}.
\end{equation*}
Since $t_0<1$ is arbitrary, we actually have $\yen[\rho]\geq0$.
\end{proof}

\section{Appendix}
Our goal in the Appendix is to prove the following result, which is used in the proof of Theorem \ref{t5.1}.
\begin{thm}\label{t6.1}
Let $1\leq p<\infty$. Let $\phi_0$, $\phi_0'$, $\phi_1$, $\phi_1'\in \mathcal{E}^p$.
Denote $\{\phi_{0,t}\}_{t\in[0,1]}$, $\{\phi_{1,t}\}_{t\in[0,1]}$ be two finite energy geodesics, such that $\phi_{0,t}$ connects $\phi_0$ and $\phi_0'$, $\phi_{1,t}$ connects $\phi_1$ and $\phi_1'$.
Then we have
$$
d_p(\phi_{0,t},\phi_{1,t})\leq (1-t)d_p(\phi_0,\phi_1)+td_p(\phi_0',\phi_1').
$$
\end{thm}
When $p=2$, this result follows from that $(\mathcal{E}^2,d_2)$ is NPC, proved in \cite{Mabuchi-completion} (see also \cite{calabi-chen}).
For general $p$, we were not able to prove $(\mathcal{E}^p,d_p)$ is NPC in the sense of Alexandrov. Nevertheless, above weaker result still holds.

In the following argument, we will mostly follow the notation in \cite{calabi-chen}.
Let $\varphi(x,s,t)\in C^{\infty}(M\times[0,1]\times[0,1])$ be such that $\varphi(\cdot,s,t)\in\mathcal{H}$. Denote $X=\partial_t\varphi$, $Y=\partial_s\varphi$. Given $U\in C^{\infty}(M\times[0,1]\times[0,1])$, consider the connection first introduced by Mabuchi:
\begin{equation}\label{6.1nn}
\nabla_XU=\partial_tU-\nabla_{\varphi}\partial_t\varphi\cdot_{\varphi}\nabla_{\varphi}U,\,\,\,\nabla_YU=\partial_sU-\nabla_{\varphi}\partial_s\varphi\cdot_{\varphi}\nabla_{\varphi}U.
\end{equation}
The dot product in the above line has the following expression in local coordinates:
$$
\nabla_{\varphi}u\cdot_{\varphi}\nabla_{\varphi}v=\frac{1}{2}g_{\varphi}^{i\bar{j}}\big(u_iv_{\bar{j}}+v_iu_{\bar{j}}\big).
$$

Given $\psi_1$, $\psi_2\in C^{\infty}(M)$, we denote 
$$
(\psi_1,\psi_2)=\int_M\psi_1\psi_2dvol_{\varphi}.
$$
This is the so-called Mabuchi's metric on $\mathcal{H}$.

Given $\varphi_0$, $\varphi_1\in\mathcal{H}$, and $\eps>0$, one can consider the so-called $\eps$-geodesic, introduced in \cite{chen00}:
\begin{equation}\label{6.1}
\begin{split}
&\big(\partial_t^2\varphi-|\nabla_{\varphi}\partial_t\varphi|_{\varphi}^2\big)\det g_{\varphi}=\eps\det 
g_0\textrm{ for $(x,t)\in M\times[0,1]$} \\
&\varphi|_{t=0}=\varphi_0,\,\, \varphi|_{t=1}=\varphi_1.
\end{split}
\end{equation}
It is shown in \cite{chen00} that (\ref{6.1}) can be written as a complex Monge-Amp$\grave{e}$re equation on $M\times[0,1]$ with nondegenerate and smooth right hand side, hence is smooth.

The key to prove Theorem \ref{t6.1} is the following estimate:
\begin{prop}\label{p6.1}
Let $\varphi_i: s\in[0,1]\rightarrow\mathcal{H}$, $i=0,1$ be two smooth curves in $\mathcal{H}$. Let $\chi:\mathbb{R}\rightarrow\mathbb{R}_+$ be smooth and convex. Suppose that for each $s
\in[0,1]$, $[0,1]\ni t\mapsto \varphi_{\eps}(s,t)$ is the $\eps$-geodesic connecting $\varphi_0(s)$ and $\varphi_1(s)$. Denote $X=\partial_t\varphi$, $Y=\partial_s\varphi$, then we have
$$
\partial_t^2\int_M\chi(\partial_s\varphi)dvol_{\varphi}\geq\int_M\chi''(\partial_s\varphi)(\nabla_XY)^2dvol_{\varphi}.
$$
\end{prop}
We will postpone the proof of this proposition later, and we will show next how to use this proposition to deduce Theorem \ref{t6.1}.

First we apply Proposition \ref{p6.1} to obtain
\begin{lem}
Let $\phi_0$, $\phi_0'$, $\phi_1$, $\phi_1'\in\mathcal{H}$.
Let $c_1(s):[0,1]\rightarrow\mathcal{H}$ be a smooth curve connecting $\phi_0$ and $\phi_1$, $c_2(s):[0,1]\rightarrow\mathcal{H}$ be a smooth curve connecting $\phi_0'$ and $\phi_1'$.
Let $\{\varphi^{\eps}(s,t)\}_{(s,t)\in[0,1]^2}$ be such that for each fixed $s$, $[0,1]\ni t\mapsto \varphi^{\eps}(s,t)$ is the $\eps$-geodesic connecting $c_1(s)$ with $c_2(s)$. Denote $L_p^{\eps}(t)$ be the length of the curve  $[0,1]\ni s\mapsto\varphi^{\eps}(s,t)\in\mathcal{H}$ under the distance $d_p$, then $t\mapsto L_p^{\eps}(t)$ is convex.
\end{lem}
\begin{proof}
In the following, we will write $L_p^{\eps}(t)$ simply as $L_p(t)$. 
By definition, we have
\begin{equation}
L_p(t)=\int_0^1\bigg(\int_M|\partial_s\varphi|^pdvol_{\varphi}\bigg)^{\frac{1}{p}}ds.
\end{equation}
Denote $\chi_{\delta}(x)=(x^2+\delta^2)^{\frac{p}{2}}$ and put $L_{p,\delta}(t)=\int_0^1\big(\int_M\chi_{\delta}(\partial_s\varphi)dvol_{\varphi}\big)^{\frac{1}{p}}ds=\int_0^1|Y|_{\chi_{\delta}}^{\frac{1}{p}}ds$.
Here for simplicity, we use the notation: $|Y|_{\chi_{\delta}}=\int_M\chi_{\delta}(\partial_s\varphi)dvol_{\varphi}$. Then we have
$$
\frac{d^2}{dt^2}L_{p,\delta}(t)=\int_0^1\partial_t^2(|Y|_{\chi_{\delta}}^{\frac{1}{p}})ds.
$$
We claim that $\partial_t^2(|Y|_{\chi_{\delta}}^{\frac{1}{p}})\geq0$. If this were true, then we know $t\mapsto L_{p,\delta}(t)$ is convex. Also we know that $L_{p,\delta}(t)\rightarrow L_p(t)$ for each $t\in[0,1]$ as $\delta\rightarrow 0$.
This will imply the desired result. 
Hence it only remains to verify the claim.
We can compute
\begin{equation}\label{19}
\begin{split}
&\partial_t^2\big(|Y|_{\chi_{\delta}}^{\frac{1}{p}}\big)=\partial_t\big(\frac{1}{p}|Y|_{\chi_{\delta}}^{\frac{1}{p}-1}\partial_t(|Y|_{\chi_{\delta}})\big)\\
&=\frac{1}{p}|Y|_{\chi_{\delta}}^{\frac{1}{p}-1}\partial_t^2(|Y|_{\chi_{\delta}})-\frac{1}{p}\big(1-\frac{1}{p}\big)|Y|_{\chi_{\delta}}^{\frac{1}{p}-2}|\partial_t(|Y|_{\chi_{\delta}})|^2\\
&\geq\frac{1}{p}|Y|_{\chi_{\delta}}^{\frac{1}{p}-1}\bigg(\int_M\chi_{\delta}''(\partial_s\varphi)(\nabla_XY)^2dvol_{\varphi}-\big(1-\frac{1}{p}\big)|Y|_{\chi_{\delta}}^{-1}|\partial_t(|Y|_{\chi_{\delta}})|^2\bigg).
\end{split}
\end{equation}
In the last inequality, we used Proposition \ref{p6.1}.

On the other hand
\begin{equation}
\begin{split}
\partial_t(&|Y|_{\chi_{\delta}})=\int_M\big(\chi_{\delta}'(\partial_s\varphi)\partial_{st}\varphi+\chi_{\delta}(\partial_s\varphi)\Delta_{\varphi}(\partial_t\varphi)\big)dvol_{\varphi}\\
&=\int_M\chi'_{\delta}(\partial_s\varphi)\big(\partial_{st}\varphi-\nabla_{\varphi}\partial_s\varphi\cdot_{\varphi}\nabla_{\varphi}\partial_t\varphi\big)dvol_{\varphi}=\int_M\chi_{\delta}'(\partial_s\varphi)\nabla_XYdvol_{\varphi}.
\end{split}
\end{equation}
Hence we may apply Cauchy-Schwarz inequality to get
\begin{equation}\label{21}
|\partial_t(|Y|_{\chi_{\delta}})|^2\leq\int_M\frac{(\chi_{\delta}'(\partial_s\varphi))^2}{\chi_{\delta}''(\partial_s\varphi)}dvol_{\varphi}\times\int_M\chi_{\delta}''(\partial_s\varphi)(\nabla_XY)^2dvol_{\varphi}.
\end{equation}
It is straightforward to calculate
$$
\chi_{\delta}'(x)=p(x^2+\delta^2)^{\frac{p}{2}-1}x.
$$
$$
\chi_{\delta}''(x)=p(p-2)(x^2+\delta^2)^{\frac{p}{2}-2}x^2+p(x^2+\delta^2)^{\frac{p}{2}-1}.
$$
Therefore
\begin{equation*}
\begin{split}
\chi''_{\delta}&\chi_{\delta}=p(p-2)(x^2+\delta^2)^{p-2}x^2+p(x^2+\delta^2)^{p-1}\\
&\geq p(p-1)(x^2+\delta^2)^{p-2}x^2=\frac{p-1}{p}(\chi_{\delta}')^2.
\end{split}
\end{equation*}
Hence we obtain from (\ref{21}) that
\begin{equation}\label{22}
\begin{split}
\big(1-\frac{1}{p}\big)|\partial_t(|Y|_{\chi_{\delta}})|^2&\leq\int_M\frac{p-1}{p}\frac{(\chi_{\delta}'(\partial_s\varphi))^2}{\chi_{\delta}''(\partial_s\varphi)}dvol_{\varphi}\times \int_M\chi_{\delta}''(\partial_s\varphi)(\nabla_XY)^2dvol_{\varphi}\\
&\leq\int_M\chi_{\delta}(\partial_s\varphi)dvol_{\varphi}\times\int_M\chi_{\delta}''(\partial_s\varphi)(\nabla_XY)^2dvol_{\varphi}.
\end{split}
\end{equation}
Combining (\ref{19}) and (\ref{22}), the result follows.
\end{proof}
As a consequence, we have
\begin{cor}
Let $\phi_0$, $\phi_0'$, $\phi_1$, $\phi_1'\in\mathcal{H}$.
Let $\{\rho_0(t)\}_{t\in[0,1]}$ be the $C^{1,1}$ geodesic connecting $\phi_0$ and $\phi_0'$, and $\{\rho_1(t)\}_{t\in[0,1]}$ be the $C^{1,1}$ geodesic connecting $\phi_1$ and $\phi_1'$. Then we have
$$
d_p(\rho_0(t),\rho_1(t))\leq (1-t)d_p(\phi_0,\phi_1)+td_p(\phi_0',\phi_1'),\textrm{ for any $t\in[0,1]$.}
$$
\end{cor}
\begin{proof}
Let $\eps>0$. Let $c_1^{\eps}(s):[0,1]\rightarrow\mathcal{H}$ be the $\eps$-geodesic connecting $\phi_0$ and $\phi_1$, $c_2^{\eps}(s):[0,1]\rightarrow\mathcal{H}$ be the $\eps$-geodesic connecting $\phi_0'$ and $\phi_1'$.
Then define $\{\varphi^{\eps}(s,t)\}_{(s,t)\in[0,1]^2}$ be such that for each fixed $s$, $t\mapsto \varphi^{\eps}(s,t)$ is the $\eps$-geodesic connecting $c_1^{\eps}(s)$, $c_2^{\eps}(s)$.

We can apply the previous lemma to conclude that
\begin{equation}\label{29}
d_p(\varphi^{\eps}(0,t),\varphi^{\eps}(1,t))\leq L_p^{\eps}(t)\leq (1-t)L_p^{\eps}(0)+tL_p^{\eps}(1).
\end{equation}
Then we let $\eps\rightarrow0$. Since $t\mapsto\varphi^{\eps}(0,t)$ is the $\eps$-geodesic connecting $\phi_0$, $\phi_0'$, we have $\varphi^{\eps}(0,t)\rightarrow \rho_0(t)$ uniformly (c.f. \cite{chen00}, Lemma 7, point 3), hence in $d_p$ distance, for each fixed $t$, as $\eps\rightarrow0$.
Similarly, $\varphi^{\eps}(1,t)\rightarrow\rho_1(t)$ in $d_p$.
Therefore,
$$
d_p(\varphi^{\eps}(0,t),\varphi^{\eps}(1,t))\rightarrow d_p(\rho_0(t),\rho_1(t)),\textrm{ as $\eps\rightarrow0$.}
$$
While $L_p^{\eps}(0)$ is the length of $c_1^{\eps}$, hence $L_p^{\eps}(0)\rightarrow d_p(\phi_0,\phi_1)$ as $\eps\rightarrow0$. Similarly $L_p^{\eps}(1)\rightarrow d_p(\phi_0',\phi_1')$.
\end{proof}
Now we are ready to prove Theorem \ref{t6.1}, via an approximating argument. 
\begin{proof}
(of Theorem \ref{t6.1})
We choose smooth approximations of $\phi_0$, $\phi_0'$, $\phi_1$, $\phi_1'$. Namely we choose $\phi_{0,k}\rightarrow\phi_0$, $\phi_{0,k}'\rightarrow \phi_0'$, $\phi_{1,k}\rightarrow\phi_1$, $\phi_{1,k}'\rightarrow\phi_1'$ as $k\rightarrow\infty$ under distance $d_p$.
Then from previous corollary, we know
\begin{equation}\label{30}
d_p(\phi_{0,k}(t),\phi_{1,k}(t))\leq (1-t)d_p(\phi_{0,k},\phi_{1,k})+td_p(\phi_{0,k}',\phi_{1,k}'),\textrm{ for any $t\in[0,1]$.}
\end{equation}
In the above, $\{\phi_{0,k}(t)\}_{t\in[0,1]}$ is the $C^{1,1}$ geodesic connecting $\phi_{0,k}$, $\phi_{0,k}'$ and $\{\phi_{1,k}(t)\}_{t\in[0,1]}$ is the $C^{1,1}$ geodesic connecting $\phi_{1,k}$, $\phi_{1,k}'$.

From the end point stability of finite energy geodesic segment (c.f. \cite{Darvas1602}, Proposition 4.3), we know that $\phi_{0,k}(t)\rightarrow \phi_{0,t}$ in $d_p$ as $k\rightarrow\infty$, and $\phi_{1,k}(t)\rightarrow\phi_{1,t}$ as $k\rightarrow\infty$.
Taking limit as $k\rightarrow\infty$ in (\ref{30}), the result follows. 
\end{proof}

It only remains to show Proposition \ref{p6.1}.
\begin{proof}
For simiplicity, we denote $|Y|_{\chi}=\int_M\chi(\partial_s\varphi)dvol_{\varphi}$. Then we may calculate:
\begin{equation}
\begin{split}
\partial_t(&|Y|_{\chi})=\int_M\big(\chi'(\partial_s\varphi)\partial_{st}\varphi+\chi(\partial_s\varphi)\Delta_{\varphi}(\partial_t\varphi)\big)dvol_{\varphi}\\
&=\int_M\chi'(\partial_s\varphi)\big(\partial_{st}\varphi-\nabla_{\varphi}\partial_s\varphi\cdot_{\varphi}\nabla_{\varphi}\partial_t\varphi\big)dvol_{\varphi}=(\chi'(\partial_s\varphi),\nabla_YX).
\end{split}
\end{equation}
Differentiate in $t$ once more, we have 
\begin{equation}
\begin{split}
&\partial_t^2(|Y|_{\chi})=\int_M\chi''(\partial_s\varphi)\partial_{st}\varphi\nabla_YXdvol_{\varphi}+\int_M\chi'(\partial_s\varphi)\partial_t(\nabla_YX)dvol_{\varphi}\\
&+\int_M\chi'(\partial_s\varphi)\nabla_YX\Delta_{\varphi}(\partial_t\varphi )dvol_{\varphi}\\
&=\int_M\chi''(\partial_s\varphi)(\nabla_YX)^2dvol_{\varphi}+(\chi'(\partial_s\varphi),\nabla_X\nabla_YX)\\
&=\int_M\chi''(\partial_s\varphi)(\nabla_YX)^2dvol_{\varphi}+(\chi'(\partial_s\varphi),\nabla_Y\nabla_XX)\\
&\qquad\qquad+(\chi'(\partial_s\varphi),\nabla_X\nabla_YX-\nabla_Y\nabla_XX).
\end{split}
\end{equation}
Since $t\mapsto \varphi^{\eps}(s,t)$ is an $\eps$-geodesic, we have $\nabla_XX=\eps H$, where $H=\frac{\det g_0}{\det g_{\varphi}}$.
Hence
\begin{equation}
\begin{split}
(&\chi'(\partial_s\varphi),\nabla_Y\nabla_XX)=\int_M\chi'(\partial_s\varphi)\eps\nabla_YHdvol_{\varphi}\\
&=\int_M\eps \chi'(\partial_s\varphi)(\partial_sH-\nabla_{\varphi}\partial_s\varphi\cdot_{\varphi}\nabla_{\varphi}H)dvol_{\varphi}\\
&=\int_M\eps\chi'(\partial_s\varphi)(-H\Delta_{\varphi}(\partial_s\varphi)-\nabla_{\varphi}\partial_s\varphi\cdot_{\varphi}\nabla_{\varphi}H)dvol_{\varphi}\\
&=\int_M\eps\chi''(\partial_s\varphi) H|\nabla_{\varphi}\partial_s\varphi|_{\varphi}^2dvol_{\varphi}\geq0.
\end{split}
\end{equation}
From third line to the last line above, we integrated by parts.
Hence it only remains to handle the term $(\chi'(\partial_s\varphi),\nabla_X\nabla_YX-\nabla_Y\nabla_XX)$. The following lemma shows this term is $\geq0$, so we are done.
\end{proof}
\begin{lem}
\begin{equation}\label{6.12}
\begin{split}
(&\chi'(\partial_s\varphi),\nabla_Y\nabla_XX-\nabla_X\nabla_YX)=\int_M\frac{1}{4}\chi''(\partial_s\varphi)g_{\varphi}^{i\bar{j}}\bigg((\partial_t\varphi)_i(\partial_s\varphi)_{\bar{j}}-(\partial_s\varphi)_i(\partial_t\varphi)_{\bar{j}}\bigg)\\
&\times g_{\varphi}^{p\bar{q}}\bigg((\partial_t\varphi)_p(\partial_s\varphi)_{\bar{q}}-(\partial_t\varphi)_{\bar{q}}(\partial_s\varphi)_p\bigg)dvol_{\varphi}=-\int_M\chi''(\partial_s\varphi)\big(\{\partial_t\varphi,\partial_s\varphi\}\big)^2dvol_{\varphi}.
\end{split}
\end{equation}
In the above, $\{\cdot,\cdot\}$ is the Poisson product, defined as 
\begin{equation*}
\{f,g\}_{\varphi}:=Im\big(g_{\varphi}^{i\bar{j}}f_ig_{\bar{j}}\big),\textrm{ $f,\,g\in C^{\infty}(M)$, $\varphi\in\mathcal{H}$.}
\end{equation*}
In particular, if $\chi''\geq0$, the expression in (\ref{6.12}) $\leq0$.
\end{lem}
When $\chi(x)=\frac{1}{2}x^2$, this lemma just expresses the well-known fact that $\mathcal{H}$ has nonpositive sectional curvature under Mabuchi metric.
\begin{proof}
We know that the curvature operator can be represented in terms of Poisson product:
\begin{equation*}
R_{\varphi}(X,Y)Z:=\nabla_X \nabla_Y Z-\nabla_Y \nabla_X Z=\{ \{X,Y\},Z\},\textrm{ $X,\,Y,\,Z\in C^{\infty}(M)$, $\varphi\in\mathcal{H}$.}
\end{equation*}
Therefore,
\begin{equation}
\begin{split}
\big(&\chi'(\partial_s\varphi),\nabla_Y\nabla_XX-\nabla_X\nabla_YX\big)=-\int_M\chi'(\partial_s\varphi)R_{\varphi}(\partial_t\varphi,\partial_s\varphi)\partial_t\varphi dvol_{\varphi}\\
&=-\int_M\chi'(\partial_s\varphi)\{\{\partial_t\varphi,\partial_s\varphi\},\partial_t\varphi\}dvol_{\varphi}\\
&=-\int_MIm\bigg(\chi'(\partial_s\varphi)g_{\varphi}^{i\bar{j}}\{\partial_t\varphi,\partial_s\varphi\}_i(\partial_t\varphi)_{\bar{j}}\bigg)dvol_{\varphi}\\
&=\int_MIm\bigg(\chi'(\partial_s\varphi)g_{\varphi}^{i\bar{j}}\{\partial_t\varphi,\partial_s\varphi\}(\partial_t\varphi)_{i\bar{j}}\bigg)dvol_{\varphi}\\
&+\int_MIm\bigg(\chi''(\partial_s\varphi)g_{\varphi}^{i\bar{j}}(\partial_s\varphi)_i\{\partial_t\varphi,\partial_s\varphi\}(\partial_t\varphi)_{\bar{j}}\bigg)dvol_{\varphi}\\
&=-\int_M\chi''(\partial_s\varphi)\bigg(\{\partial_t\varphi,\partial_s\varphi\}\bigg)^2dvol_{\varphi}.
\end{split}
\end{equation}
From the third line to forth line above, we integrated by parts. Also we noticed that $g_{\varphi}^{i\bar{j}}(\partial_t\varphi)_{i\bar{j}}=\Delta_{\varphi}(\partial_t\varphi)$ is real.
\end{proof}
As an immediate consequence of Theorem \ref{t6.1}, we have
\begin{cor}\label{c5.5}
Let $\rho_i:[0,\infty)\rightarrow\mathcal{E}_0^p$, $i=1,\,2$ be two locally finite energy geodesic rays, then the function $t\mapsto d_p(\rho_1(t),\rho_2(t))$ is convex on $[0,\infty)$.
\end{cor}
As a consequence of this corollary and elementary properties of convex functions on $[0,\infty)$, we can conclude
\begin{cor}\label{c5.6}
Let $\rho_i:[0,\infty)\rightarrow\mathcal{E}_0^p$, $i=1,\,2$ be two locally finite energy geodesic rays. Then exactly one of the two alternative holds:
\begin{enumerate}
\item The limit $\lim_{t\rightarrow\infty}\frac{d_p(\rho_1(t),\rho_2(t))}{t}$ exists and is positive.(may be $+\infty$.);
\item $t\mapsto d_p(\rho_1(t),\rho_2(t))$ is decreasing. In particular, $d_p(\rho_1(t),\rho_2(t))\leq d_p(\rho_1(0),\rho_2(0))$ for any $t>0$.
\end{enumerate}
\end{cor}
The rest of this section is devoted to proving Theorem \ref{t1.2}.
First the uniqueness of such a geodesic ray $\rho_2$ parallel to $\rho_1$ initiating from $\varphi$ follows immediately from Corollary \ref{c5.6}.
The existence part is given by Lemma \ref{l4.6}.
Here we need the assumption $\yen[\rho_1]<\infty$ to show that for each fixed $t$, $K(r_k(t))$ is uniformly bounded from above when $k$ is sufficiently large (by convexity of $K$-energy),
 and then we can use the compactness result of \cite{Darvas1602}, Corollary 4.8 to conclude the convergence of $\{r_k(t)\}_k$ up to a subsequence. 

It only remains to check that $\yen$ invariants are equal for two parallel locally finite energy geodesic rays.
\begin{prop}
Suppose $\rho_i:[0,\infty)\rightarrow\mathcal{E}_0^p$, $i=1,\,2$ are two parallel geodesic rays  with unit speed, 
then we have $\yen[\rho_1]=\yen[\rho_2]$.
\end{prop}
\begin{proof}
It is clear that we just need to show $\yen[\rho_1]\leq\yen[\rho_2]$.
The reverse inequality can be obtained by reversing the role of $\rho_1$ and $\rho_2$.
Also we may assume that $\yen[\rho_2]<\infty$, otherwise there is nothing to prove.

Choose $t_k\nearrow\infty$,
and let \sloppy $r_k:[0,d_1(\rho_1(0),\rho_2(t_k))]\rightarrow\mathcal{E}_0^1$
be the unit speed geodesic segment connecting $\rho_1(0)$ and $\rho_2(t_k)$(with $t=0$ corresponding to $\rho_1(0)$).
Let $t\in[0,d_1(\rho_1(0),\rho_2(t_k))]$, we know from the convexity of $K$-energy:
\begin{equation}
\begin{split}
K(r_k(t))&\leq \big(1-\frac{t}{d_1(\rho_1(0),\rho_2(t_k))}\big)K(\rho_1(0))+\frac{t}{d_1(\rho_1(0),\rho_2(t_k))}K(\rho_2(t_k))\\
&\leq \big(1-\frac{t}{d_1(\rho_1(0),\rho_2(t_k))}\big)K(\rho_1(0))+\frac{t}{t_k-d_1(\rho_1(0),\rho_2(0))}K(\rho_2(t_k)).
\end{split}
\end{equation}
In the second inequality, we used 
$$d_1(\rho_1(0),\rho_2(t_k))\geq d_1(\rho_2(0),\rho_2(t_k))-d_1(\rho_1(0),\rho_2(0))=t_k-d_1(\rho_1(0),\rho_2(0)).$$
Hence
\begin{equation}\label{25}
\frac{K(r_k(t))}{t}\leq \bigg(\frac{1}{t}-\frac{1}{d_1(\rho_1(0),\rho_2(t_k))}\bigg)K(\rho_1(0))+\frac{t_k}{t_k-d_1(\rho_1(0),\rho_2(0))}a.
\end{equation}
Next we make the following claim
\begin{claim}
$r_k(t)\rightarrow \rho_1(t),\textrm{ for fixed $t\geq0$ in $d_1$ distance as $k\rightarrow\infty$}.$
\end{claim}
Assuming this claim for the moment, we can fix $t$, and take limit in (\ref{25}) as $k\rightarrow\infty$, and use lower semicontinuity of $K$-energy to get:
\begin{equation}
\frac{K(\rho_1(t))}{t}\leq \lim\inf_k\frac{K(r_k(t))}{t}\leq \frac{K(\rho_1(0))}{t}+a,\textrm{ for any $t>0$.}
\end{equation}
Then we take limit as $t\rightarrow\infty$, and conclude $\yen[\rho_1]\leq a$.

Now it only remains to show the claim. We define the reparametrization: for $\tau\in[0,1]$, $\tilde{r}_k(\tau)=r_k\big(\tau d_1(\rho_1(0),\rho_2(t_k))\big)$, $\tilde{\rho}_1(\tau)=\rho_1\big(\tau d_1(\rho_1(0),\rho_2(t_k))\big)$.
Then we may use Theorem 0.1 to conclude(here $s_k=d_1(\rho_1(0),\rho_2(t_k))$.)
\begin{equation}
d_1(\tilde{r}_k(\tau),\tilde{\rho_1}(\tau))\leq \tau d_1(\rho_2(t_k),\rho_1(s_k)),\textrm{ for any $\tau\in[0,1]$.}
\end{equation}
Then choose $\tau=\frac{t}{s_k}$, we have 
\begin{equation}\label{28}
\begin{split}
d_1(r_k(t),\rho_1(t))&\leq\frac{t}{s_k}d_1(\rho_2(t_k),\rho_1(s_k))\\
&\leq\frac{t}{s_k}\big(d_1(\rho_2(t_k)
,\rho_1(t_k))+d_1(\rho_1(t_k),\rho_1(s_k))\big)\\
&\leq\frac{t}{s_k}\big(\sup_td_1(\rho_2(t),\rho_1(t))+|t_k-s_k|\big)\\
&\leq\frac{t}{s_k}\big(\sup_td_1(\rho_2(t),\rho_1(t))+d_1(\rho_1(0),\rho_2(0))\big).
\end{split}
\end{equation}

In the second inequality, we used triangle inequality.

In the third inequality, we used that $\rho_1$ is a unit speed geodesic ray.

In the last inequality, we used triangle inequality again to conclude $$|t_k-s_k|=|d_1(\rho_2(0),\rho_2(t_k))-d_1(\rho_1(0),\rho_2(t_k))|\leq d_1(\rho_1(0),\rho_2(0)).$$

Finally we let $k\rightarrow\infty$ in (\ref{28}) to see the claim.
\end{proof}
Next we observe that $\yen$-invariant has the following ``lower semicontinuity" property.
\begin{prop}
Let $\rho_k,\,\rho:[0,\infty)\rightarrow\mathcal{E}_0^p$ be locally finite energy geodesic rays with unit speed. Define $d_k=\lim_{t\rightarrow\infty}\frac{d_p(\rho_k(t),\rho(t))}{t}$(This is well-defined according to Corollary \ref{c5.6}).
Suppose that $d_k\rightarrow0$ and $d_p(\rho_k(0),\rho(0))\rightarrow0$ as $k\rightarrow\infty$, then $\yen[\rho]\leq\lim\inf_{k\rightarrow\infty}\yen[\rho_k]$.
\end{prop}
\begin{proof}
Observe that for any $s>0$, we have $d_p(\rho_k(s),\rho(s))\rightarrow0$.
Indeed, from the convexity property of $t\mapsto d_p(\rho_k(t),\rho(t))$ obtained in Corollary \ref{c5.5}, we know that for any $s'>s>0$, and any $k$
\begin{equation*}
\frac{d_p(\rho_k(s),\rho(s))-d_p(\rho_k(0),\rho(0))}{s}\leq\frac{d_p(\rho_k(s'),\rho(s'))-d_p(\rho_k(0),\rho(0))}{s'}.
\end{equation*}
Let $s'\rightarrow\infty$, we know that 
\begin{equation}
\frac{d_p(\rho_k(s),\rho(s))}{s}\leq d_k+\frac{d_p(\rho_k(0),\rho(0))}{s}\rightarrow0,\textrm{ as $k\rightarrow\infty$ by  assumption.}
\end{equation}
Hence from the lower semicontinuity with respect to $d_p$ convergence, we can conclude that
\begin{equation}\label{5.31n}
\frac{K(\rho(s))}{s}\leq\lim\inf_{k\rightarrow\infty}\frac{K(\rho_k(s))}{s},\textrm{ for any $s>0$.}
\end{equation}
On the other hand, from the convexity of $K$-energy along $\rho_k$, it follows that for any $s''>s>0$,
\begin{equation}
\frac{K(\rho_k(s))}{s}\leq\frac{K(\rho_k(s''))}{s''}+\bigg(\frac{1}{s}-\frac{1}{s''}\bigg)K(\rho_k(0)).
\end{equation}
Let $s''\rightarrow\infty$ in the above and use the definition of $\yen$-invariant, we conclude
\begin{equation}\label{5.33n}
\frac{K(\rho_k(s))}{s}\leq \yen[\rho_k]+\frac{K(\rho_k(0))}{s},\textrm{ for any $s>0$.}
\end{equation}
Finally we let $k\rightarrow\infty$ in (\ref{5.33n}) and combine (\ref{5.31n}), we see
\begin{equation}\label{5.34n}
\frac{K(\rho(s))}{s}\leq\lim\inf_k\frac{K(\rho_k(s))}{s}\leq\lim\inf_k\yen[\rho_k]+\frac{K(\rho(0))}{s},\textrm{ for any $s>0$.}
\end{equation}
Finally we let $s\rightarrow\infty$ in (\ref{5.34n}) to conclude the proof.
\end{proof}

\noindent Xiuxiong Chen\\
University of Science and Technology of China and Stony Brook University\\

\noindent Jingrui Cheng\\
University of Wisconsin at Madison.


\begin{thebibliography}{05}
\bibitem{Ber14-01}
R.J. Berman, B. Berndtsson:
\newblock Convexity of the $K$-energy on the space of K\"ahler metrics.
\newblock J. Amer. math. soc, 30(2017), no 4, 1165-1196.


\bibitem{BBJ15}
R. J. Berman, S. Boucksom, and M. Jonsson.
\newblock A variational approach to the Yau-Tian-Donaldson conjecture.
\newblock  arXiv:1509.04561.

\bibitem{BBEGZ}
R. Berman, S. Boucksom, P. Eyssidieux, V. Guedj, A. Zeriahi.
K\"ahler-Einstein metrics and K\"ahler-Ricci flow on log Fano varieties. Preprint.
\newblock arXiv: 1111.7158.



\bibitem{BHJ16}
S. Boucksom, T. Hisamoto and M. Jonsson.
\newblock Uniform K-stability and asymptotics of energy functionals in K\"ahler geometry.
\newblock   Preprint, arXiv:1603.01026v3.


\bibitem{Darvas1605}
R. Berman, T. Darvas, Chinh H Lu.
\newblock Regularity of weak minimizers of the $K$-enegy and applications to properness and $K$-stability. Preprint.
\newblock arXiv: 1602.03114.

\bibitem{Darvas1602}
R. Berman, T. Darvas, Chinh H Lu.
\newblock Convexity of the extended $K$-energy and the large time behaviour of the weak Calabi flow.
\newblock Geom. Topol. 21(2017), no. 5, 2945-2988.


\bibitem{blocki}
Z. Blocki:
\newblock Uniqueness and stability for the complex Monge-Ampere equation on compact K\"ahler manifolds.
\newblock Indiana Univ. math. J. vol 52, No. 6(2003), 1697-1701.




\bibitem{calabi82}
E. Calabi.
\newblock Extremal K\"ahler metrics.
\newblock Seminar on Differential Geometry, Ann. of Math. Stud., 102(1982), 259-290, Princeton University Press.

\bibitem{calabi85}
E. Calabi.
\newblock Extremal K\"ahler metrics(II).
\newblock Differential Geometry and complex analysis, 1985, 95-114, Springer, Berlin.

\bibitem{calabi-chen}
E. Calabi, X. X. Chen.
\newblock The space of K\"ahler metrics II,
\newblock J. Differential Geom, vol 61(2002), no. 2, 173-193.

\bibitem{CLS10}
Bohui. Chen, An-Min Li and Li Sheng.
\newblock Extremal metrics on toric surfaces.
\newblock Preprint, arXiv:1008.2607

\bibitem{CLS11}
Bohui. Chen, An-Min Li and Li Sheng.
\newblock Uniform K-stability for extremal metrics on toric varieties
\newblock J. Differential Equations. 257 (2014), no. 5, 1487–1500.

\bibitem{CHLS14}
Bohui, Chen, Qing, Han, An-Min Li and Li Sheng.
\newblock Interior estimates for the $n$-dimensional Abreu's equation.
\newblock Adv. Math. 251(2014), 35-46.


\bibitem{chen00}
X. X. Chen.
\newblock The space of K\"ahler metrics.
\newblock J. Differential Geom, vol 56(2000), no. 2, 189-234.

\bibitem{chen-imrn}
X.-X. Chen.
\newblock On the lower bound of the Mabuchi energy and its application.
\newblock International Mathematics Research Notices, vol  12(2000), 607-623.

\bibitem{chen15}
X.-X. Chen:
\newblock  On the existence of constant scalar curvature K\"ahler metric: a new perspective.
\newblock   To appear in Annales math\'{e}matiques de Qu\'{e}bec,	https://link.springer.com/article/10.1007/s40316-017-0086-x,
\newblock  arXiv:1506.06423.



\bibitem{cc1}
X.-X. Chen, J. Cheng.
\newblock On the constant scalar curvature K\"ahler metrics(I): A priori estimates
\newblock preprint.

\bibitem{cc2}
X.-X. Chen, J. Cheng.
\newblock On the constant scalar curvature K\"ahler metrics(II): Existence results
\newblock preprint.

\bibitem{chenhe12}
X.X. Chen and W-Y. He.
\newblock The complex Monge-Amp\'{e}re equation on compact K\"{a}hler manifolds. 
\newblock Math. Ann. 354 (2012), no. 4, 1583--1600.

\bibitem{CPZ}
X.X. Chen, M. Paun and Yu Zeng.
\newblock On deformation of extremal metrics.
\newblock Preprint,  	arXiv:1506.01290.


\bibitem{CS2012}
X. X. Chen and S. Sun
\newblock  Space of K\"ahler metrics (V)-K\"ahler quantization.
\newblock  {\it Metric and differential geometry}, 19-41, Progr. Math., 297, Birkh\"auser/Springer, Basel, 2012. 

\bibitem{CoSz}
T. Collins, G.Sz$\acute{\text{e}}$kelyhidi.
 { Convergence of the J-flow on Toric manifolds}.
 \newblock  J. Differential Geom. 107(2017), no. 1, 47-81.
 


\bibitem{Mabuchi-completion}
T. Darvas.
\newblock The Mabuchi completion of the space of K\"ahler potentials.
\newblock Amer. J. Math, vol 139(2017), no 5, 1275-1313.



\bibitem{Darvas1402}
T. Darvas.
\newblock The Mabuchi Geometry of Finite Energy Classes. 
\newblock Adv. Math. 285(2015), 182-219.



\bibitem{Dar-Rub-17}
T. Darvas, Y. Rubinstein.
\newblock Tian's properness conjecture and Finsler geometry of the space of K\"ahler metrics.
\newblock J. Amer. Math. Soc. vol 30(2017), 347-387.


\bibitem{Dervan171}
R. Dervan,
\newblock Relative K-stability for K\"ahler manifolds,
\newblock  Preprint, arXiv:1611.00569v2.


\bibitem{Dona02new}
S.~K. Donaldson:
\newblock Scalar curvature and stability of toric varieties. 
\newblock J. Differential Geom. 62(2002), no 2, 289-349.

\bibitem{Dona08}
S.~K. Donaldson:
\newblock Extremal metrics on toric surfaces: a continuity method.
\newblock J. Differential Geom. 79(2008), no 3, 384-432.

\bibitem{Dona09}
S.~K. Donaldson:
\newblock Constant scalar curvature metrics on toric surfaces.
\newblock Geom. funct. anal. 19(2009), 83-136.

\bibitem{Hashi}
Y. Hashimoto:
\newblock  Existence of twisted constant scalar curvature K\"ahler metrics with a large twist. Preprint.
\newblock arXiv: 1508.00513.

\bibitem{He-Zeng}
W-Y. He, Y. Zeng.
\newblock Constant scalar curvature equation and the regularity of its weak solutions. Preprint,
\newblock arXiv: 1705.01236.


\bibitem{HIS16}
T. Hisamoto:
\newblock  Stability and coercivity for toric polarizations.
\newblock Preprint, arxiv: 1610.07998v1.

\bibitem{Kolo98}
S. Kolodziej:
\newblock The complex Monge-Amp$\grave{\textrm{e}}$re equations.
\newblock Acta Mathematica, vol 180, no. 1(1998), 69-117.


\bibitem{Ross05}
J. Ross:
\newblock Unstable products of smooth curves.
\newblock Invent. Math. 165(2006), 153-162.

\bibitem{Ross11}
Julius Ross, David Witt Nystrom.
\newblock Analytic test configurations and geodesic rays.
\newblock  Preprint, arXiv:1101.1612

\bibitem{stoppa}
J. Stoppa.
\newblock $K$-stability of constant scalar curvature K\"ahler manifolds.
\newblock Advances in Math. vol 221, (2009), no. 4, 1397-1408.

\bibitem{Sz06}
G. Sz$\acute{\text{e}}$kelyhidi:
\newblock  Extremal metrics and K-stability. 
\newblock arXiv:0611002. Ph.D Thesis. 1

\bibitem{tian87}
G. Tian.
\newblock On K\"ahler-Einstein metrics on certain K\"ahler manifolds with $C_1(M)>0$.
\newblock {Invent. Math.} 89(1987), 225-246.

\bibitem{Yau78}
S.-T. Yau.
\newblock On the {R}icci curvature of a compact {K}\"ahler manifold and the
  complex Monge-Amp$\grave{\text{e}}$re equation, ${I}^*$.
\newblock {\em Comm. Pure Appl. Math.}, 31:339--441, 1978.

\bibitem{zeng}
Y.  Zeng:
\newblock Deformations from a given K\"ahler metric to a twisted cscK metric. Preprint.
\newblock arXiv: 1507.06287.

\bibitem{zz08}
B. Zhou and X.H. Zhou
\newblock Relative $K-$stability and modified $K$-energy on Toric manifolds.
\newblock Advance in Mathematics, 219 (2008), no. 4, 1327-1362.
 
\end{thebibliography}
\end{document}